\documentclass[a4, 11pt]{article}

\usepackage[latin1]{inputenc}
\usepackage[T1]{fontenc}
\usepackage[english]{babel}
\usepackage{layout}

\usepackage[active]{srcltx}
\usepackage{amsmath,amsfonts,amssymb,epsf}
\usepackage{eepic,epic,graphicx,amsmath,amssymb,xypic}
\xyoption{curve}
\usepackage{verbatim}

\usepackage{soul}
\usepackage{bm}
\usepackage{lscape}
\usepackage{dsfont}
\usepackage{lmodern}
\usepackage{multicol}
\usepackage{colortbl}
\usepackage{wrapfig}
\usepackage{graphicx}
\usepackage[svgnames]{xcolor}
\usepackage{stmaryrd}
\usepackage{listings}
\usepackage{hobsub}
\usepackage{hobsub-generic}
\usepackage{hobsub-hyperref}
\usepackage{yhmath}
\usepackage{array} 
\usepackage{color}
\usepackage[top=3cm, bottom=3.0cm, left=2.75cm, right=2.75cm]{geometry} 

\usepackage{amsthm}
\usepackage{amsmath}
\usepackage{amssymb}
\usepackage{stmaryrd} 
\usepackage{mathrsfs}
\usepackage{ifthen,fp}
\usepackage{fancyhdr}
\usepackage{hyperref}
\usepackage{lastpage}
\usepackage{pgf, tikz,tkz-tab, tkz-fct}
\usetikzlibrary{calc}
\usepackage{tkz-euclide}
\usepackage{frcursive}
\usetikzlibrary[patterns]
\usepackage[french]{minitoc}
\usepackage{tcolorbox}

\usepackage{pgfplots}
\pgfplotsset{compat=1.15}
\usepackage{mathrsfs}
\usetikzlibrary{arrows}

\def\II{{\mathds 1}}
\def\R{{\mathbb R}}

\def\P{{\mathbb P}}
\def\T{{\mathbb T}}

\def\CCCC{{\mathscr C}}

\def\KK{{\mathcal K}}
\def\QQ{{\mathcal Q}}

\def\PPP{{\mathbf P}}
\def\III{{\mathcal I}}
\def\FF{{\mathcal F}}

\def\EE{{\mathcal E}}

\def\BB{{\mathcal B}}

\def\NN{{\mathcal N}}
\def\KK{{\mathcal K}}
\def\MM{{\mathcal M}}
\def\TT{{\mathcal T}}
\def\F{{\mathbb F}}
\def\D{{\mathbb D}}

\def\N{{\mathbb N}}
\def\E{{\mathbb E}}

\def\Q{{\mathbb Q}}
\def\P{{\mathbb P}}
\def\KK{{\mathcal K}}

\def\XX{{\mathcal X}}
\def\MM{{\mathcal M}}

\def\Supp{{\rm{Supp \,}}}
\def\Fit{{\rm{Fit}}}

\def\Diam{{\rm{Diam}}}
\def\Hess{{\rm{Hess}}}

\def\Fit{{\rm{Fit}}}

\def\id{{\rm{id}}}

\def\XX{{\mathcal{X}}}

\def\III{{\mathcal{I}}}

\def\LL{{\mathcal{L}}}

\newcommand{\dd}{\mathrm{d}}

\makeatletter
\newcounter{subsubsubsection}[subsubsection]
\renewcommand\thesubsubsubsection{\@roman\c@subsubsubsection}
\newcommand\subsubsubsection{\@startsection{subsubsubsection}{4}{\z@}%
                                     {-3.25ex\@plus -1ex \@minus -.2ex}%
                                     {1.5ex \@plus .2ex}%
                                     {\normalfont\small\bfseries}}
\newcommand*\l@subsubsubsection{\@dottedtocline{3}{5.2em}{1em}}
\newcommand*{\subsubsubsectionmark}[1]{}
\makeatother

\definecolor{jaune clair}{rgb}{0, 0.9,0.5}
\definecolor{gris clair}{gray}{0.75}
\definecolor{Mon_rouge}{rgb}{1.,0.,0.}
\definecolor{Mon_violet}{rgb}{0.4,0.,0.6}
\definecolor{Mon_violet2}{rgb}{0.498039,0.,1.}
\definecolor{Mon_orange}{rgb}{1.,0.4980392156862745,0.}
\definecolor{Mon_vert}{rgb}{0.,0.8,0.} 
\definecolor{Mon_vert2}{rgb}{0.,0.4,0.2}
\definecolor{Mon_vert3}{rgb}{0.2,0.6,0.}
\definecolor{Mon_vert4}{rgb}{0.,0.39215686274509803,0.}
\definecolor{Mon_bleu}{rgb}{0.,0.,0.8}
\definecolor{Mon_bleu_coalescence}{rgb}{0.49,0.49,1}

\definecolor{Mon_bleu_doux}{rgb}{0.49,0.49,1.}
\definecolor{bcduew}{rgb}{0.7372549019607844,0.8313725490196079,0.9019607843137255}
\definecolor{erfvff}{rgb}{0.8823529411764706,0.9607843137254902,1.} % FOND bleu 1 : OK
\definecolor{ebfvff}{rgb}{0.9215686274509803,0.9607843137254902,1.} % FONC bleu 2 - Plus clair - OK
\definecolor{qqzzqq}{rgb}{0.0,0.6,0.0} % JOLI VERT
\definecolor{qqqqff}{rgb}{0.0,0.0,1.0}

\definecolor{qqqqff}{rgb}{0,0,1}
\definecolor{ffqqqq}{rgb}{1,0,0}
\definecolor{xdxdff}{rgb}{0.49019607843137253,0.49019607843137253,1}
\definecolor{ffxfqq}{rgb}{1,0.4980392156862745,0}

\newtheorem{commun}{Commun}[section]

\newtheorem{Thm}[commun]{Theorem}%[section] % 
\newtheorem{Prop}[commun]{Proposition}%[section]
\newtheorem{Lem}[commun]{Lemma}%[section]
\newtheorem{Cor}[commun]{Corollary}%[section]
\newtheorem{Rem}[commun]{Remark}
\definecolor{Mon_rouge}{rgb}{1.,0.,0.}
\definecolor{Mon_violet}{rgb}{0.4,0.,0.6}
\definecolor{Mon_orange}{rgb}{1.,0.4980392156862745,0.}
\definecolor{Mon_vert}{rgb}{0.,0.8,0.}
\definecolor{Mon_bleu}{rgb}{0.,0.,0.8}

\definecolor{bleu clair}{rgb}{0, 0,0.8}
\definecolor{gris clair}{gray}{0.75}

\begin{document}

\pagestyle{fancy}
\renewcommand{\rightmark}{}
\renewcommand{\leftmark}{}
\cfoot{\textbf{\thepage/\pageref{LastPage}}}

\vspace{0.5cm}

\begin{center}
\LARGE \textbf{Convergence of population processes with small and frequent mutations to the canonical equation of adaptive dynamics} \\ \vspace{0.25cm}
\large \textsc{Nicolas Champagnat}$^{\textsc{*}}$, 
\qquad \textsc{Vincent Hass}$^{\textsc{\dag}}$\footnote{Corresponding author at: IUT Nord Franche-Comt\'{e}, D\'{e}partement G\'{e}nie Civil et Construction Durable, 19 Avenue du Mar\'{e}chal Juin, 90 016 Belfort, Cedex, France. \\
 \textit{E-mail addresses:} nicolas.champagnat@inria.fr (Nicolas Champagnat), vincent.hass@univ-fcomte.fr (Vincent Hass).} \\ \vspace{0.5cm} 
\textsc{\today} \\  \vspace{0.5cm}
\textit{{}$^{\textsc{*}}$ Universit\'{e} de Lorraine, CNRS, Inria, IECL, F-54000 Nancy, France} \vspace{0.15cm} \\
\textit{{}$^{\textsc{\dag}}$ Universit\'{e} de Franche-Comt\'{e}, IUT Nord Franche-Comt\'{e}, F-90000 Belfort, France}
\end{center}

\begin{center}
\rule{10cm}{.5pt} %\textwidth
\end{center}

\textbf{Abstract.} In this article, a stochastic individual-based model describing Darwinian evolution of asexual, phenotypic trait-structured population, is studied. We consider a large population with constant population size characterised by a resampling rate modeling competition pressure driving selection and a mutation rate where mutations occur during life. In this model, the population state at fixed time is given as a measure on the space of phenotypes and the evolution of the population is described by a continuous time, measure-valued \textsc{Markov} process. We investigate the asymptotic behaviour of the system, where mutations are frequent, in the double simultaneous limit of large population $(K\to+\infty
)$ and small mutational effects $({\color{black}\sigma_{K}} \to 0)$ proving convergence to an ODE known as the canonical equation of adaptive dynamics. This result holds only for a certain range of {\color{black}$\sigma_{K}$} parameters (as a function of $K$) which must be small enough but not too small either. The canonical equation describes the evolution in time of the dominant trait in the population driven by a fitness gradient. This result is based on an slow-fast asymptotic analysis. We use an averaging method, inspired by \textsc{Kurtz} {\color{blue}\cite{Kurtz}}, which exploits a martingale approach and compactness-uniqueness arguments. The contribution of the fast component, which converges to the centered \textsc{Fleming-Viot} process, is obtained by averaging according to its invariant measure, recently characterised in {\color{blue} \cite{Champagnat_Hass_FVr_2022}}. \vspace{0.2cm}

 \vspace{0.2cm}


\begin{center}
\rule{10cm}{.5pt} 
\end{center}

\normalsize

\section{Introduction}

In this article we study, at the individual level and in the interplay between ecology and Darwinian evolution, a population model, structured by a  $1-$dimensional quantitative phenotypic trait. The Darwinian evolution is based on three basic mechanisms. Firstly, \emph{heredity} which allows the transmission of the individual phenotypic characteristics from one generation to another. Secondly, a source of variation in the individual phenotypic characteristics: in our case it is only \emph{mutations}. Finally, a \emph{selection} mechanism which can result from interaction between individuals in the population such as competition. Our model is an \emph{individual-based model} (in short, IBM) which involves a finite and asexual population with constant population size in which each individual's birth, death and mutation events are described. IBMs were first introduced in ecology as a tool to describe local interactions or complex phenomena at the individual level {\color{blue} \cite{Oborny_1994, Bolker_Pacala_1997, Bolker_Pacala_1999, Doebeli_Dieckmann_2003, Grimm_Railsback_2005, DeAngelis_individual-based_2018}}. 
 Ecological studies using IBMs are mainly numerical and the models are mostly posed in discrete space as systems of interacting particles {\color{blue} \cite{Oborny_1994}} and more rarely in continuous space {\color{blue} \cite{Bolker_Pacala_1997, Bolker_Pacala_1999, Doebeli_Dieckmann_2003}}. 
  Many IBMs (with non-constant population size) have been proposed and {\color{black}studied} in the context of Darwinian {\color{black}evolution} by the biology community {\color{blue} \cite{Metz, Dieckmann_Law_2000, Ferriere_Bronstein_Rinaldi_Law_Gauduchon_2002}}  and the mathematical community {\color{blue} \cite{Champagnat, Champagnat_Ferriere_Meleard_2006, Bans}}.  Others are dispersal models in spatially structured populations where the trait is viewed as a spatial location and mutations as dispersal {\color{blue} \cite{Bolker_Pacala_1997, Bolker_Pacala_1999, Law_Dieckmann_2002, Fournier}}. Other models, structured in age, were developed in {\color{blue} \cite{Charlesworth_1994}} and studied mathematically in {\color{blue}  \cite{Oelschlager_1990, Tran_2008}}, or structured in age and traits in {\color{blue} \cite{Tran_Ferriere_2009, Meleard_Tran_2009, Meleard_Tran}}.  \\ 
 \indent We consider an IBM with fixed population size, so that births and deaths occur simultaneously in so-called \emph{resampling} (or swap phenomenon) events. The mutation and resampling rate of an individual depends on its phenotype. When a mutation occurs, the new mutant trait is close to its parent's one yielding a slow variation of the trait. In population genetics, the \textsc{Wright-Fisher} model (and its extensions with selection, mutation or immigration), \textsc{Cannings} model or the \textsc{Moran} model are interested in the evolution of allele frequencies according to various mechanisms {\color{blue} \cite{Etheridge}}. In {\color{blue} \cite{fleming_no_1979, dawson_wandering_1982, Dawson, Etheridge_2}}, the authors construct the \textsc{Fleming-Viot} process as a scaling limit of large population from the \textsc{Moran} process.  Several extensions exist, including frequency-dependent selection, recombination, other reproduction mechanisms {\color{blue} \cite{Ethier_Kurtz_1993, Gonzalez-Casanova_Smadi_2020}}. In particular, the last article provides a bridge between population genetics and eco-evolutionary models. Others, based on \textsc{Kermack-McKendrick}'s model, are interested in epidemiological questions {\color{blue} \cite{Murray_mathematical_2004, Nepomuceno_Takahashi_Aguirre_2016}}. \\

The aim of this article is to describe the evolutionary dynamics of the dominant trait, at the population level, on a long time scale, as a solution to an ordinary differential equation (in short, ODE) called \emph{Canonical Equation of the Adaptive Dynamics}  (in short, CEAD). Canonical equations are well-known tools in evolutionary biology, used to predict the evolutionary fate of ecological communities. More precisely, such equations describe the evolution of dominant traits in a biological population as driven by mutations and a fitness gradient which describes the strength of selection that pushes the population to locally increase its fitness {\color{blue} \cite{Wright_1931, Lande_1979, Roughgarden_1983,  Dieckmann_Law_1996}}.  
\indent Fitness measures the selective value of a given individual in a given environment including the population under consideration itself. This individual can be any (fictitious) mutant individual that can be born in the population at the time under consideration. The way to construct a fitness landscape depends on the ecological context {\color{blue} \cite{Metz_Nisbet_Geritz_1992}}. For continuous time homogeneous \textsc{Markov} models as studied below, it is the instantaneous growth rate (birth rate minus death rate) of the individual considered in the environment considered. If we further assume that the population constituting the environment is in a stationary state, then we can consider that the fitness of a given mutant individual in this population governs the possibility of invasion of the descendants of this mutant in the population. \\ 
\indent The canonical equation has been studied and derived in different contexts such as game theory {\color{blue} \cite{Hofbauer_Sigmund_1988}} and quantitative genetics  {\color{blue} \cite{Lande_1979}}. In the branch of evolutionary biology called \emph{adaptive dynamics} {\color{blue} \cite{Hofbauer_Sigmund_1990, Nowak_Sigmund_1990, Marrow_Law_Cannings_1992}}, 
the CEAD has been introduced heuristically in {\color{blue} \cite{Dieckmann_Law_1996}}. The theory of adaptive dynamics studies the links between ecology and Darwinian evolution, more precisely, it investigates the effects of the ecological aspects of population dynamics on the evolutionary process, and so describes the population dynamics on the phenotypic level instead of the genotypic level. The theory of adaptive dynamics is based on biological assumptions of \emph{rare} and \emph{small mutations} and of \emph{large population} under which the CEAD was proposed. \\ 
 \indent Two mathematical approaches were developed to give a proper mathematical justification of this theory: a \emph{deterministic} one, and a \emph{stochastic} one.  All these approaches are based on the use of IBMs and different combinations of the three previous biological assumptions  and consider a parameter scaling under which the population distribution over the trait space concentrates to \textsc{Dirac} masses, i.e. to subpopulations in which all the individuals have the same trait. The canonical equation corresponds to the motion of \textsc{Dirac} masses. Let us introduce three parameters corresponding to the different biological assumptions: $p$ for the mutation probability, $\sigma$ for the mutation size and $K$ for the population size. {\color{black} In the sequel, we denote $p_{K}:=p(K)$ and $\sigma_{K} := \sigma(K)$ to emphasise the dependence of $K$.} \\ 
 \indent In the deterministic approach, {\color{blue} \cite{Fournier}} first establishes that, under the asymptotic of large population ($K \to +\infty$), the IBM converges in law to a deterministic process which is a weak solution of a partial differential equation. Then by adding the small mutation assumption ($\sigma \to 0$) and an appropriate time scaling $1/\sigma$,  {\color{blue} \cite{Diekmann_Jabin_Mischler_Perthame, Perthame_Barles_2008, Lorz_Mirrahimi_Perthame_2011}} establish the convergence to a version of the canonical equation different from the one of {\color{blue} \cite{Dieckmann_Law_1996}} and described by a \textsc{Hamilton-Jacobi} equation with constraint. \\ 
 \indent The stochastic approach was developed in {\color{blue} \cite{Champagnat, Champagnat_PTRF_2011}}. In {\color{blue} \cite{Champagnat}} it is proved that the IBM converges, for finite dimensional distributions, under the double simultaneous asymptotic of large population ($K \to +\infty$) and rare mutation (${\color{black}p_{K}} \to 0$) to a stochastic process: the \emph{TSS} for \emph{Trait Substitution Sequence} introduced in {\color{blue} \cite[\color{black} Section 6.4]{Metz}}. This is a pure jump \textsc{Markov} process in the trait space where the population is at all times monomorphic and where the jumps describe the invasion and then the fixation of a mutant $y$ in a monomorphic resident population of trait $x$.  By adding the small mutation assumption ($\sigma \to 0$) to the TSS, its  convergence to the canonical equation proposed in {\color{blue} \cite{Dieckmann_Law_1996}} is  established in {\color{blue} \cite{Champagnat_PTRF_2011}}. The time scale involved in observing the canonical equation phenomenon is $1/K{\color{black}\sigma_{K}^{2}} {\color{black}p_{K}}$. Extensions of these results were obtained in {\color{blue} \cite{Champagnat_Jabin_Meleard_JMPA_2014}} for chemostat models, in {\color{blue} \cite{Ernande_Dieckmann_Heino_2004, Meleard_Tran_2009}} for age-structured populations models, in {\color{blue} \cite{Leman_2016}} for spatial-structured populations models and {\color{blue} \cite{Baar_Bovier_Champagnat_2017}} which studies the simultaneous application of the three limits 
 ($K \to +\infty, {\color{black}p_{K}} \to 0, {\color{black}\sigma_{K}} \to 0$) in order to determine precisely the range of parameters leading to the canonical equation. {\color{black} Note that another approach, studied in {\color{blue}\cite{Coquille_Kraut_Smadi_2021, Esser_Kraut_2021}}, consider models of large population ($K\to + \infty$) scaling with rare-but-not-too-rare mutations (with power-law mutation rates $p_{K} := K^{-\alpha} \to 0$). This specific mutational scale implies that negligible subpopulations of size $K^{\beta}$, $\beta \in (0,1)$ may have a strong contribution to evolution.  A similar scaling was also studied in {\color{blue}\cite{Durrett_Mayberry_2011, Bovier_Coquille_Smadi_2019, Champagnat_Meleard_Tran_2021, Blath_Paul_Tobias_2023}} %. Note that {\color{blue}\cite{Durrett_Mayberry_2011, Bovier_Coquille_Smadi_2019, Champagnat_Meleard_Tran_2021, Coquille_Kraut_Smadi_2021, Esser_Kraut_2021, Blath_Paul_Tobias_2023}} not consider the limit of small mutation steps and thus do not derive a CEAD.} 
 {\color{black} with fixed mutation size. Building on these works,{\color{blue}~\cite{Paul_2024}} recently
   proposed a derivation of a CEAD assuming small mutations.} 
 \\ 
\indent Despite their success, the proposed approaches are criticised by biologists {\color{blue}\cite{Waxman, Perthame_Gauduchon_2010}}. Among the biological assumptions of adaptive dynamics, the assumption of rare mutations is the most critised as unrealistic {\color{blue} \cite{Waxman}}. The rate of molecular mutation is relatively well known and generally involves several nucleotide substitutions per generation. The adaptive dynamics response is based on the fact that only non-synonymous mutations (changing phenotypes) producing viable individuals should be taken into account {\color{blue} \cite{Anfinsen_principles_1973} \cite[\color{black} Section 6.4]{Metz}}. Since only a small fraction of DNA codes for proteins and many mutations produce non-functional proteins, and thus non-viable individuals, it is not unreasonable to assume that mutations are rare, but probably not as rare as assumed in the stochastic approach {\color{blue} \cite{Champagnat, Champagnat_PTRF_2011, Baar_Bovier_Champagnat_2017}}. Another criticism of stochastic approaches is that the phenomenon of the canonical equation takes place on a too long evolutionary time scale. In order to solve these problems, we propose to apply a double simultaneous asymptotic of small mutations (${\color{black}\sigma_{K}} \to 0$) and large population ($K \to +\infty$), but \emph{frequent mutations} ($p\equiv 1$). After conveniently scaling the population state, this leads to a \emph{slow-fast dynamics} where the phenomenon of the canonical equation is visible on a time scale $1/K{\color{black}\sigma_{K}^{2}}$. So, there is some consistency between the previous stochastic approaches developed in {\color{blue} \cite{Champagnat_PTRF_2011, Baar_Bovier_Champagnat_2017}} and ours: it is exactly the same canonical equation. However in our situation, the CEAD is visible on a shorter evolutionary scale than theirs, and therefore is biologically more reasonable. {\color{black} Note however that the restrictions on $\sigma_K$ in~\cite{Baar_Bovier_Champagnat_2017} are in a completely different regime than ours ($\sigma_K\gg K^{-1/2}$, as opposed to $\sigma_K\ll K^{-2}$ here). Hence it is possible to choose $p_K$ in~\cite{Baar_Bovier_Champagnat_2017} in such a way that the CEAD appears on a faster time scale than in our work, where the fastest scale is in $K^3$.}\\

As the mutation size parameter is small, the population distribution tends to be close to a \textsc{Dirac} mass. Our aim is to describe the evolution of the support of this \textsc{Dirac} mass on a large time scale. The mean trait appears as the natural slow component. It will be proved to act on the time scale $1/K{\color{black}\sigma_{K}^{2}}$.  In our case, the fast component acts on the time scale $K$ and is given by the dynamics of the centered and dilated distribution of traits corresponding to a discrete version of the centered \textsc{Fleming-Viot} process {\color{blue} \cite{Champagnat_Hass_FVr_2022}}. Since the centered \textsc{Fleming-Viot} process is ergodic {\color{blue} \cite{Champagnat_Hass_FVr_2022}}, we expect the centered and dilated distribution of traits to stabilise on the slow time scale, hence the distribution of traits should stabilise to a \textsc{Dirac} mass. Therefore, the dynamics of the dominant trait corresponds to that of the mean trait. \\    
\indent The reason for not considering the same IBM as in {\color{blue} \cite{Champagnat, Champagnat_PTRF_2011, Baar_Bovier_Champagnat_2017}} is  because it involves three time scales, a slow one corresponding to the dynamics of the mean trait acting on the evolutionary time scale $1/K{\color{black}\sigma_{K}^{2}}$, a fast one corresponding to the dynamics of the centered, dilated distribution of traits acting on the evolutionary time scale $K$ and a very fast one corresponding to the population size dynamics acting on the ecological time scale $1$. For simplicity, we focus here on a model with constant population size to reduce the number of time scales to two. We expect our results to extend to general IBMs as in {\color{blue} \cite{Champagnat, Champagnat_PTRF_2011, Baar_Bovier_Champagnat_2017}} and we leave this for future works. \\

\indent To prove convergence in the framework of slow-fast dynamics (also called \emph{stochastic singular perturbations}), different techniques can be used. \\ 
\indent Firstly, the method of the \emph{perturbed test function} initially proposed by \textsc{Papanicolaou}, \textsc{Stroock} and \textsc{Varadhan}  in {\color{blue} \cite{Papanicolaou_1977}} identifies the generator of the limit process with a martingale approach whose idea is the following.  If we consider a family of stochastic processes $\left(\left(X^{\varepsilon}, Y^{\varepsilon}\right)\right)_{\varepsilon>0}$ of generator $L^{\varepsilon}$ where $X^{\varepsilon}$ is the \emph{slow component} and $Y^{\varepsilon}$ is the \emph{fast component} in the form $Y^{\varepsilon}(t) = Y\left(t/\varepsilon\right)$ where $Y$ is a \textsc{Markov} process, the slow-fast problem consists in identifying the limit process of $X^{\varepsilon}$ using ergodicity properties  of the fast dynamics. Assuming that the family $\left(X^{\varepsilon}\right)_{\varepsilon >0}$ is tight, we consider $X$ a limiting value. We can expect to characterise $X$ with a martingale problem derived from the martingale problem of $\left(X^{\varepsilon}, Y^{\varepsilon} \right)$ provided that the solution to this problem is unique.  We would then obtain the convergence of $X^{\varepsilon}$ to $X$ in law. However, for multiscale singular problems, the convergence of $L^{\varepsilon}\varphi\left(\left(X_{t}^{\varepsilon}, Y_{t}^{\varepsilon}\right)\right)$ when $\varepsilon \to 0$, for a test function $\varphi$ depending only on the $X$ component, does not (in general) take place because it contains diverging terms in $\varepsilon$. To overcome this difficulty, firstly the idea is to decompose the generator $L^{\varepsilon}$ in the following form $L^{\varepsilon} = \frac{1}{\varepsilon}L_{1} + L_{0}$ where $L_{1}$ is the infinitesimal generator of the \textsc{Markov} process $Y$ in the variable $y$ and $L_{0}$ the operator of the slow component depending on slow and fast variables. Secondly, the idea is to perturb the initial test function $\varphi(x)$ into a test function $\varphi^{\varepsilon}(x,y) := \varphi(x) + \varepsilon \varphi_{1}(x,y)$ such that 
\begin{equation}
\begin{aligned}
L^{\varepsilon}\varphi^{\varepsilon}  = \left[\frac{1}{\varepsilon}L_{1} + L_{0} \right]\left(\varphi + \varepsilon \varphi_{1} \right) &=   \left(L_{1}\varphi_{1} + L_{0}\varphi - \overline{L}_{0}\varphi \right) + \overline{L}_{0}\varphi + \varepsilon L_{0}\varphi_{1}
\end{aligned}
\label{Eq_L_epsilon_Intro}
\end{equation}
because of $L_{1}\varphi = 0$ and where $\overline{L}_{0}$ is the operator of the limit slow component \emph{averaged} by the unique invariant probability measure of the limit fast component $Y$. Provided that $\varphi_{1}$ solves the \textsc{Poisson} equation (with respect to $L_{1}$ and the variable $y$) 
\begin{equation}
L_{1}\varphi_{1} + L_{0}\varphi - \overline{L}_{0}\varphi = 0,
\label{Eq_Phi_1_Intro}
\end{equation}
we deduce that $L^{\varepsilon}\varphi^{\varepsilon} = \overline{L}_{0}\varphi  + O(\varepsilon)$.
 Then, $\varphi^{\varepsilon}(x,y) \to \varphi(x)$ and $L^{\varepsilon}\varphi^{\varepsilon}(x,y) \to \overline{L}_{0}\varphi(x)$ when $\varepsilon \to 0$ as expected.  \\
 \indent The perturbed test function method has been extended in {\color{blue} \cite{Blankenship_Papanicolaou_1978, Kushner_jump-diffusion_1979, Kushner_martingale_1980}} and in the books {\color{blue} \cite[\color{black} Section 6.3]{Fouque_Garnier_Papanicolaou_Solna_2007}}, {\color{blue} \cite{Kushner_approximation_1984}}. Many other references apply the perturbed test function strategy in various settings: in finance {\color{blue} \cite{Fouque_Papanicolaou_Sircar_Solna_2003}} in transport problems as in {\color{blue} \cite{Papanicolaou_1977}} or in {\color{blue} \cite{Costantini_Kurtz_2006}} where the tightness of the fast component is established using its occupation measure. Similar methods are used in homogenisation {\color{blue} \cite{Papanicolaou_1977, Bensoussan_Lions_Papanicolaou_2011, Pavliotis_multiscale_2008}}  (see also {\color{blue} \cite{Benoist_Bernardin_Chetrite_Chhaibi_Najdunel_Pellegrini_2021}} which exploits spectral and semi-group properties in addition), and in stochastic stability and control {\color{blue} \cite{Kushner_approximation_1984}}.   \\ 
 \indent Finally, an important method is the stochastic averaging, in a generic framework, that was   proposed by \textsc{Kurtz} in {\color{blue} \cite{Kurtz}}.  Many approaches {\color{blue} \cite{Ball_Kurtz_Popovic_Rempala, Meleard_Tran, Gupta, costa_piecewise_2016, Genadot_2019, Bonnet_Meleard_2021, Ballif_Clement_Yvinec_2022, Billiard_Ferriere_Meleard_Tran_2015}}  including ours, are based on it. The main idea consists in exploiting the \emph{occupation measure} $\Gamma^{\varepsilon}$ of the fast component $Y^{\varepsilon}$ which is formally defined for all $t \geqslant 0$, for any Borelian $B$ by \[\Gamma^{\varepsilon}\left([0, t] \times B \right) = \int_{0}^{t}{\II_{B}\left(Y^{\varepsilon}_{s}\right)\dd s},\] and to establish the convergence of the couple $\left(X^{\varepsilon}, \Gamma^{\varepsilon}\right)$ when $\varepsilon \to 0$ using \emph{compactness-uniqueness} arguments. Proceeding in this way allows us to escape the difficulties created by the fluctuations of the fast component and avoids to obtain tightness result of the ``slow-fast'' couple $\left(X^{\varepsilon}, Y^{\varepsilon} \right)$.  The proof can be divided into four steps. \\
 \indent The first one consists in establishing uniform tightness of the sequence of laws of the couple  $\left(X^{\varepsilon}, \Gamma^{\varepsilon}\right)$.  The second one consists in establishing a martingale problem for any limiting value of the family of laws of the couple. The third one is based on establishing the uniqueness of the limiting value of $\Gamma^{\varepsilon}$ expressed in terms of the limit $X$ of $X^{\varepsilon}$ which is given. The characterisation of the limiting value $\Gamma$ of $\Gamma^{\varepsilon}$ is usually based on ergodicity arguments. Finally, the last step is to establish the uniqueness of the martingale problem for the slow component limit $X$ when $\Gamma$ is given as above. \\
\indent The \textsc{Kurtz} approach seems to be better adapted to our situation than the perturbed test function method. {\color{black} Indeed, %the generator $L_{1}$ being that of a measure-valued diffusion, 
it is delicate to find the good class of test function $\varphi_{1}$ satisfying  (\ref{Eq_Phi_1_Intro}) and the computations are difficult because of the moment terms (see Section \ref{Sous-Section_2_3_Sketch_of_proof}).} In addition, the remainder terms generated by (\ref{Eq_L_epsilon_Intro}) are difficult to control and the fast component  $Y^{\varepsilon}$ does not take the form $Y(t/\varepsilon)$. However, \textsc{Kurtz}'s approach can neither be implemented directly because of difficulties inherent to our model described in Section \ref{Section_2_IBM}. Therefore, our result required a complete reworking of the classical arguments. In particular,  in our case, we do not have uniform moment estimates but only a uniform control, in probability, of the second moment of the fast variable up to a stopping time $\check{\tau}^{K}$ where the diameter of the support of the fast component becomes large. \\ 
 \indent This paper is organised as follows. In Section \ref{Section_2_IBM}, we define our trait structured IBM,  state the central theorem about the CEAD characterising the limit model which consists of an ODE ruling the dynamics of the limit slow component. We establish in this section the sketch and outline of the main proof and the difficulties encountered. We give also some prospects that let us believe that we can improve our main result by relaxing some assumptions. The rest of this paper is devoted to prove the central theorem of Section \ref{Section_2_IBM}. In Section \ref{Section_3_Gene_LENT_RAPIDE} some approximations of the infinitesimal generator of the slow-fast process are proved. In Section \ref{Section_4_Estimees_de_moments}, we establish estimates of moments, in particular of the sixth and second order moment up to time $\check{\tau}^{K}$.  We also prove the convergence of $\check{\tau}^{K}$ to $+\infty$ when $K \to +\infty$ exploiting coupling arguments between the moment of order $2$ and a certain biased random walk for which large deviations estimates are established. In Section \ref{Section_5_Tension_LENT-RAPIDE}, we prove the tightness of the couple ``slow-occupation measure fast'' in the torus case.  In Section \ref{Section_6_Caract_Gamma_Limite}, we establish, in the torus case, the uniqueness of the limit occupation measure  which is described by the unique invariant probability measure of the centered \textsc{Fleming-Viot} process. This result is used  in Section \ref{Section_7_Caract_Composante_LENTE} to characterise the limit slow component as the unique solution of the CEAD with values in the torus. Thanks to the non-explosion of this ODE and choosing the torus large enough we are able to conclude the proof of our main result given in Section \ref{Section_2_IBM}. 

 \section{A trait structured IBM\label{Section_2_IBM}}
Let us describe the microscopic model which models population evolution, in a Darwinian sense, at the individual level. 

\subsection{Parameters and assumptions of the model}
\label{sec:model}

 We consider a discrete population of constant size in continuous time in which the survival and reproductive capacity of each individual is characterised by a continuous quantitative phenotypic trait $x \in \R$, i.e. an overall characteristic subject to selection such as body size at maturity.  The individuals reproduce asexually during their lives, i.e. the reproduction scheme is assumed clonal simultaneously with death of another individual, with frequency-dependent rates. Each birth of an individual occurs simultaneously with the death of another individual in the population. \\
 
 We are interested in approximating the long-term dynamics of a large population. We assume that the number of individuals alive at each time $t \geqslant 0$ is always equal to $K$. Let us denote by $x_{1}(t), \cdots, x_{K}(t)$, the phenotypic trait values of these individuals at time $t$. {\color{black} A mutation occurs during life at the individual level at rate $1$: in this sense mutations are considered \emph{frequent}. Each mutation has an amplitude of the order of magnitude ${\color{black}\sigma_{K}} \in (0, + \infty)$. Small ${\color{black}\sigma_{K}}$ means mutations have a small phenotypic effect, i.e. evolution acts slowly on the individual phenotypic characteristics.} The state of the population at time $t$, can be described by the finite point measure on $\R$ rescaled by~$K$
\begin{equation*}
\nu_{t}^{K} := \frac{1}{K}\sum\limits_{i\, = \, 1}^{K}{\delta_{x_{i}(t)}} 
\label{Processus_Nu_t}
\end{equation*} 
where $\delta_{x}$ is the \textsc{Dirac} measure at $x$.  \\ 

For all $x, y \in \R$, let us introduce the following  biological parameters:
\begin{itemize}
\item $b(x,y) \in \R_{+}$ is the \emph{resampling} rate, i.e. the rate of simultaneous birth of an individual holding trait $y$ and death of an individual holding trait $x$ and it can be interpreted as modeling a competitive pressure driving selection.
\item $\theta(x) \in \R_{+}$ is the rate of mutation of an individual holding trait $x$.
\item $m(x, \dd h)$ is the mutation law of the scaled mutation step $h$, born from an individual with trait $x$, where the mutant trait is given by $y := x + {\color{black}\sigma_{K}} h \in \R$. It is a probability measure on~$\R$.
\end{itemize}

Let us also introduce the following notations, used throughout this paper:
\begin{itemize}
\item $\beta(x) := \frac{\theta(x)}{b(x,x)}$, the ratio between the mutation rate and the resampling rate in a monomorphic population with trait $x$, which can be interpreted (in the scaling limit considered below) as the mean number of mutations between two resampling events.
\item $\Fit(y,x) := b(x,y) - b(y,x)$ is the adaptive value or \emph{fitness} of a mutant individual with trait $y$ in the population of $(K-1)$ individuals of trait $x$. The fitness can be interpreted as the initial growth rate of a mutant individual $y$ in a resident monomorphic population with trait $x$. \\
\indent Indeed, the total birth rate of a mutant individual $y$ in this population is $b(x,y)(K-1)$ and the total death rate of a mutant individual $y$ in this population is $b(y,x)(K-1)$. Hence, the (initial) growth rate of the mutant population is $\Fit(y,x)(K-1)$. 
\end{itemize}

Let $\MM_{1}(\R)$ be the set of probability measures on $\R$, endowed with the topology of weak convergence making it a Polish space {\color{blue} \cite{Billingsley}}. For a measurable real bounded function $f$, and a measure $\nu \in \MM_{1}(\R)$, we denote $\left\langle f, \nu \right\rangle := \int_{\R}^{}{f(x)\nu(\dd x)}.$ We denote by $\id$ the identity function. We denote also $\N :=\{0, 1, 2, \cdots \}$ and $\N^{\star} := \N \setminus{\{0\}}$. If $\III$ is an interval of $\R$, then for all $\ell \in \N$, $q\in \N^{\star}$, we denote by $\CCCC^{\ell}(\III^{q},\R)$ the space of functions of class $\CCCC^{\ell}$ from $\III^{q}$ to $\R$ and by $\CCCC^{\ell}_{b}(\III^{q},\R)$ the space of functions of class $\CCCC^{\ell}(\III^{q},\R)$ with bounded derivatives. Finally, we define for all $K \in \N^{\star}$,
\begin{align*}
 \MM_{1,K}(\R) & := \left\{\left.\frac{1}{K}\sum\limits_{i\,= \, 1}^{K}{\delta_{x_{i}}} \right| \left(x_{i}\right)_{1\leqslant i \leqslant K} \in \R^{K} \right\} %, \\
% \widehat{\MM}_{1}^{K, \sigma}(\R) & := \left\{\frac{1}{K}\sum\limits_{i\,= \, 1}^{K}{\delta_{\frac{1}{\sigma \sqrt{K} }}\left(x_{i} - \left\langle \id, \nu \right\rangle \right)} \left| \phantom{1^{1^{1^{1^{1^{1}}}}}} \hspace{-1cm} \right. \left(x_{i}\right)_{1\leqslant i \leqslant K} \in \XX^{K}, \nu \in \MM_{1}^{K}(\XX) \right\}.%,
\end{align*}
the set of probability measures on $\R$ of $K$ atoms of mass $1/K$. \\

\textbf{Assumptions.} Let us denote by \textbf{(A)} the following two assumptions: \\

\indent \textbf{(A1)} The maps $b : \R^{2} \to \R_{+}$, and $\theta : \R \to \R_{+}$ are respectively in $\CCCC_{b}^{2}(\R^{2}, \R)$ and $\CCCC_{b}^{2}(\R, \R)$ and there exists $0 < \underline{b}, \overline{b}, \underline{\theta}, \overline{\theta} < \infty$ such that: 
\[ \underline{b} \leqslant b(\cdot, \cdot) \leqslant \overline{b},  \qquad {\rm{and}} \qquad \underline{\theta} \leqslant \theta(\cdot) \leqslant \overline{\theta}. \]   
 \indent \textbf{(A2)} \textbf{(a)} There exists $A_{m} \in (0, +\infty)$ such that the law $m(x, \dd h)$ is absolutely continuous with respect to the \textsc{Lebesgue} measure on $\R$ with density $m(x,h)$ centered and satisfies \[\forall x \in \R, \ \forall \left|h \right| \geqslant A_{m}, \quad m(x,h) = 0. \]
\indent \indent \textbf{(b)} For all $\alpha \in \R$, for all $\ell \in \N$, we denote by \[m_{\ell}\left(\alpha \right) := \int_{\R}^{}{\left|h\right|^{\ell}m\left(\alpha, h \right)\dd h} \] the $\ell^{\rm{th}}$ moment of the mutation law. We assume for $\ell \in \{2, 4, 6\}$ that $m_{\ell}$ is \textsc{Lipschitz} and there exists $\underline{m}_{\ell}, \overline{m}_{\ell} \in (0, +\infty)$ such that for all $\alpha \in \R$, $\underline{m}_{\ell} \leqslant m_{\ell}(\alpha) \leqslant \overline{m}_{\ell}$. \\

Let us now give the infinitesimal generator $L^{K}$ of the $\MM_{1,K}(\R)-$valued \textsc{Markov} process $\nu^{K} := \left(\nu_{t}^{K} \right)_{t\geqslant 0}$ describing the ecological dynamics of the population with resampling. The generator $L^{K}$ is defined for any bounded measurable map $\phi$ from $\MM_{1,K}(\R)$ to $\R$, by
\begin{equation}
		\begin{aligned}
			L^{K}\phi(\nu) & = K \int_{\R}^{}{\nu(\dd x)\int_{\R}^{}{\nu(\dd y)b(x,y)\left[\phi\left(\nu - \frac{\delta_{x}}{K} + \frac{\delta_{y}}{K} \right) - \phi(\nu) \right]}} \\
& \qquad +  K\int_{\R}^{}{\theta(x)\nu(\dd x)\int_{\R}^{}{m(x,h)\left[\phi\left(\nu - \frac{\delta_{x}}{K} + \frac{\delta_{x + {\color{black}\sigma_{K}} h}}{K} \right) - \phi(\nu) \right]\dd h} }.
		\end{aligned} \label{Generateur_initial}
\end{equation}

The first term describes the resampling event of one individual by another and the second term describes the effect of mutations over the lifetime of individuals. 

\subsection{Main result}

The main result of this article is the following theorem. For all $\nu \in \MM_{1}(\R)$, let us denote by $\Diam\left(\Supp \nu \right)$ the diameter of the support of $\nu$.
\begin{Thm} \label{Thm_CEAD_Jouet}
Assume {\rm \textbf{(A)}} and for all $K \in \N^{\star}$, ${\color{black}\sigma_{K}} \in (0, +\infty)$, for some $x_{0} \in \R$, for all $\ell \in \{1,2,3 \}$, 
\begin{equation}
\left\langle \id, \nu_{0}^{K} \right\rangle := x_{0}, \qquad {\rm and} \qquad \left\langle \left(\id - x_{0} \right)^{2\ell}, \nu_{0}^{K} \right\rangle \leqslant C_{2\ell}^{\star} K^{\ell}{\color{black}\sigma_{K}^{2\ell}}
\label{Hypothese_C_etoile}
\end{equation}
 for some $C_{2 \ell}^{\star}>0$. Assume that $K \to +\infty$, ${\color{black}\sigma_{K}} \to 0$ such that 
\begin{equation}
\exists \, \varepsilon>0, \quad \forall C >0,  \qquad K^{-C\log(K)} \ll {\color{black}\sigma_{K}} \ll  K^{-\left(2 + \varepsilon\right)}.
\label{Hypothese_Gamme_sigma}
\end{equation}

\indent Then, for the \textsc{Skorohod} topology, the sequence of the mean trait processes $$\left(\left(\left\langle \id, \nu_{t/K{\color{black}\sigma_{K}^{2}}}^{K} \right\rangle \right)_{t\geqslant 0}\right)_{K \in \N^{\star}}$$ converges in law in $\D\left(\R_{+}, \R\right)$ when $K \to +\infty$, ${\color{black}\sigma_{K}} \to 0$, to the unique deterministic process $\left(\zeta_{t}\right)_{t \geqslant 0} \in \CCCC^{0}\left(\R_{+}, \R \right)$ with initial condition $\zeta_{0} := x_{0}$, which is solution to 
the \emph{Canonical Equation of Adaptive Dynamics (CEAD)}: 
\begin{equation*}
\dot{y} = \partial_{1}\Fit(y,y) \times \frac{\beta(y)m_{2}(y)}{2}.
\label{Eq_CEAD_Jouet}
\end{equation*}
In addition, we have the following support concentration property : for all $T>0$, 
\begin{equation}
\sup_{t\in[0,T]}{\Diam \left(\Supp \nu^{K}_{t/K{\color{black}\sigma_{K}^{2}}} \right)} \leqslant \frac{1}{K}
\label{Eq_Thm_Concentration_Support}
\end{equation}
holds with probability which tends to $1$ when $K\to +\infty$. 
\end{Thm}
The canonical equation is composed of two terms. The first term:  the fitness gradient, describes the strength of selection that pushes the population to increase its adaptive value locally. The second term describes the effect of mutations. Note that (\ref{Eq_Thm_Concentration_Support}) describes, in some sense, when $K \to + \infty$, the convergence of the population distribution $\left(\nu_{t/K{\color{black}\sigma_{K}^{2}}}^{K} \right)$ to a \textsc{Dirac} mass. 

\begin{Rem} \begin{itemize}
\item[{\rm (i)}] We conjecture that the assumption ${\color{black}\sigma_{K}} \ll K^{-(2 + \varepsilon)}$ given by {\rm (\ref{Hypothese_Gamme_sigma})} is too restrictive and that it can be weakened by ${\color{black}\sigma_{K}} \ll K^{-(\frac{3}{2} + \varepsilon)}$. This conjecture is supported by the discussion in {\rm Section \ref{Sous-sous-section_Perspectives}} below. 
\item[{\rm(ii)}] We actually prove {\rm Theorem \ref{Thm_CEAD_Jouet}} (see {\rm Section \ref{Sous-sous-section_4_5_3_Attractive_domain_exit_estimates_for_random_walks}}) under a slightly weaker assumption than the left bound of {\rm (\ref{Hypothese_Gamme_sigma})}: for a universal constant $C_{0}$ depending only on $b, \theta, m$ and not on $K$, % and $\sigma$, 
\[\exists \, \varepsilon >0, \qquad K^{-\varepsilon^{2} C_{0}\log(K)} \ll {\color{black}\sigma_{K}} \ll K^{-\left(2+\varepsilon\right)}. \]
\item[{\rm(iii)}] Assumption {\rm\textbf{(A2)(a)}} is only used in {\rm Section \ref{Sous_section_4_5_Sortie_domaine}} and could be weakened with some tail bounds (typically exponential) on $m(x,\cdot)$.
\end{itemize}
\end{Rem}

The aim of the rest of the paper is to prove Theorem \ref{Thm_CEAD_Jouet}. %Since $\sigma = \sigma(K)$, we will use the notation $\nu^{K}$ to designate $\nu^{K, \sigma}$.

\subsection{Sketch and outline of the proof\label{Sous-Section_2_3_Sketch_of_proof}}
Note that, the convergence result of Theorem \ref{Thm_CEAD_Jouet} takes place on the time scale $1/K{\color{black}\sigma_{K}^{2}}$. The idea of the proof is based on slow-fast asymptotic analysis techniques developed by \textsc{Kurtz} {\color{blue} \cite{Kurtz}}. Let us begin by introducing the slow and fast dynamics involved in our model,  then the difficulties encountered and finally the outline of the proof.   

\subsubsection{Slow-fast asymptotic analysis\label{Sous-sous-section_2_3_1_Slow-fast_Heuristic}}
Let $a, \alpha \in \R$ and $\BB(\R)$ the \textsc{Borel} $\sigma-$field on $\R$. Let us define respectively by $\tau_{\alpha}$ and $h_{a}$, the translation operator of vector $\alpha$ and the homothety of ratio $a$. For all $x \in \R$, for all $A \in \BB(\R)$, \[\left(h_{a} \circ \tau_{\alpha} \right) \sharp\, \delta_{x}(A) = \delta_{a(x + \alpha)}(A) \qquad {\rm{and}} \qquad\left(\tau_{\alpha} \circ h_{a}  \right) \sharp \, \delta_{x}(A) = \delta_{a x + \alpha}(A).\]
Note that for all $K \in \N^{\star}$, any $\nu \in \MM_{1, K}(\R)$ has all its moments finite. We define for all $K \in \N^{\star}$, $\MM_{1, K}^{c}(\R)$ the set of centered probability measures of $\MM_{1, K}(\R)$.  From the population process $\left(\nu_{t/K{\color{black}\sigma_{K}^{2}}}^{K} \right)_{t\geqslant 0}$, 
we introduce two evolutionary dynamics:
\begin{itemize}
\item The slow dynamics, with value in $\R$, corresponds to that of the mean trait defined by 
\begin{equation*}
 z^{K}_{t} := \left\langle \id, \nu_{t/K{\color{black}\sigma_{K}^{2}}}^{K} \right\rangle. \label{Def_trait_moyen}
\end{equation*}
\item The fast dynamics, with value in $\MM_{1, K}^{c}(\R)$, corresponds to that of the centered and dilated distribution of traits defined by
\begin{equation}
\mu_{t}^{K} := \left(h_{\frac{1}{{\color{black}\sigma_{K}} \sqrt{K}}} \circ \tau_{-z_{t}^{K}} \right) \sharp \, \nu_{t/K{\color{black}\sigma_{K}^{2}}}^{K} = \frac{1}{K}\sum\limits_{i\, = \, 1}^{K}{\delta_{\frac{1}{{\color{black}\sigma_{K}} \sqrt{K}}\left(x_{i} (t)- z_{t}^{K} \right) }}. \label{Def_distribution_recentre_dilate}
\end{equation}
Note that $\nu_{t}^{K} = \left(\tau_{z_{tK{\color{black}\sigma_{K}^{2}}}^{K}} \circ h_{{\color{black}\sigma_{K}} \sqrt{K}} \right) \sharp \, \mu_{tK{\color{black}\sigma_{K}^{2}}}^{K} $.
\end{itemize}
For all $\ell \in \N$ and $\nu \in \MM_{1}(\R)$, we denote by $M_{\ell}\left(\nu\right) := \left\langle \left| \id \right|^{\ell}, \nu \right\rangle$ the $\ell^{\rm{th}}$ moment of the measure $\nu$. Let us consider for all $\ell \in \N^{\star}$, \[\MM_{1}^{c,\ell}(\R) := \left\{ \mu \in \MM_{1}(\R) \left| \phantom{1^{1^{1^{1}}}} \hspace{-0.65cm} \right. M_{\ell}\left(\mu \right) < \infty, \left\langle \id, \mu \right\rangle = 0\right\}\] and the convention $M_{\ell}\left(\mu \right) = +\infty$ if $\mu \notin \MM_{1}^{c, \ell}(\R)$. Let us denote by $\BB\left(E, \R \right)$ the set of \textsc{Borel} functions on the metric space $E$.\\ 

In the following, for all $K \in \N^{\star}$, we study the couple of processes $\left(\left(z^{K}_{t}, \mu^{K}_{t} \right)\right)_{t\geqslant 0}$. Denoting $\LL^{K}$ the infinitesimal generator of the process $\left(\left(z^{K}_{t}, \mu^{K}_{t} \right)\right)_{t\geqslant 0}$, computed in Proposition \ref{Prop_Generateur_Couple_Lent_Rapide}, we will check the following assertions, proved respectively in Propositions \ref{Prop_Generateur_Lent_Decomposition} and \ref{Prop_Generateur_Rapide_Decomposition}.
\begin{itemize}
\item[(i)] If we consider a function $f : (z, \mu) \mapsto f(z) \in \CCCC_{b}^{2}(\R, \R)$, the generator $\LL^{K}$ satisfies the following decomposition 
\begin{equation}
\LL^{K}f\left(z, \mu \right) =  \LL_{{\rm SLOW}}f\left(z,\mu\right) + O\left(\frac{1}{K^{2}} + \frac{M_{2}(\mu)}{K} + {\color{black}\sigma_{K}}\sqrt{K}M_{3}\left(\mu \right) \right)
\label{Decomposition_L_Phi_L_SLOW}
\end{equation}
where the operator $\LL_{\rm SLOW}$, of the limit slow component, is defined from $\CCCC^{1}_{b}(\R, \R)$ to $\BB\left(\R \times \MM_{1}^{c, 2}(\R), \R \right)$ by
\begin{equation}
\LL_{{\rm SLOW}}f\left(z, \mu\right) = f'(z)M_{2}\left(\mu \right)\partial_{1}\Fit(z,z).
\label{Generateur_LENT} 
\end{equation}
\item[(ii)] If we consider a function $F_{\varphi} : (z, \mu) \mapsto F_{\varphi}(\mu) := F\left(\left\langle \varphi, \mu \right\rangle \right)$ with $F \in \CCCC^{3}(\R, \R)$ and $\varphi \in \CCCC_{b}^{3}(\R, \R)$, the generator $\LL^{K}$ satisfies the following decomposition 
\begin{equation}
\LL^{K}F_{\varphi}\left(z, \mu \right) = \frac{\theta(z)m_{2}(z)}{K^{2}{\color{black}\sigma_{K}^{2}}} \left[\LL_{\rm FVc}^{\lambda(z)}F_{\varphi}(z, \mu) + O\left(\frac{1}{\sqrt{K}} + {\color{black}\sigma_{K}} K^{\frac{3}{2}}M_{2}\left(\mu \right) + \frac{M_{3}(\mu)}{K}\right) \right]
\label{Decomposition_L_Phi_L_FAST}
\end{equation}
where $\LL_{{\rm FVc}}^{\lambda(z)}$ is the generator of the \emph{centered \textsc{Fleming-Viot} process} with resampling rate $\lambda(z) := \frac{b(z,z)}{\theta(z)m_{2}(z)}$, as studied in {\color{blue} \cite{Champagnat_Hass_FVr_2022}} and whose definition is recalled below. For all $F \in \CCCC^{2}(\R, \R)$ and $g \in \CCCC_{b}^{2}\left(\R, \R \right)$, 
\begin{equation}
\begin{aligned}
\LL_{{\rm FVc}}^{\lambda}F_{g}\left(\mu\right) & := F'\left(\left\langle g, \mu \right\rangle \right)\left(\left\langle \frac{g''}{2}, \mu \right\rangle + \lambda \left[\left\langle g'', \mu  \right\rangle M_{2}(\mu) - 2\left\langle g' \times \id , \mu \right\rangle \right] \right) \\
& \hspace{-0.8cm} + \lambda F''\left(\left\langle g, \mu  \right\rangle \right)\left[ \left\langle g^{2}, \mu  \right\rangle - \left\langle g, \mu \right\rangle^{2} +  \left\langle g', \mu  \right\rangle^{2}M_{2}\left(\mu \right) - 2\left\langle g', \mu  \right\rangle\left\langle g \times \id, \mu   \right\rangle  \right].
\end{aligned} 
\label{Generateur_FVc}
\end{equation}
\end{itemize}

\noindent From (\ref{Decomposition_L_Phi_L_FAST}) and (\ref{Decomposition_L_Phi_L_SLOW}), note that the fast component $\left(\mu_{K{\color{black}\sigma_{K}^{2}}t}^{K} \right)_{t\geqslant 0} = \left(\left(h_{\frac{1}{{\color{black}\sigma_{K}}\sqrt{K}}}\circ \tau_{-\left\langle\id, \nu_{t}^{K} \right\rangle}\right) \sharp \, \nu_{t}^{K} \right)_{t\geqslant 0}$ moves on the evolutionary time scale $K$ whereas the slow component $\left(z_{K{\color{black}\sigma_{K}^{2}}t}^{K} \right)_{t\geqslant 0} = \left(\left\langle\id, \nu_{t}^{K} \right\rangle \right)_{t\geqslant 0}$ moves over a much longer evolutionary time scale $1/K{\color{black}\sigma_{K}^{2}}$. The different time scales involved in this model can be represented as in Figure \ref{Fig2}. It follows that the fast component will be the first to stabilise in its equilibrium state.  Note that the operator of the slow component (\ref{Generateur_LENT}) depends on the fast variable through the second-order moment $M_{2}\left(\mu\right)$. This is a standard difficulty in slow-fast analysis, usually solved by assuming that once the fast component is stabilised in its equilibrium state, $M_{2}\left(\mu \right)$ can be characterised in terms of the slow component leading to an autonomous slow dynamics in the limit $K \to +\infty$.  

 \definecolor{ffwwqq}{rgb}{1.0,0.0,0.0}
\definecolor{qqqqff}{rgb}{0.0,0.0,1.0}
\definecolor{qqzzqq}{rgb}{0.0,0.6,0.0}
\definecolor{ebfvff}{rgb}{0.9215686274509803,0.9607843137254902,1.}

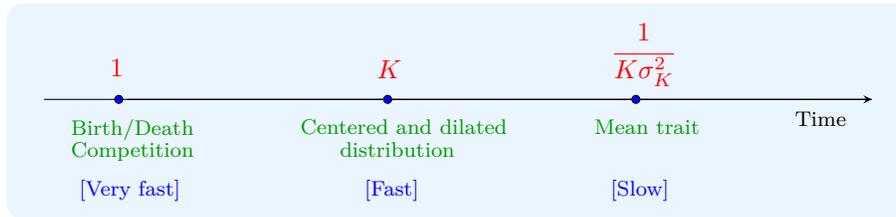
\begin{figure}[!h]
\begin{center}
\begin{minipage}{12cm} \centering
    \begin{tcolorbox}[arc=1ex, colback=ebfvff, colframe=ebfvff, left=3pt, right=3pt, top=3pt, bottom=2pt] 
    \begin{center}
\begin{tikzpicture}[line cap=round,line join=round,>=triangle 45,x=0.55cm,y=0.55cm]
\clip(-5.8,0.5) rectangle (14.2,4.85);
\draw [domain=-3.8:26] plot(\x,{(--75.0-0.0*\x)/25.0});
\draw [>=stealth,->] (-5.8,3.0) -- (14.2,3.0);
\draw (-4.45,4.2) node[anchor=north west] {\small $\color{ffwwqq}1$};
\draw (-5.4,2.76) node[anchor=north west] {\scriptsize \color{qqzzqq} \rm Birth/Death};
\draw (-5.4,2.21) node[anchor=north west] {\scriptsize \color{qqzzqq} \rm Competition}; % 2
\draw (-5.2,1.3) node[anchor=north west] {\scriptsize \color{blue} \rm [Very fast]};

\draw (2,4.2) node[anchor=north west] {\small $\color{ffwwqq}K$};
\draw (0.15,2.76) node[anchor=north west] {\scriptsize \color{qqzzqq} \rm Centered and dilated}; % 1. + 1.
\draw (1.1,2.21) node[anchor=north west] {\scriptsize \color{qqzzqq} \rm distribution};
\draw (1.7,1.3) node[anchor=north west] {\scriptsize \color{blue} \rm [Fast]};

\draw (7.65,5.2) node[anchor=north west] {\small $\color{ffwwqq}\cfrac{1}{K{\sigma_{K}^{2}}}$};
\draw (7.25,2.76) node[anchor=north west] {\scriptsize \color{qqzzqq} \rm Mean trait}; % 2
\draw (7.65,1.3) node[anchor=north west] {\scriptsize \color{blue} \rm [Slow]};

\draw (12.1,2.96) node[anchor=north west] {\scriptsize \rm Time};

\begin{scriptsize}
\draw [fill=qqqqff] (-4.0,3.0) circle (1.5pt);
\draw [fill=qqqqff] (2.5,3.0) circle (1.5pt);
\draw [fill=qqqqff] (8.5,3.0) circle (1.5pt);
\end{scriptsize}
\end{tikzpicture}   
   \end{center}
\end{tcolorbox}
\caption{Different time scales involved in our individual-based model expected in the original time scale of the process $\left(\nu_{t}^{K} \right)_{t\geqslant 0}$.}
\label{Fig2} 
\end{minipage}
\end{center}
\end{figure}

 As the fast limit component is driven by a centered \textsc{Fleming-Viot} process, we expect that it will inherit ergodicity properties  as stated in  {\color{blue} \cite[\color{black} Section 4]{Champagnat_Hass_FVr_2022}}. This reference establishes the existence of a unique invariant probability measure $\pi^{\lambda}$ for the centered \textsc{Fleming-Viot} process and characterises it. In particular,  we expect that the averaging principle applied to the drift coefficient in (\ref{Generateur_LENT}) will lead to the averaged drift
 \[\partial_{1}\Fit\left(z_{t}, z_{t}\right) \int_{\MM_{1}(\R)}^{}{M_{2}\left(\mu \right)\pi^{\lambda\left(z_{t} \right)}(\dd \mu)} = \frac{\partial_{1}\Fit\left(z_{t}, z_{t}\right)}{2\lambda\left(z_{t} \right)} = \partial_{1}\Fit\left(z_{t}, z_{t}\right) \frac{\beta\left(z_{t}\right)m_{2}\left(z_{t}\right)}{2}\] where $\left(z_{t} \right)_{t  \geqslant 0}$ is the limit slow component (see {\color{blue} \cite[\color{black} Corollary 4.16]{Champagnat_Hass_FVr_2022}} for the computation of the integral). Formally, by replacing $M_{2}\left(\mu\right)$ by its mean value $\int_{\MM_{1}(\R)}^{}{M_{2}\left(\mu \right)\pi^{\lambda\left(z_{t} \right)}(\dd \mu)}$ in (\ref{Generateur_LENT}), we obtain that (see details in Section \ref{Section_7_Caract_Composante_LENTE})
 \[\LL_{\rm CEAD}f(z,\mu) = f'(z) \partial_{1}\Fit(z, z)\frac{\beta(z)m_{2}(z)}{2}  \]
 and we recognise the generator of the announced CEAD. 

{\color{black}
\subsubsection{An example\label{Sous-sous-section_2_3_2_Example}}

The goal of this section is to illustrate numerically the phenomenon of the CEAD. We have performed numerical simulations of the process $\nu^{K}$ for a simple model from the biological literature. The simulation is based %on the acceptance-rejection method. The \texttt{R}-source of the program, provided in Appendix \ref{Appendix_R_code}, is based 
on the \texttt{IBMPopSim} package developed by \textsc{Giorgi et al} in {\color{blue} \cite{Kaakai_Giorgi_Lemaire_2023}}. \\

The biological model is adapted from \textsc{Kisdi} {\color{blue}\cite{kisdi_evolutionary_1999}} for which the trait space is $\XX := [0,4]$ (and not $\R$ as previously) and the parameters are\!\! : 
\begin{equation}
\theta(x) := 1, \quad c(x,y) :=  1 - \frac{1}{1+1.2\exp(-4(x-y))}, \quad b(x,y) := (4 - y) \times c(x,y)
\label{Eq_Param_KISDI}
\end{equation}
 and $m\left(x,\dd h\right)$ is a Gaussian distribution %of mean $x$ and variance $\sigma^{2}$ 
 $\NN\left(x, {\color{black}\sigma_K^{2}} \right)$ conditional on the mutant trait remaining in~{\color{black}$\XX$}. \\

Let us recall that $b(x,y)$ is a resampling rate, i.e. the rate of simultaneous death of an individual with trait $x$ and birth of an individual with trait $y$ and it can be interpreted as modelling a competitive pressure driving selection. Here, it is composed of two terms which both have a biological interpretation. Note that the first term $c(x,y)$ depends only on $x-y$ and tends to $0$ when $x-y \to +\infty$ and to $1$ when $x-y \to -\infty$. So, this function models an \emph{asymmetric competition} in favour of high traits: $x$ feels little competition from a smaller trait $y$, but strong competition from a higher trait $y$. Since the second term $4-y$ only depends on the reproducing trait and is favourable to small traits, evolution is expected to select an optimal trait. \\ \indent {\color{black} We consider the previous model with three scaling parameters: $K$ (for the population size) and $\sigma_K$ (for the size of mutation) as in Section~\ref{sec:model}, and a parameter $p_K$ giving the probability of mutation at each birth event.
  This amouts to replace in Section~\ref{sec:model} the mutation kernel $m(x,h){\dd}h$ by $p_K m(x,h){\dd}h+(1-p_K)\delta_x({\dd}h)$.}  We performed simulations of \textsc{Kisdi}'s model, starting with a monomorphic initial condition of trait $1$ and varying the parameters $K, {\color{black}p_K}$ and {\color{black}$\sigma_K$}. Some of these simulations are presented in the following figures and show a wide variety of evolutionary behaviours. \\
\indent Note that, when $K$ and {\color{black}$\sigma_K$} are not too large and {\color{black}$p_K$} not too small, %$\left(K \overset{<}{\sim} 5000, \ \sigma \overset{<}{\sim} 0.1, \ p \overset{>}{\sim} 0.001\right)$, 
the population evolves according to a relatively stable general scenario as in Figures \ref{Fig08}, \ref{Fig09} and \ref{Fig10}\!\! : it initially remains concentrated around the trait value equal to $1$ that progressively moves towards a trait $x^{\star}$ close to $3.5$ corresponding to the optimal trait value of this population, in the sense that it is best adapted to survive and reproduce.  Once this point is reached, the population stabilises in its steady state. {\color{black} In other words, Figures \ref{Fig09} and \ref{Fig10} show the long-term behaviour of the mean trait of the population, i.e. the one of the slow component. Figure \ref{Fig11} shows the behaviour of the fast component by zooming in the initial dynamics of Figure \ref{Fig10}.}  Note that Figure \ref{Fig07} illustrates the approach developed by 
{\color{blue} \cite{Champagnat_PTRF_2011}}.

\begin{figure}[!h]
\centering		
\begin{minipage}[h]{.47\textwidth} \centering
\begin{center}
\includegraphics[scale = 0.45]{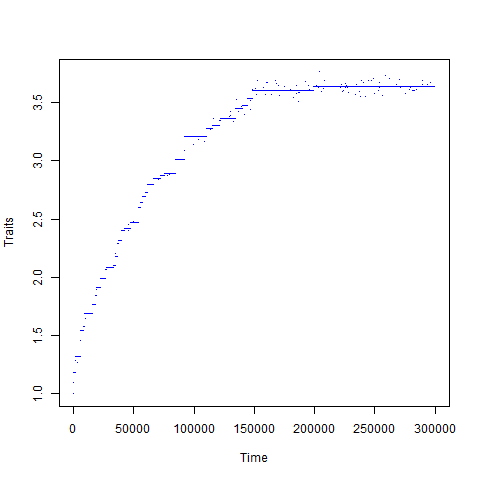}
\end{center}
\caption{Numerical simulation of the trait distribution of the microscopic model with parameters (\ref{Eq_Param_KISDI}) and $K = 300$, {\color{black}$p_K$} $= 0.00003$, {\color{black}$\sigma_K$} $= 0.05$.}
\label{Fig07}
\end{minipage}
\hspace{0.5cm}
\begin{minipage}[h]{.47\textwidth}\centering
\begin{center}
\includegraphics[scale = 0.45]{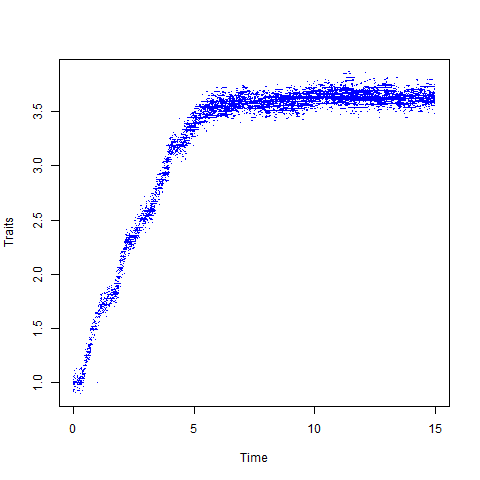} 
\end{center}
\caption{Numerical simulation of the trait distribution of the microscopic model with parameters (\ref{Eq_Param_KISDI}) and $K = 300$, {\color{black}$p_K$} $= 1$, {\color{black}$\sigma_K$} $= 0.05$.}
\label{Fig08}
\end{minipage}
\end{figure}

\begin{figure}[!h]
\centering		
\begin{minipage}[h]{.47\textwidth} \centering
\begin{center}
\includegraphics[scale = 0.45]{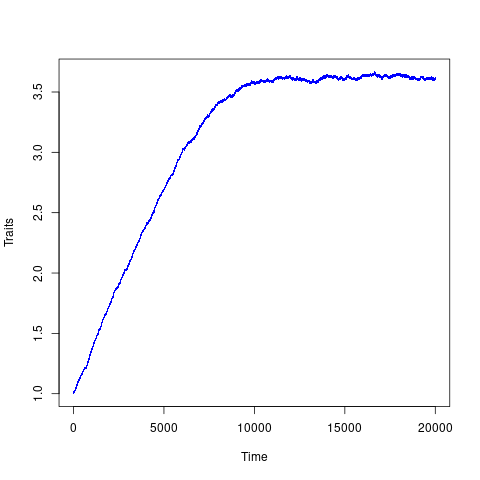}
\end{center}
\caption{Numerical simulation of the trait distribution of the microscopic model with parameters (\ref{Eq_Param_KISDI}) and $K = 300$, {\color{black}$p_K$} $= 1$, {\color{black}$\sigma_K$} $= 0.001$.}
\label{Fig09}
\end{minipage}
\hspace{0.5cm}
\begin{minipage}[h]{.47\textwidth}\centering
\begin{center}
\includegraphics[scale = 0.45]{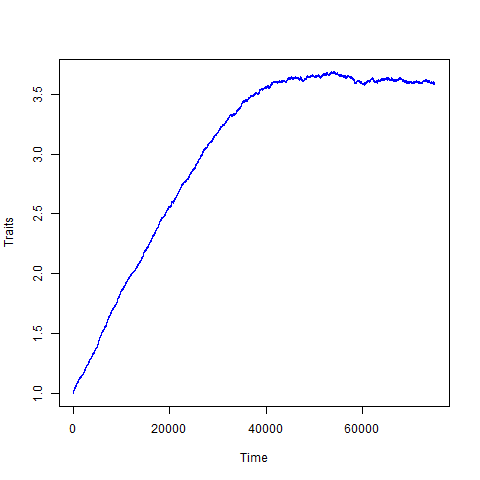} 
\end{center}
\caption{Numerical simulation of the trait distribution of the microscopic model with parameters (\ref{Eq_Param_KISDI}) and $K = 300$, {\color{black}$p_K$} $= 1$, {\color{black}$\sigma_K$} $= 0.0005$.}
\label{Fig10}
\end{minipage}
\end{figure}

\begin{figure}[!h]
\begin{center}
\includegraphics[scale = 0.45]{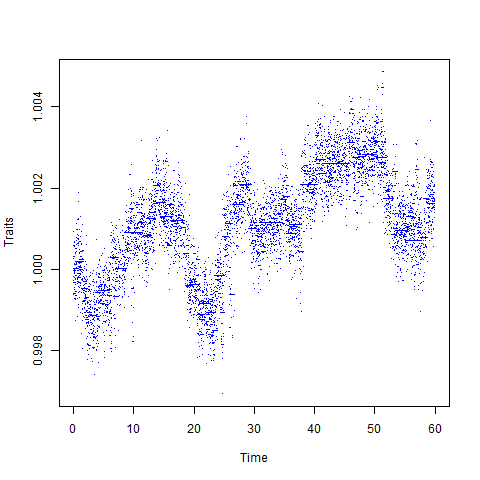}
\end{center}
\caption{Numerical simulation of the trait distribution of the microscopic model with parameters (\ref{Eq_Param_KISDI}) and $K = 300$, {\color{black}$p_K$} $= 1$, {\color{black}$\sigma_K$} $= 0.0005$.}
\label{Fig11}
\end{figure}

}
 
 \subsubsection{Difficulties encountered\label{Sous-sous-section_2_3_2_Difficultes}} 
The slow-fast analysis relies on a stochastic averaging result that exploits tightness arguments. The classical approach to prove tightness of sequences of laws {\color{blue} \cite{Bans, Fournier}} requires to have uniform pathwise estimates on the moments of the process. In our situation, we note the presence of moments in (\ref{Decomposition_L_Phi_L_SLOW}) and (\ref{Decomposition_L_Phi_L_FAST}). However, we could not establish control in expectation of $\sup_{0\leqslant t \leqslant T}{M_{2}\left(\mu_{t}^{K} \right)}$ for all $T \geqslant 0$ because of the long time scale which does not allow to exploit the martingale problem associated to the decomposition of $M_{2}\left(\mu_{t}^{K} \right)$. This is an important difference, for example with {\color{blue} \cite{Meleard_Tran}} which studies similar forms of processes (individual-based models). Instead, we will first use fine pathwise estimates and expectation bounds on $M_{2}\left(\mu_{t}^{K}\right)$ up to the stopping time $\tau^{K}$ defined for all $K \in \N^{\star}$ by 
\begin{equation}
 \tau^{K} := \check{\tau}^{K} \wedge \widehat{\tau}^{K}.
 \label{Temps_arret_check_chapeau}
 \end{equation}
where 
\begin{equation}
\widehat{\tau}^{K}  := \inf\left\{t\geqslant 0 \left| \phantom{1^{1^{1^{1}}}} \hspace{-0.7cm} \right.   M_{2}\left(\mu^{K }_{t}\right) \geqslant K^{\varepsilon}  \right\}
\label{Temps_arret_chapeau}
\end{equation}
with $\varepsilon>0$ given in Theorem \ref{Thm_CEAD_Jouet} and where
\begin{equation}
\check{\tau}^{K }  := \inf\left\{t\geqslant 0 \left| \phantom{1^{1^{1^{1}}}} \hspace{-0.7cm} \right.  \Diam\left(\Supp{\mu^{K}_{t}}\right) > \frac{1}{{\color{black}\sigma_{K}} K^{\frac{3 + \varepsilon}{2}}}\right\}
\label{Temps_arret_check}
\end{equation}
 which allows to control the diameter of the support of the centered and dilated distribution of the traits $\mu^{K}$. An important difficulty arises in Lemma \ref{Cor_moment_ordre_2}  from the presence of third order moment in the \textsc{Doob} semi-martingale decomposition of $M_{2}\left(\mu_{t}^{K} \right)$: 
  \begin{align*}
 M_{2}\left(\mu^{K }_{t} \right) & =   M_{2}\left(\mu_{0}^{K } \right) - \frac{1}{K^{2}{\color{black}\sigma_{K}^{2}}}\int_{0}^{t}{\left\{2 b\left(z_{s}^{K }, z_{s}^{K } \right) M_{2}\left(\mu_{s}^{K } \right) - \theta\left(z_{s}^{K} \right)m_{2}\left(z_{s}^{K} \right) \right\}\dd s } \\
& \hspace{1cm}   +  \frac{1}{K^{2}{\color{black}\sigma_{K}^{2}}} \int_{0}^{t}{O\left(\frac{1}{K} + {\color{black}\sigma_{K}} K M_{2}\left(\mu_{s}^{K} \right) + {\color{black}\sigma_{K}} K^{\frac{3}{2}}M_{3}\left(\mu_{s}^{K} \right) \right)\dd s} + M_{t}^{K, P_{\id^{2}, 1}}
\end{align*}
where $\left(M_{t}^{K, P_{\id^{2}, 1}} \right)_{t\geqslant 0}$ is a local-martingale. This leads to introduce the stopping time $\check{\tau}^{K}$ which guarantees that the error term in $M_{3}\left(\mu_{t}^{K} \right)$ remains under control. This leads to a new difficulty: establishing that $\check{\tau}^{K}$ tends to $+\infty$ in probability when $K \to +\infty$. To prove this, we use pathwise estimates on $M_{2}\left(\mu_{t}^{K} \right)$ up to time $\check{\tau}^{K}$ to estimate the different transitions of $M_{2}\left(\mu_{t}^{K} \right)$ between thresholds of the form $3^{\ell - 1}K^{\frac{\varepsilon}{2}}$ and $3^{\ell + 1}K^{\frac{\varepsilon}{2}}$  and to construct a coupling between these transitions and biased random walks. This allows us to use large deviations results on random walks and estimates on the exit from an attracting domain (see e.g. {\color{blue} \cite{Freidlin_Wentzell_1984, Dembo_Zeitouni_2010}}) to prove that the stopping time $\widehat{\tau}^{K}$, defined by (\ref{Temps_arret_chapeau}), converges to $+ \infty$ in probability when $K \to +\infty$. This in turn implies that $\check{\tau}^{K} \to +\infty$ in probability when $K \to +\infty$ (see Section \ref{Sous-sous-section_Perspectives} for a discussion of the method used for this step). \\
 \indent The implementation of the slow-fast method of \textsc{Kurtz} {\color{blue} \cite{Kurtz}} is done in two steps. To establish the tightness of the sequence of laws of the slow component, we exploit criteria developed by \textsc{Ethier-Kurtz} {\color{blue} \cite[\color{black} Theorems 3.9.1 and 3.9.4]{Ethier_markov_1986}} by restricting ourselves to the torus case. This strategy allows to overcome the difficulty related to the verification of the compact containment condition, on the real line, of the slow component stopped at the stopping time $\tau^{K}$.  The second step consists in characterising each accumulation point of the sequence of laws of the slow dynamics and the occupation measure of the fast dynamics. For this, we need  to check that any measure $\gamma$ in $\MM_{1}\left(\MM_{1}(\R) \right)$ satisfying 
    \begin{equation}
    \int_{\MM_{1}(\R)}^{}{\LL_{\rm FVc}^{\lambda} \phi(\mu) \gamma(\dd \mu) } = 0
    \label{Eq_Nullite_contre_gamma_difficulties_encoutered}
    \end{equation}
    for a certain class $\FF$ of functions $\phi$ must be $\pi^{\lambda}$. Equation (\ref{Decomposition_L_Phi_L_FAST}) suggests to take $\FF$ as the set of functions of the form $F_{\varphi}(\mu) := F\left(\left\langle \varphi, \mu \right\rangle \right)$. However, the last property seems hard to prove for this choice of $\FF$. Instead, we adapt a result of \textsc{Dawson} {\color{blue} \cite[\color{black} Theorem 2.7.1]{Dawson}} that applies to the set $\FF$ of so-called \emph{polynomial functions} of the form \[P_{f, n}(\mu) := \int_{\R}^{}{\cdots\int_{\R}^{}{f\left(x_{1}, \cdots, x_{n}\right)\mu(\dd x_{1})\cdots\mu(\dd x_{n})}}. \]
    However the duality property used by \textsc{Dawson} in {\color{blue} \cite{Dawson}} does not hold in our case. In {\color{blue} \cite{Champagnat_Hass_FVr_2022}}, it is proved only a weak duality relation involving stopping times. This difficulty will be solved by proving that the measure $\gamma$ in (\ref{Eq_Nullite_contre_gamma_difficulties_encoutered}) gives mass only to measures having its first four moments finite. This will allow us to characterise the limit fast component and hence the limit slow component on the torus, as solution to an ODE. Since this ODE is non-explosive, we choose the torus large enough to conclude the proof on the real line.
    
    \subsubsection{Outline of the proof} In Section \ref{Section_3_Gene_LENT_RAPIDE}, we begin by giving {\color{black} in Proposition \ref{Prop_Generateur_Couple_Lent_Rapide}} an approximation of the infinitesimal generator of the slow-fast process $\left(\left(z_{t}^{K}, \mu_{t}^{K}\right) \right)_{t\geqslant 0}$ for a class of test functions on $\R \times \MM_{1, K}^{c}(\R)$, large enough to be convergence determining and  we characterise martingales associated to our process.  We give also martingale problems of the slow-fast process for polynomials in $\mu$ {\color{black} where the main results are given in Lemma \ref{Lem_PB_Mg_Polynomes} for bounded test functions and Lemmas \ref{Cor_moment_ordre_2} and \ref{Cor_moment_ordre_6} for unbounded test functions}.  In Section \ref{Section_4_Estimees_de_moments}, we prove some moment estimates.  In Section \ref{Sous_Section_4_1_new_Moments_ordre_6}, we give estimates of the moment of order $6$. {\color{black} The most important result is given by Lemma \ref{Lem_Controle_M6_et_M4M2} because consequences are useful for all results in Section \ref{Section_4_Estimees_de_moments} as well as in Theorem \ref{Thm_Kurtz_adapte}.} In Section \ref{Sous_Section_4_3_Inegalite_M2}, we establish some inequalities on $M_{2}\left(\mu_{t}^{K} \right)$ and we control its bracket that we will use in Section \ref{Sous_Section_4_4_Convergence_tau} to prove 
 $\P-$a.s. that $M_{2}\left(\mu_{t}^{K} \right)$ takes superlinear (of the order of $K^{C\log(K)}$) long time before hitting $K^{\varepsilon}$. The fact that $\tau^{K} \to + \infty$ in probability when $K\to +\infty$ is proved in Section \ref{Sous_Section_4_4_Convergence_tau} {\color{black} using the technical Lemma \ref{Lem_Proba_sup_M2_K_epsilon}. Section \ref{Sous_section_4_5_Sortie_domaine} is dedicated to prove this lemma by constructing a coupling between the moment of order 2 and a biased random walk.} The rest of the proof deals with the compactness-uniqueness argument associated to our slow-fast problem. Firstly, in Section \ref{Section_5_Tension_LENT-RAPIDE}, we establish uniform tightness of the sequence of laws of $\left(\left(z^{K}, \Gamma^{K} \right)\right)_{K\in \N^{\star}}$ stopped at time $\check{\tau}^{K}$ in the torus case where $\Gamma^{K}$ designates the occupation measure of the process $\mu^{K}$. {\color{black} The main result of section is given by Theorem \ref{Thm_Kurtz_adapte}.}  In {\color{black} Proposition \ref{Prop_Caract_Gamma_Limite} of} Section \ref{Section_6_Caract_Gamma_Limite}, we identify and characterise in a unique way the limiting distribution of the fast component on the torus. {\color{black} Section \ref{Sous_Section_6_1_Proof_Main_result} is devoted to proving this result by exploiting the key Lemma \ref{Lem_proba_invariante} whose proof is given in Section \ref{Sous_Section_6_3_Proba_Invariante_Dawson}.} In {\color{black} Lemma \ref{Lem_Unicite_ODE} of} Section \ref{Sous_section_7_1}, we proceed similarly for the slow component. Thanks to  the uniqueness on the torus of the limit slow component, we deduce in Section \ref{Sous_section_7_2} the announced result of Theorem \ref{Thm_CEAD_Jouet}.
 
 \subsubsection{Prospects\label{Sous-sous-section_Perspectives}}

In the first part of this section, we explain the origin of exponent $2$ in the inequality ${\color{black}\sigma_{K}} \ll K^{-\left(2 + \varepsilon \right)}$ of Assumption (\ref{Hypothese_Gamme_sigma}). In the second part, it is explained how Assumption (\ref{Hypothese_Gamme_sigma}) may be improved into the assumption ${\color{black}\sigma_{K}} \ll K^{-\left(\frac{3}{2} + \varepsilon \right)}$ by using estimates for moments of order $2\ell.$ Finally, in the last part of this section, we explain the difficulties in obtaining these moment estimates. Note that the assumption ${\color{black}\sigma_{K}} \ll K^{-\frac{3}{2}}$ is the best we can expect because of the error term ${\color{black}\sigma_{K}} K^{\frac{3}{2}}$ in~(\ref{Decomposition_L_Phi_L_FAST}). \\

\textit{(a) Origin of the exponent $2$ in Assumption (\ref{Hypothese_Gamme_sigma}).} On the one hand, as for all $t \geqslant 0$ and $K \in \N^{\star}$, $\Diam\left( \Supp \mu_{t}^{K} \right)^{2} \leqslant \left(2 \max\left\{\left|x \right| \left| \phantom{1^{1^{1}}} \hspace{-0.6cm} \right. x \in \Supp \mu_{t}^{K} \right\} \right)^{2} \leqslant 4K M_{2}\left(\mu_{t}^{K}\right)$, we deduce that 
\begin{equation*}
\begin{aligned}
\P\left(\check{\tau}^{K} < \widehat{\tau}^{K} \wedge T \right) & \leqslant \P\left(\exists t < T\wedge \widehat{\tau}^{K}, \Diam\left(\Supp \mu_{t}^{K} \right) > \frac{1}{{\color{black}\sigma_{K}} K^{\frac{3+\varepsilon}{2}}}  \right) \\
& \leqslant  \P\left(\exists t < T\wedge \widehat{\tau}^{K},   M_{2}\left(\mu_{t}^{K} \right) > \frac{1}{4{\color{black}\sigma_{K}^{2}}K^{4+\varepsilon}} \right)
\end{aligned}
\label{Origine_exposant_2}
\end{equation*}
On the other hand, note that thanks to the definition (\ref{Temps_arret_chapeau}) of the stopping time $\widehat{\tau}^{K}$, \[\P\left(\exists t <T\wedge \widehat{\tau}^{K}, M_{2}\left(\mu_{t}^{K} \right)>K^{\varepsilon} \right) = 0.\]
Hence, Assumption (\ref{Hypothese_Gamme_sigma}) implies that $\P\left(\check{\tau}^{K} < \widehat{\tau}^{K} \wedge T \right) = 0$ for $K$ large enough. Using this relation, we prove in Section \ref{Sous_Section_4_4_Convergence_tau} that $\P\left(\tau^{K} < T \right) \leqslant \P\left(\widehat{\tau}^{K}  \leqslant T \wedge \check{\tau}^{K} \right) $. Hence, to establish that $\tau^{K} \to + \infty$ when $K \to + \infty$ in probability, we need to prove that \[\lim_{K\to + \infty}{\P\left(\widehat{\tau}^{K} \leqslant T \wedge \check{\tau}^{K} \right)} =  \lim_{K \to +\infty}{\P\left(\sup_{0\leqslant t \leqslant T \wedge \check{\tau}^{K}}{M_{2}\left(\mu_{t}^{K} \right) \geqslant K^{\varepsilon}} \right)} = 0\] as established in Section \ref{Sous_section_4_5_Sortie_domaine}. \\

\textit{(b) Improvement of Assumption (\ref{Hypothese_Gamme_sigma}) with moments of order $2\ell$.} Assume now that ${\color{black}\sigma_{K}} \ll K^{-\left(\frac{3}{2} + \varepsilon \right)} $. Then, there exists $\ell \in \N^{\star}$, $\widetilde{\varepsilon}>0$ such that 
\begin{equation}
%\forall \ell \in \N^{\star}, \ \exists \, \widetilde{\varepsilon}>0, \ 
\forall C >0, \qquad K^{-C\log(K)} \ll {\color{black}\sigma_{K}} \ll K^{-\left(\frac{3}{2} + \frac{1}{2\ell} + \frac{\ell + 1}{2\ell}\widetilde{\varepsilon} \right)} 
\label{Hypothese_Gamme_sigma_BIS}
\end{equation}
and we consider the stopping times 
\begin{align*}
\widehat{\tau}_{\rm bis}^{K} & := \inf\left\{t\geqslant 0 \left| \phantom{1^{1^{1^{1}}}} \hspace{-0.7cm} \right.   M_{2}\left(\mu^{K }_{t}\right) \geqslant K^{\widetilde{\varepsilon}}  \right\} 
\end{align*}
and 
\begin{align*}
\check{\tau}_{\rm bis}^{K } & := \inf\left\{t\geqslant 0 \left| \phantom{1^{1^{1^{1}}}} \hspace{-0.7cm} \right.  \Diam\left(\Supp{\mu^{K}_{t}}\right) > \frac{1}{{\color{black}\sigma_{K}} K^{\frac{3 + \widetilde{\varepsilon}}{2}}}\right\} .
\end{align*}
 Using now that for all $t \geqslant 0$, for all $K \in \N^{\star}$, for all $\ell \in \N^{\star}$, $\Diam\left( \Supp \mu_{t}^{K} \right)^{2\ell} \leqslant 2^{2\ell}  K M_{2 \ell}\left(\mu_{t}^{K}\right)$, we deduce that 
\begin{align*}
& \P\left(\exists t < T \wedge \widehat{\tau}_{\rm bis}^{K},  \Diam\left(\Supp \mu_{t}^{K} \right) > \frac{1}{{\color{black}\sigma_{K}} K^{\frac{3+\widetilde{\varepsilon}}{2}}} \right) \\ 
& \hspace{5cm} \leqslant \P\left(\exists t < T \wedge \widehat{\tau}_{\rm bis}^{K}, M_{2\ell}\left(\mu_{t}^{K} \right) > \frac{1}{2^{2\ell}{\color{black}\sigma_{K}^{2\ell}}K^{3\ell + 1 + \ell\widetilde{\varepsilon}}}  \right),
\end{align*}
which is zero for $K$ large enough by (\ref{Hypothese_Gamme_sigma_BIS}). As previously, if we can prove for all $T\geqslant 0$ that 
\begin{equation}
\lim_{K\to + \infty}{\P\left(\widehat{\tau}_{\rm bis}^{K} \leqslant T \wedge \check{\tau}_{\rm bis}^{K} \right)} =  \lim_{K\to + \infty}{\P\left(\sup_{0\leqslant t \leqslant T\wedge \check{\tau}_{\rm bis}^{K}}{M_{2\ell}\left( \mu_{t}^{K}\right)} \geqslant K^{\widetilde{\varepsilon}}\right)} = 0,
\label{Perspective_Sup_M2ell}
\end{equation} 
then, we can conclude as before. \\

\textit{(c) Difficulties encountered.} The main difficulty consists in proving (\ref{Perspective_Sup_M2ell}). For this we could seek for an extension of Lemma \ref{Lem_Controle_M6_et_M4M2} to the framework of moments of order $2\ell$. Lemma \ref{Lem_Controle_M6_et_M4M2} relies on the inequality for all $t \geqslant 0$, $K$ large enough,  
\begin{align*}
& \frac{3}{4}M_{6}\left(\mu_{t\wedge \check{\tau}^{K}} \right) + 3 M_{4}\left(\mu_{t\wedge \check{\tau}^{K}}^{K} \right)M_{2}\left(\mu_{t\wedge \check{\tau}^{K}}^{K} \right) + \frac{3}{2}M_{2}^{3}\left(\mu_{t\wedge \check{\tau}^{K}}^{K} \right) \\
&  \leqslant \frac{3}{4}M_{6}\left(\mu_{0}^{K} \right) + 7 M_{4}\left(\mu_{0}^{K} \right)M_{2}\left(\mu_{0}^{K} \right) + \frac{3}{2}M_{2}^{3}\left(\mu_{t\wedge \check{\tau}^{K}} \right)  \\
&  - \frac{C_{1}}{K^{2}{\color{black}\sigma_{K}^{2}}}\int_{0}^{t\wedge \check{\tau}^{K}}{\left(\frac{3}{4}M_{6}\left(\mu_{t\wedge \check{\tau}^{K}} \right) + 3 M_{4}\left(\mu_{s}^{K} \right)M_{2}\left(\mu_{s}^{K} \right) + \frac{3}{2}M_{2}^{3}\left(\mu_{s}^{K} \right) - C_{2} \right)\dd s} + {\rm Mart}_{t\wedge\check{\tau}^{K}}
\end{align*}
for some constants $C_{1}, C_{2} >0$ and where $\left({\rm Mart}_{t\wedge\check{\tau}^{K}} \right)_{t\geqslant 0}$ is a martingale. However a similar calculation does not seem to be successful 
for moments of order 8 or more.

\section{Infinitesimal generator approximation\label{Section_3_Gene_LENT_RAPIDE}}

Let us begin this section by giving the generator of the couple of process $\left(z^{K}, \mu^{K} \right)$. Then, we give an approximation of the previous formula for a class of test functions which is convergence determining. Finally, we give an extension to this result in the case of polynomials in $\mu$.

\subsection{Generators and martingales}

For all $K \in \N^{\star}$, let us introduce the filtration $\left(\FF_{t}^{K} \right)_{t \geqslant 0}$ defined by $\FF_{t}^{K} := \sigma\left(z_{s}^{K}, \mu_{s}^{K}\left| \phantom{1^{1^{1}}} \hspace{-0.55cm} \right.  s\leqslant t \right)$.

\begin{Prop}
The infinitesimal generator of the $\R \, {\color{black}\times} \, \MM_{1, K}^{c}(\R)-$valued \textsc{Markov} process $\left(z^{K}, \mu^{K} \right)$, defined for all bounded measurable function $\Phi$ from $\R \times\MM_{1, K}^{c}(\R)$ to $\R$, is given by
\begin{align*}
\LL^{K }\Phi\left(z, \mu\right) & = \frac{K}{K{\color{black}\sigma_{K}^{2}}} \int_{\R}^{}{\mu(\dd x)\int_{\R}^{}{\mu(\dd y)b\left({\color{black}\sigma_{K}} \sqrt{K} x + z, {\color{black}\sigma_{K}}\sqrt{K} y + z \right)}} \\
& \hspace{1.5cm} \times \left[\Phi\left(z + \frac{{\color{black}\sigma_{K}} \sqrt{K}}{K}(y-x), \tau_{-\frac{y-x}{K}} \sharp \, \left[\mu - \frac{\delta_{x}}{K} + \frac{\delta_{y}}{K} \right] \right) - \Phi\left(z, \mu\right) \right] \\
& \quad + \frac{K}{K{\color{black}\sigma_{K}^{2}}} \int_{\R}^{}{\mu(\dd x)\theta\left({\color{black}\sigma_{K}} \sqrt{K} x + z \right) \int_{\R}^{}{ m\left({\color{black}\sigma_{K}} \sqrt{K} x + z, h \right)}} \\
& \hspace{1.5cm} \times \left[\Phi\left(z + \frac{{\color{black}\sigma_{K}} \sqrt{K}h}{K^{\frac{3}{2}}}, \tau_{-\frac{h}{K^{3/2}}} \sharp \, \left[\mu - \frac{\delta_{x}}{K} + \frac{\delta_{x + \frac{h}{\sqrt{K}}}}{K} \right] \right) - \Phi\left(z, \mu\right) \right]\dd h.
\end{align*}
Moreover, for all $K \in \N^{\star}$, for all bounded measurable function $\Phi$ from $\R \times \MM_{1, K}^{c}(\R)$ to $\R$, the following process $\left(M_{t}^{K, \Phi}\right)_{t\geqslant 0}$ defined by %\wedge \tau^{K}
\begin{equation}
M_{t}^{K, \Phi} := \Phi\left(z_{t}^{K}, \mu_{t}^{K} \right) -\Phi\left(z_{0}^{K}, \mu_{0}^{K} \right) - \int_{0}^{t}{\LL^{K}\Phi\left(z_{s}^{K}, \mu_{s}^{K} \right)\dd s} %+ \int_{0}^{t\wedge\tau^{K}}{O\left(\frac{M_{2}\left(\mu_{s}^{K } \right)}{K} + \frac{1}{K^{2}} \right) \dd s}
\label{PB_Mg_LENT_Approx}
\end{equation}
is a $\left(\FF_{t}^{K } \right)_{t\geqslant 0}-$martingale, square integrable, with quadratic variation:
\begin{align*}
\left\langle M_{t}^{K, \Phi} \right\rangle_{t}& = \frac{K}{K{\color{black}\sigma_{K}^{2}}} \int_{0}^{t}{\dd s\int_{\R}^{}{{\color{black}\mu_{s}^{K}}(\dd x)\int_{\R}^{}{{\color{black}\mu_{s}^{K}}(\dd y)b\left({\color{black}\sigma_{K}} \sqrt{K} x + {\color{black}z_{s}^{K}}, {\color{black}\sigma_{K}}\sqrt{K} y + {\color{black}z_{s}^{K}} \right)}}} \\
& \hspace{1.5cm} \times \left[\Phi\left({\color{black}z_{s}^{K}} + \frac{{\color{black}\sigma_{K}} \sqrt{K}}{K}(y-x), \tau_{-\frac{y-x}{K}} \sharp \, \left[{\color{black}\mu_{s}^{K}} - \frac{\delta_{x}}{K} + \frac{\delta_{y}}{K} \right] \right) - \Phi\left({\color{black}z_{s}^{K}}, {\color{black}\mu_{s}^{K}}\right) \right]^{2} \\
& \quad + \frac{K}{K{\color{black}\sigma_{K}^{2}}} \int_{0}^{t}{\dd s\int_{\R}^{}{{\color{black}\mu_{s}^{K}}(\dd x)\theta\left({\color{black}\sigma_{K}} \sqrt{K} x + {\color{black}z_{s}^{K}} \right) \int_{\R}^{}{ m\left({\color{black}\sigma_{K}} \sqrt{K} x + {\color{black}z_{s}^{K}}, h \right)}}} \\
& \hspace{1.5cm} \times \left[\Phi\left({\color{black}\mu_{s}^{K}} + \frac{{\color{black}\sigma_{K}} \sqrt{K}h}{K^{\frac{3}{2}}}, \tau_{-\frac{h}{K^{3/2}}} \sharp \, \left[{\color{black}\mu_{s}^{K}} - \frac{\delta_{x}}{K} + \frac{\delta_{x + \frac{h}{\sqrt{K}}}}{K} \right] \right) - \Phi\left({\color{black}\mu_{s}^{K}}, {\color{black}\mu_{s}^{K}}\right) \right]^{2}\dd h.
\end{align*}
\label{Prop_Generateur_Couple_Lent_Rapide}
\end{Prop}

Note that the factor $1/K{\color{black}\sigma_{K}^{2}}$ corresponds to the time scaling of ${\color{black}\nu_{t}^{K}}$ used to define ${\color{black}z_{t}^{K}}$ and ${\color{black}\mu_{t}^{K}}$. 

\begin{proof}  \textit{Step 1. About the generators.} On the one hand, from (\ref{Generateur_initial}), note that for all $\Phi : \R \times \MM_{1, K}^{c}(\R, \R) \to \R$ and $\phi : \MM_{1, K}(\R) \to \R$ such that $\phi(\nu) = \Phi\left(\left\langle \id, \nu \right\rangle, \left(h_{\frac{1}{{\color{black}\sigma_{K}}\sqrt{K}}} \circ \tau_{-\left\langle \id, \nu \right\rangle} \right)  \sharp \, \nu \right) $, 
\begin{equation}
\LL^{K}\Phi\left(z, \mu \right) = \frac{1}{K{\color{black}\sigma_{K}^{2}}} \times L^{K}\phi\left(\left(\tau_{z} \circ h_{{\color{black}\sigma_{K}}\sqrt{K}}\right) \sharp \ \mu \right).
\label{Eq_passage_generateurs_L_ronf_L_droit}
\end{equation}
On the other hand, note that for all $\nu \in \MM_{1,K}(\R)$, $u, v \in \R$,  $\psi : \R \to \R$ and $\Psi : \MM_{1, K}(\R) \to \R$ measurable bounded functions, we have
\begin{equation}
\begin{aligned}
& \int_{\R}^{}{\psi(x)\Psi\left(\nu + \frac{\delta_{x + {\color{black}\sigma_{K}} \sqrt{K} u}}{K} - \frac{\delta_{v}}{K} \right) \nu(\dd x)} \\ 
& \hspace{1cm} = \int_{\R}^{}{\psi\left({\color{black}\sigma_{K}} \sqrt{K} y + \left\langle \id, \nu \right\rangle \right)}  \Psi\left(\left(\tau_{\left\langle \id, \nu \right\rangle} \circ h_{{\color{black}\sigma_{K}}\sqrt{K}} \right) \sharp \, \left[\mu + \frac{\delta_{y + u}}{K} - \frac{\delta_{v}}{K}\right]\right)\mu(\dd y).
\end{aligned} 
\label{Eq_transformation_mu_nu}
\end{equation}
Finally, by noting that for all $u \in \R$ the mean trait of $\left(\tau_{\left\langle \id, \nu \right\rangle} \circ h_{{\color{black}\sigma_{K}} \sqrt{K}} \right) \sharp \, \left[\mu - \frac{\delta_{x}}{K} + \frac{\delta_{u}}{K} \right]$ is $$\left\langle \id, \nu \right\rangle + \frac{{\color{black}\sigma_{K}} \sqrt{K}}{K}(u-x)$$ and the centered and dilated distribution of traits of $\left(\tau_{\left\langle \id, \nu \right\rangle} \circ h_{{\color{black}\sigma_{K}} \sqrt{K}} \right) \sharp \, \left[\mu - \frac{\delta_{x}}{K} + \frac{\delta_{u}}{K} \right]$ is $$\tau_{-\frac{u-x}{K}} \sharp \, \left[ \mu - \frac{\delta_{x}}{K} + \frac{\delta_{u}}{K} \right],$$ the announced result follows from (\ref{Eq_passage_generateurs_L_ronf_L_droit}) and (\ref{Eq_transformation_mu_nu}). \\

\textit{Step 2. About the martingales.} The martingale property follows from classical arguments since $\Phi$ and $\LL^{K}\Phi$ are bounded {\color{blue} \cite[\color{black} Chapter 4]{Ethier_markov_1986}}. To compute the bracket,  we proceed according to the following classical method (see e.g. {\color{blue} \cite{Fournier}}). We apply (\ref{PB_Mg_LENT_Approx}) replacing $\Phi$ by $\Phi^{2}$ to obtain the martingale $M^{K, \Phi^{2}}$. Then, we apply the \textsc{Itô} formula to compute $\Phi^{2}\left(z_{t}^{K}, \mu_{t}^{K}  \right)$ from (\ref{PB_Mg_LENT_Approx}). We deduce that 
\begin{align*}
{\rm Mart}_{t} & := \Phi^{2}\left(z_{t}^{K}, \mu_{t}^{K} \right) - \Phi^{2}\left(z_{0}^{K}, \mu_{0}^{K} \right) - 2\int_{0}^{t}{{\color{black}\Phi\left(z_{s}^{K}, \mu_{s}^{K} \right)}\LL^{K}\Phi\left(z_{s}^{K}, \mu_{s}^{K} \right)\dd s} - \left\langle M^{K, \Phi}\right\rangle_{t}
\end{align*}
is a martingale.  Hence, the martingale $M_{t}^{K, \Phi^{2}} - {\rm Mart}_{t}$ has finite variation, so it is null {\color{blue} \cite[\color{black} Theorem 4.1]{LeGall}}, leading to 
\[\left\langle M^{K, \Phi}\right\rangle_{t} = \int_{0}^{t}{\left[\LL^{K}\Phi^{2}\left(z_{s}^{K}, \mu_{s}^{K} \right) - 2\Phi\left(z_{s}^{K}, \mu_{s}^{K} \right)\LL^{K}\Phi\left(z_{s}^{K}, \mu_{s}^{K} \right) \right]\dd s} \]
then the announced conclusion. \qedhere
\end{proof}

\subsection{Asymptotic expansions} 
Recall that, a set $S \subset \CCCC^{0}_{b}\left(\R, \R \right)$ is called $\MM_{1, K}^{c}\left(\R \right)-$ \emph{convergence determining} if whenever $\left(P_{n} \right)_{n \in \N} \in  \MM_{1,K}^{c}\left(\R \right)^{\N}, P \in \MM_{1,K}^{c}(\R)$ and  $$\lim_{n\to + \infty}{\left\langle f, P_{n} \right\rangle} = \left\langle f, P \right\rangle$$ for all $f \in S$, we have that $\left(P_{n} \right)_{n\in \N}$ converges weakly to $P$ {\color{blue} \cite[\color{black} Chapter 3, Section 4, p.112]{Ethier_markov_1986}}. As developed in {\color{blue} \cite[\color{black} Theorem 3.2.6]{Dawson}}, the class  of functions on $\MM_{1,K}^{c}(\R)$, %$\F_{b}^{2}$ 
\[\F_{b}^{2} := \left\{F_{\varphi} \left| \phantom{1^{1^{1^{1}}}} \hspace{-0.7cm} \right. F_{\varphi}\left(\mu\right) := F\left( \left\langle \varphi, \mu \right\rangle\right), F \in \CCCC^{2}\left(\R, \R\right), \varphi \in \CCCC^{2}_{b}(\R, \R) \right\}\] 
is $\MM_{1, K}^{c}\left(\R \right)-$convergence determining. \\

In the following sections, we prove the assertions (i) and (ii) of Section \ref{Sous-sous-section_2_3_1_Slow-fast_Heuristic}.

\subsubsection{Slow component} 
\begin{Prop}
For all $f \in \CCCC_{b}^{2}(\R, \R)$, understood as a function of $(z, \mu)$ which depends only on $z$, the infinitesimal generator $\LL^{K}$ defined in {\rm Proposition \ref{Prop_Generateur_Couple_Lent_Rapide}}, satisfies the following decomposition 
\[\LL^{K}f\left(z, \mu \right) =  \LL_{{\rm SLOW}}f\left(z,\mu\right) + O\left(\frac{1}{K^{2}} + \frac{M_{2}(\mu)}{K}  + {\color{black}\sigma_{K}}\sqrt{K}M_{3}\left(\mu \right) \right)\] where the operator $\LL_{\rm SLOW}$ is given by {\rm (\ref{Generateur_LENT})}.
\label{Prop_Generateur_Lent_Decomposition}
\end{Prop}

\begin{proof}
 From Proposition \ref{Prop_Generateur_Couple_Lent_Rapide} for the choice of test functions $\Phi(z, \mu) := f(z)$ we obtain that $\LL^{K}f(z, \mu) = \textbf{\rm{\textbf{(A)}}}^{K} + \textbf{\rm{\textbf{(B)}}}^{K}$ where
\begin{align*}
\textbf{\rm{\textbf{(A)}}}^{K} & := \frac{K}{K{\color{black}\sigma_{K}^{2}}}\int_{\R}^{}{\mu(\dd x)\int_{\R}^{}{\mu(\dd y)b\left({\color{black}\sigma_{K}}\sqrt{K} x + z, {\color{black}\sigma_{K}}\sqrt{K} y +z \right)}}, \\
& \hspace{6cm} \times \left[f\left(z + \frac{{\color{black}\sigma_{K}} \sqrt{K}}{K}(y-x) \right) -f(z) \right] \\
\textbf{\rm{\textbf{(B)}}}^{K} & :=  \frac{K}{K{\color{black}\sigma_{K}^{2}}} \int_{\R}^{}{\mu(\dd x)\theta\left({\color{black}\sigma_{K}}\sqrt{K}x + z \right)\int_{\R}^{}{m\left({\color{black}\sigma_{K}} \sqrt{K}x + z, h \right)}} \\ 
& \hspace{6cm} \times \left[f\left(z + \frac{{\color{black}\sigma_{K}} \sqrt{K}}{K^{\frac{3}{2}}}h \right) - f(z) \right]\dd h.
\end{align*}

Now, we want to decompose and study $\textbf{\rm{\textbf{(A)}}}^{K}$. Denoting $\textbf{\rm{\textbf{(C)}}}^{K}_{x,y}(z) := f\left(z + \frac{{\color{black}\sigma_{K}} \sqrt{K}}{K}(y-x) \right)$ $-f(z)$, from \textsc{Taylor}'s formula, note that $\textbf{\rm{\textbf{(A)}}}^{K} = \textbf{\rm{\textbf{(A)}}}^{K}_{1} + \textbf{\rm{\textbf{(A)}}}^{K}_{2} + \textbf{\rm{\textbf{(A)}}}^{K}_{3}$ where 
\begin{align*}
\textbf{\rm{\textbf{(A)}}}^{K}_{1} & := \frac{K}{K{\color{black}\sigma_{K}^{2}}}\int_{\R}^{}{\mu(\dd x)\int_{\R}^{}{\mu(\dd y)b\left(z, z \right)\textbf{\rm{\textbf{(C)}}}^{K}_{x,y}(z)}}, \\
\textbf{\rm{\textbf{(A)}}}^{K}_{2} & := \frac{K}{K{\color{black}\sigma_{K}^{2}}} \times {\color{black}\sigma_{K}}\sqrt{K}\int_{\R}^{}{\mu(\dd x)\int_{\R}^{}{\mu(\dd y)\left(x\partial_{1}b\left(z, z \right) +y\partial_{2}b(z,z) \right)\textbf{\rm{\textbf{(C)}}}^{K}_{x,y}(z)}}, \\
\textbf{\rm{\textbf{(A)}}}^{K}_{3} & := \frac{K}{K{\color{black}\sigma_{K}^{2}}}  \int_{\R}^{}{\mu(\dd x)\int_{\R}^{}{\mu(\dd y)\left[b\left({\color{black}\sigma_{K}}\sqrt{K} x + z, {\color{black}\sigma_{K}}\sqrt{K} y +z \right) - b(z,z) \right.}}, \\
& \hspace{6cm} -\left. {\color{black}\sigma_{K}}\sqrt{K}\left(x\partial_{1}b\left(z, z \right) +y\partial_{2}b(z,z) \right) \right]\textbf{\rm{\textbf{(C)}}}^{K}_{x,y}(z).
\end{align*}
Using that $\textbf{\rm{\textbf{(C)}}}^{K}_{x,y}(z) = f'(z)\frac{{\color{black}\sigma_{K}}\sqrt{K}}{K}(x-y)  + O\left(\frac{{\color{black}\sigma_{K}^{2}}}{K}\left|y-x \right|^{2} \right)$, we deduce that
\[\textbf{\rm{\textbf{(A)}}}^{K}_{1}  = O\left(\frac{M_{2}\left(\mu \right)}{K} \right), \ \
\textbf{\rm{\textbf{(A)}}}^{K}_{2}  = f'(z) \left(\partial_{2}b(z,z) - \partial_{1}b(z,z) \right)M_{2}\left(\mu \right) + O\left(\frac{{\color{black}\sigma_{K}}\sqrt{K} M_{3}\left(\mu \right)}{K} \right), \]
%\begin{align*}
%\textbf{\rm{\textbf{(A)}}}^{K}_{1} & = O\left(\frac{M_{2}\left(\mu \right)}{K} \right), \\
%\textbf{\rm{\textbf{(A)}}}^{K}_{2} & = f'(z) \left(\partial_{2}b(z,z) - \partial_{1}b(z,z) \right)M_{2}\left(\mu \right) + O\left(\frac{\sigma\sqrt{K} M_{3}\left(\mu \right)}{K} \right),
%\end{align*}
and noting that $\left|\textbf{\rm{\textbf{(C)}}}^{K}_{x,y} \right| \leqslant C_{1}\frac{{\color{black}\sigma_{K}}\sqrt{K}}{K}\left|y-x \right|$ for some constant $C_{1} >0$ and 
\begin{multline*}
 \left|b\left({\color{black}\sigma_{K}}\sqrt{K} x + z, {\color{black}\sigma_{K}}\sqrt{K} y +z \right) - b(z,z) - {\color{black}\sigma_{K}}\sqrt{K}\left(x\partial_{1}b\left(z, z \right) +y\partial_{2}b(z,z) \right)\right| \\ 
\leqslant C_{2}{\color{black}\sigma_{K}^{2}}K\left(x^{2} + y^{2} \right)
\end{multline*}
 for some constant $C_{2} >0$, we deduce that $\textbf{\rm{\textbf{(A)}}}^{K}_{3} = O\left({\color{black}\sigma_{K}}\sqrt{K}M_{3}\left(\mu \right) \right)$. Therefore,
\[\textbf{\rm{\textbf{(A)}}}^{K} = f'(z) \left(\partial_{2}b(z,z) - \partial_{1}b(z,z) \right)M_{2}\left(\mu \right) + O\left(\frac{M_{2}\left(\mu \right)}{K} + {\color{black}\sigma_{K}}\sqrt{K} M_{3}\left(\mu \right)\right).\] 
As previously, we obtain that
\[\textbf{\rm{\textbf{(B)}}}^{K} = O\left(\frac{\overline{\theta}\overline{m}_{2}}{K^{2}} \right).\] 
and the announced result follows. 
\qedhere
\end{proof}

\subsubsection{Fast component}
\begin{Prop}
For all $F \in \CCCC^{3}(\R, \R)$ and $\varphi \in \CCCC_{b}^{3}(\R, \R)$, the infinitesimal generator $\LL^{K}$ defined in {\rm Proposition \ref{Prop_Generateur_Couple_Lent_Rapide}}, satisfies the following decomposition 
\[ \LL^{K}F_{\varphi}\left(z, \mu \right) = \frac{\theta(z)m_{2}(z)}{K^{2}{\color{black}\sigma_{K}^{2}}} \left[\LL_{\rm FVc}^{\lambda(z)}F_{\varphi}(z, \mu) + O\left(\frac{1}{\sqrt{K}} + {\color{black}\sigma_{K}} K^{\frac{3}{2}}M_{2}\left(\mu \right) + \frac{M_{3}(\mu)}{K}\right) \right] \]
where the operator $\LL_{{\rm FVc}}^{\lambda(z)}$ is given by {\rm(\ref{Generateur_FVc})} and $F_{\varphi}$ is understood as a function of $\left(z, \mu\right)$ which depends only on $\mu$.
\label{Prop_Generateur_Rapide_Decomposition}
\end{Prop}

\begin{proof} Noting that \begin{align*}
\textbf{\rm{\textbf{(D)}}}^{K}_{x,y} & := \left\langle \varphi, \tau_{-\frac{y-x}{K}} \sharp \, \left[\mu - \frac{\delta_{x}}{K} + \frac{\delta_{y}}{K} \right] \right\rangle - \left\langle \varphi, \mu \right\rangle \\
& \phantom{:}= \left\langle \varphi \circ \tau_{-\frac{y-x}{K}}, \mu \right\rangle + \frac{1}{K}\left\{\varphi\left(y - \frac{y-x}{K} \right) - \varphi\left(x - \frac{y-x}{K} \right) \right\} - \left\langle \varphi, \mu \right\rangle,
\end{align*} 
setting $\textbf{\rm{\textbf{(C)}}}^{K}_{x,y} := F\left(\left\langle\varphi, \mu \right\rangle + \textbf{\rm{\textbf{(D)}}}^{K}_{x,y} \right) - F\left(\left\langle \varphi, \mu \right\rangle \right)$ and from Proposition \ref{Prop_Generateur_Couple_Lent_Rapide} we obtain that $\LL^{K}F_{\varphi}(z, \mu) = \textbf{\rm{\textbf{(A)}}}^{K} + \textbf{\rm{\textbf{(B)}}}^{K}$ where
\begin{align*}
\textbf{\rm{\textbf{(A)}}}^{K} & := \frac{K}{K{\color{black}\sigma_{K}^{2}}}\int_{\R}^{}{\mu(\dd x)\int_{\R}^{}{\mu(\dd y)b\left({\color{black}\sigma_{K}}\sqrt{K} x + z, {\color{black}\sigma_{K}}\sqrt{K} y +z \right)\textbf{\rm{\textbf{(C)}}}^{K}_{x,y}}}, \\
\textbf{\rm{\textbf{(B)}}}^{K} & :=  \frac{K}{K{\color{black}\sigma_{K}^{2}}} \int_{\R}^{}{\mu(\dd x)\theta\left({\color{black}\sigma_{K}}\sqrt{K}x + z \right)\int_{\R}^{}{m\left({\color{black}\sigma_{K}} \sqrt{K}x + z, h \right)\textbf{\rm{\textbf{(C)}}}^{K}_{x,x+h/\sqrt{K}}\,\dd h} }.
\end{align*}

\textit{Step 1. Decomposition and study of $\textbf{\rm{\textbf{(A)}}}^{K}$.} From \textsc{Taylor}'s formula, we obtain that $\textbf{\rm{\textbf{(A)}}}^{K} = \textbf{\rm{\textbf{(A)}}}^{K}_{1} + \textbf{\rm{\textbf{(A)}}}^{K}_{2}$ where
\begin{align*}
\textbf{\rm{\textbf{(A)}}}^{K}_{1} & := \frac{K}{K{\color{black}\sigma_{K}^{2}}}\int_{\R}^{}{\mu(\dd x)\int_{\R}^{}{\mu(\dd y)b\left(z, z \right)\textbf{\rm{\textbf{(C)}}}^{K}_{x,y}}}, \\
\textbf{\rm{\textbf{(A)}}}^{K}_{2} & := \frac{K}{K{\color{black}\sigma_{K}^{2}}}\int_{\R}^{}{\mu(\dd x)\int_{\R}^{}{\mu(\dd y)\left[b\left({\color{black}\sigma_{K}}\sqrt{K} x + z, {\color{black}\sigma_{K}}\sqrt{K} y +z \right) - b(z,z)\right]\textbf{\rm{\textbf{(C)}}}^{K}_{x,y}}}.
\end{align*}
Denoting $\textbf{\rm{\textbf{(E)}}}^{K}_{x,y} := \frac{1}{K}\left(\varphi(y) - \varphi(x) - (y-x)\left\langle \varphi', \mu \right\rangle \right)$ and from \textsc{Taylor}'s formula again, we obtain that $\textbf{\rm{\textbf{(A)}}}^{K}_{1} = \textbf{\rm{\textbf{(A)}}}^{K}_{11} + \textbf{\rm{\textbf{(A)}}}^{K}_{12} + \textbf{\rm{\textbf{(A)}}}^{K}_{13} + \textbf{\rm{\textbf{(A)}}}^{K}_{14}$ where 
\begin{align*}
\textbf{\rm{\textbf{(A)}}}^{K}_{11} & := \frac{K}{K{\color{black}\sigma_{K}^{2}}}\int_{\R}^{}{\mu(\dd x)\int_{\R}^{}{\mu(\dd y)b\left(z, z \right)F'\left(\left\langle \varphi, \mu \right\rangle \right)\textbf{\rm{\textbf{(D)}}}^{K}_{x,y}}}, \\
\textbf{\rm{\textbf{(A)}}}^{K}_{12} & := \frac{K}{2K{\color{black}\sigma_{K}^{2}}}\int_{\R}^{}{\mu(\dd x)\int_{\R}^{}{\mu(\dd y)b\left(z, z \right)F''\left(\left\langle \varphi, \mu \right\rangle \right)\left[\textbf{\rm{\textbf{(E)}}}^{K}_{x,y}\right]^{2}}}, \\
\textbf{\rm{\textbf{(A)}}}^{K}_{13} & := \frac{K}{2K{\color{black}\sigma_{K}^{2}}}\int_{\R}^{}{\mu(\dd x)\int_{\R}^{}{\mu(\dd y)b\left(z, z \right)F''\left(\left\langle \varphi, \mu \right\rangle \right)\left[\left[\textbf{\rm{\textbf{(D)}}}^{K}_{x,y}\right]^{2} - \left[\textbf{\rm{\textbf{(E)}}}^{K}_{x,y}\right]^{2}\right]}}, \\
\textbf{\rm{\textbf{(A)}}}^{K}_{14} & := \frac{K}{K{\color{black}\sigma_{K}^{2}}}\int_{\R}^{}{\mu(\dd x)\int_{\R}^{}{\mu(\dd y)b\left(z, z \right)}} \\
& \hspace{3cm} \times \left(\textbf{\rm{\textbf{(C)}}}^{K}_{x,y} - F'\left(\left\langle \varphi, \mu \right\rangle \right)\textbf{\rm{\textbf{(D)}}}^{K}_{x,y} - \frac{F''\left(\left\langle \varphi, \mu\right\rangle \right)}{2}\left[\textbf{\rm{\textbf{(D)}}}^{K}_{x,y} \right]^{2} \right)
\end{align*}

From \textsc{Taylor}'s formula again and the centered condition of $\mu$, we obtain that 
\begin{align*}
\textbf{\rm{\textbf{(A)}}}^{K}_{11} & = \frac{Kb\left(z, z \right)}{K{\color{black}\sigma_{K}^{2}}}\int_{\R}^{}{\mu(\dd x)\int_{\R}^{}{\mu(\dd y)F'\left(\left\langle \varphi, \mu \right\rangle \right)\left[\textbf{\rm{\textbf{(E)}}}^{K}_{x,y} - \frac{1}{K^{2}}(y-x)\left(\varphi'(y) - \varphi'(x) \right) \right.}} \\
& \hspace{5.5cm} + \left. \frac{(y-x)^{2}\left\langle \varphi'', \mu \right\rangle}{2K^{2}} + O\left(\frac{\left|y-x\right|^{2} + \left|y-x\right|^{3}}{K^{3}}\right) \right] \\
%& \hspace{3.75cm} +\left. \frac{1}{2K^{2}}(y-x)^{2}\left(\left\langle \varphi'', \mu \right\rangle - \frac{\varphi''(x)}{K} + \frac{\varphi''(y)}{K} \right) + O\left(\frac{\left|y-x\right|^{3}}{K^{3}}\right)\right] \\
& \hspace{-1cm} = \frac{b(z,z)}{K^{2}{\color{black}\sigma_{K}^{2}}} F'\left(\left\langle \varphi, \mu \right\rangle \right)\left[-2\left\langle \varphi' \times \id, \mu \right\rangle + \left\langle \varphi'', \mu \right\rangle M_{2}\left(\mu \right)  \right] + \frac{1}{K^{2}{\color{black}\sigma_{K}^{2}}} \times O\left(\frac{M_{2}\left(\mu \right) + M_{3}\left(\mu \right)}{K} \right)
\end{align*}
In a straightforward way, we obtain that 
\begin{align*}
\textbf{\rm{\textbf{(A)}}}^{K}_{12}  & = \frac{b(z,z)}{K^{2}{\color{black}\sigma_{K}^{2}}}F''\left(\left\langle \varphi, \mu \right\rangle \right)\left[ \left\langle \varphi^{2}, \mu \right\rangle - \left\langle \varphi, \mu \right\rangle^{2} + M_{2}\left(\mu \right)\left\langle \varphi', \mu\right\rangle^{2} - 2\left\langle \varphi \times \id, \mu\right\rangle \left\langle \varphi', \mu \right\rangle  \right]. 
\end{align*}
As $\varphi$ is bounded and so $F$ too, we deduce that there exists two constants $C_{1}, C_{2} >0$ such that
\begin{align*}
\left|\left[\textbf{\rm{\textbf{(D)}}}^{K}_{x,y}\right]^{2} - \left[\textbf{\rm{\textbf{(E)}}}^{K}_{x,y}\right]^{2}\right| & =\left|\textbf{\rm{\textbf{(D)}}}^{K}_{x,y} - \textbf{\rm{\textbf{(E)}}}^{K}_{x,y}\right|\left|\textbf{\rm{\textbf{(D)}}}^{K}_{x,y} + \textbf{\rm{\textbf{(E)}}}^{K}_{x,y}\right|\\
 & \leqslant  C_{1} \frac{|y-x| + (y-x)^{2}}{K^{2}} \left(\frac{1}{K} + \frac{\left|y-x \right|}{K} \right), \end{align*}
 \begin{align*}
 \left|\textbf{\rm{\textbf{(C)}}}^{K}_{x,y} - F'\left(\left\langle \varphi, \mu \right\rangle \right)\textbf{\rm{\textbf{(D)}}}^{K}_{x,y} - \frac{F''\left(\left\langle \varphi, \mu\right\rangle \right)}{2}\left[\textbf{\rm{\textbf{(D)}}}^{K}_{x,y} \right]^{2} \right| &\leqslant C_{2}\frac{1 + \left|y-x \right|^{3}}{K^{3}}.
\end{align*}
Hence, \[\textbf{\rm{\textbf{(A)}}}^{K}_{13} = \frac{1}{K^{2}{\color{black}\sigma_{K}^{2}}} \times O\left(\frac{1 + M_{3}(\mu)}{K} \right) \quad  \quad {\rm and} \quad \quad \textbf{\rm{\textbf{(A)}}}^{K}_{14} = \frac{1}{K^{2}{\color{black}\sigma_{K}^{2}}} \times O\left(\frac{1 + M_{3}(\mu)}{K} \right). \] Finally, as $\left|b\left({\color{black}\sigma_{K}}\sqrt{K}x + z, {\color{black}\sigma_{K}}\sqrt{K}y + z \right) - b(z,z) \right| \leqslant C_{3}{\color{black}\sigma_{K}}\sqrt{K}\left(\left| x\right| + \left| y\right|\right)$ and $\left| \textbf{\rm{\textbf{(C)}}}^{K}_{x,y} \right| \leqslant C_{4}\frac{1+\left|y-x\right|}{K}$ for some constants $C_{3}, C_{4} >0$, we deduce that \[\textbf{\rm{\textbf{(A)}}}^{K}_{2} = \frac{1}{K^{2}{\color{black}\sigma_{K}^{2}}} \times O\left({\color{black}\sigma_{K}} K^{\frac{3}{2}} \left[M_{1}(\mu) +  M_{2}(\mu) \right]\right).\]
Therefore, using \textsc{H\"{o}lder}'s inequalities to bound  $M_{1}(\mu) \leqslant \sqrt{M_{2}\left(\mu \right)} \leqslant 1+ M_{2}(\mu)$ and $M_{2}(\mu) \leqslant M_{3}^{2/3}(\mu) \leqslant 1 + M_{3}(\mu)$ 
we have that 
\begin{align*}
\textbf{\rm{\textbf{(A)}}}^{K} & = \frac{b(z,z)}{K^{2}{\color{black}\sigma_{K}^{2}}} F'\left(\left\langle \varphi, \mu \right\rangle \right)\left[-2\left\langle \varphi' \times \id, \mu \right\rangle + \left\langle \varphi'', \mu \right\rangle M_{2}\left(\mu \right)  \right] \\
& \hspace{0.35cm} +\frac{b(z,z)}{K^{2}{\color{black}\sigma_{K}^{2}}} F''\left(\left\langle \varphi, \mu \right\rangle \right)\left[ \left\langle \varphi^{2}, \mu \right\rangle - \left\langle \varphi, \mu \right\rangle^{2} + M_{2}\left(\mu \right)\left\langle \varphi', \mu\right\rangle^{2} - 2\left\langle \varphi \times \id, \mu\right\rangle \left\langle \varphi', \mu \right\rangle  \right] \\
& \hspace{0.35cm} + \frac{1}{K^{2}{\color{black}\sigma_{K}^{2}}} \times O\left(\frac{1}{K} + {\color{black}\sigma_{K}} K^{\frac{3}{2}}\left[1 + M_{2}\left(\mu \right)\right] +  \frac{M_{3}\left(\mu \right)}{K} \right).
\end{align*}

\textit{Step 2. Conclusion.} In similar way to Step 1, we obtain that 
\[\textbf{\rm{\textbf{(B)}}}^{K} = \frac{\theta(z)m_{2}(z)}{K^{2}{\color{black}\sigma_{K}^{2}}}F'\left(\left\langle \varphi, \mu \right\rangle \right) \left\langle \frac{\varphi''}{2}, \mu \right\rangle + \frac{1}{K^{2}{\color{black}\sigma_{K}^{2}}}\times O \left(\frac{1}{\sqrt{K}} + {\color{black}\sigma_{K}} K M_{1}(\mu) \right).\]
Since, by Assumption (\ref{Hypothese_Gamme_sigma}), ${\color{black}\sigma_{K}} K^{\frac{3}{2}} \ll \frac{1}{\sqrt{K}}$, the announced result follows.  \qedhere
\end{proof}

\subsection{Generators and martingales in the case of polynomials in $\mu$}

\subsubsection{For bounded test functions}

In this section, we extend in Lemma \ref{Lem_PB_Mg_Polynomes} the result of Proposition \ref{Prop_Generateur_Couple_Lent_Rapide} %by proposing a martingale problem 
to test functions of the form 
\begin{equation}
P_{f, n}\left(\mu \right) := \left\langle f, \mu^{n} \right\rangle := \int_{\R}^{}{\cdots \int_{\R}^{}{f\left(x_{1}, \cdots, x_{n} \right)\mu\left(\dd x_{1} \right)\cdots \mu\left(\dd x_{n} \right)} }
\label{Eq_fonctions_test_F_f_n}
\end{equation}
with $n \in \N^{\star}$, $\mu \in \MM_{1, K}^{c}(\R)$, $f \in \CCCC^{3}_{b}\left(\R^{n}, \R \right)$ and where $\mu^{n}$ is the $n-$fold product measure of $\mu$. In Lemmas \ref{Cor_moment_ordre_2} and \ref{Cor_moment_ordre_6} we extend this result to the case of specific unbounded test functions.  \\

For all $n\in \N^{\star}$, we denote by $\bm{1} \in \R^{n}$, the vector whose coordinates are all $1$ and by $\Delta$ the Laplacian operator on $\R^{n}$. Let us introduce, for all $n \in \N^{\star}$, $\lambda >0$ and $f \in \CCCC_{b}^{2}\left(\R^{n}, \R \right)$, the operator $B_{\lambda}^{(n)}$ defined by 
\begin{equation}
B_{\lambda}^{(n)}f(x) := \frac{1}{2}\Delta f(x) - 2\lambda \left(\nabla f(x)\cdot\bm{1}  \right)(x\cdot\bm{1}), \qquad x \in \R^{n}. 
\label{Operateur_B_n}
\end{equation}
Let us consider for all $n \in \N^{\star}$ for all $i, j \in \left\{1, \cdots, n \right\}$, 
\begin{itemize}
\item $\Phi_{i,j} : \CCCC^{2}_{b}(\R^{n}, \R) \longrightarrow \CCCC^{2}_{b}(\R^{n-1}, \R)$, with $i \neq j$, is the function obtained from $f$ by inserting the variable $x_{i}$ between $x_{j-1}$ and $x_{j}$ when $i<j$ and by inserting the variable $x_{i-1}$ between $x_{j-1}$ and $x_{j}$ when $i>j$: 
\begin{equation}
\begin{aligned}
 \Phi_{i,j}f\left(x_{1}, \cdots, x_{n-1}\right)  &  = f\left(x_{1}, \cdots,  x_{j-1}, x_{i}, x_{j}, x_{j+1}, \cdots, x_{n-1} \right) & \qquad i<j \\
 \Phi_{i,j}f\left(x_{1}, \cdots, x_{n-1}\right)  &  = f\left(x_{1}, \cdots,  x_{j-1}, x_{i-1}, x_{j}, x_{j+1}, \cdots, x_{n-1} \right) & \qquad i>j
\end{aligned}  \ .
 \label{Expression_Phi_ij}
\end{equation}
%\vspace{0.2cm}
\item $K_{i,j} : \CCCC^{2}_{b}(\R^{n}, \R) \longrightarrow \CCCC^{2}(\R^{n+1}, \R)$ is defined as 
\begin{equation}
 K_{i,j}f(x_{1}, \cdots, x_{n}, x_{n+1}) := \partial_{ij}^{2}f(x_{1}, \cdots, x_{n})x_{n+1}^{2}.  \label{Operateur_K_ij}
\end{equation}
\end{itemize}

We recall from {\color{blue} \cite[\color{black} Definition 2.8]{Champagnat_Hass_FVr_2022}} that the extended generator (in the sense of \textsc{Dynkin}, see (\ref{PB_Mg_FVc_Polynomes}) below) $\LL_{\rm{FVc}}^{\lambda}$ of the centered \textsc{Fleming-Viot} process with resampling rate $\lambda$ is defined for any $n \in \N^{\star}$, for any test functions $f \in \CCCC_{b}^{2}\left(\R^{n}, \R\right)$ by 
\begin{equation}
\LL_{\rm{FVc}}^{\lambda}P_{f, n}\left(\mu \right) = \left\langle B_{\lambda}^{(n)}f, \mu^{n} \right\rangle  + \lambda\sum_{\substack{i, j \, = \, 1 \\ i \, \neq \, j}}^{n}{\left[\left\langle \Phi_{i,j}f, \mu^{n-1} \right\rangle - \left\langle f, \mu^{n} \right\rangle \right]} + \lambda \sum_{i, j\, = \, 1}^{n}{\left\langle K_{i,j}f, \mu^{n+1} \right\rangle}.
\label{Def_Gene_FVc_polynome}
\end{equation}
From {\color{blue} \cite[\color{black} Definition 2.8]{Champagnat_Hass_FVr_2022}} and denoting 
\[\widetilde{\Omega} := \left\{X \in \CCCC^{0}\left(\left[0, + \infty \right), \MM_{1}^{c, 2}(\R) \right) \left| \phantom{1^{1^{1^{1}}}} \hspace{-0.6cm} \right. \forall T >0, \ \sup_{0\leqslant t \leqslant T}{M_{2}\left(X_{t} \right)} < \infty \right\},\]
 we recall that a probability measure $\P_{\mu}$ on $\widetilde{\Omega}$ is said to solve the \emph{centered \textsc{Fleming-Viot} martingale problem for polynomials} with initial condition $\mu \in \MM_{1}^{c, 2}(\R)$, if the canonical process $\left(X_{t}\right)_{t\geqslant 0}$ on $\widetilde{\Omega}$ satisfies $\P_{\mu}\left(X_{0} = \mu \right) = 1$ and, for all $n \in \N^{\star}$ and $f \in \CCCC_{b}^{2}\left(\R^{n}, \R\right)$, 
 \begin{equation}
 \widehat{M}_{t}^{P_{f,n}} :=  P_{f, n}\left(X_{t} \right) -  P_{f, n}\left(X_{0} \right) - \int_{0}^{t}{\LL_{\rm FVc}^{\lambda}P_{f, n}\left(X_{s} \right)\dd s}
 \label{PB_Mg_FVc_Polynomes}
 \end{equation}
is a $\P_{\mu}-$martingale. 
\begin{Lem}
The infinitesimal generator $\LL^{K }$ of the $\R \,  {\color{black}\times} \, \MM_{1, K}^{c}(\R)-$valued \textsc{Markov} process $\left(z^{K}, \mu^{K} \right)$ given by {\rm Proposition \ref{Prop_Generateur_Couple_Lent_Rapide}}, satisfies for all $n \in \N^{\star}$ and $f \in \CCCC_{b}^{3}\left(\R^{n}, \R \right)$, the following relations:
\begin{align*}
\LL^{K}P_{f, n}(z, \mu) & = \frac{\theta(z)m_{2}(z)}{K^{2}{\color{black}\sigma_{K}^{2}}}\LL_{\rm{FVc}}^{\lambda(z)}P_{f, n}\left(\mu \right) + \frac{1}{K^{2}{\color{black}\sigma_{K}^{2}}} \times O\left(\frac{1}{\sqrt{K}} + {\color{black}\sigma_{K}} K^{\frac{3}{2}} M_{2}\left(\mu\right) + \frac{M_{3}\left(\mu \right)}{K} \right) 
\end{align*}
where $\lambda(z) := \frac{b(z,z)}{\theta(z)m_{2}(z)}$. Moreover, for all $K, n \in \N^{\star}$ and test functions $f \in \CCCC_{b}^{3}\left(\R^{n}, \R \right)$, the process $\left(M_{t}^{K, P_{f, n}} \right)_{t\geqslant 0}$ defined by 
\begin{equation*}
M_{t}^{K, P_{f, n}} := P_{f, n}\left(\mu_{t}^{K} \right) - P_{f, n}\left(\mu_{0}^{K} \right) - \int_{0}^{t}{\LL^{K}P_{f, n}\left(z_{s}^{K}, \mu_{s}^{K} \right)\dd s }
\label{Eq_Mg_Polynome}
\end{equation*}
is a square integrale martingale started at $0$.
\label{Lem_PB_Mg_Polynomes}
\end{Lem}
The proof of this result is given in Section \ref{Sous_Section_6_2_PB_Mg_polynomes}.

\subsubsection{For some unbounded test functions}

The following results are particular extensions of Lemma \ref{Lem_PB_Mg_Polynomes} when the test functions are no longer bounded but the processes are stopped at time $\check{\tau}^{K}$ Then, we give in Lemmas \ref{Cor_moment_ordre_2} and \ref{Cor_moment_ordre_6} \textsc{Doob}'s semi-martingale decomposition of moments of order $2$ and those whose degree is $6$. 

\begin{Prop} For all $n \in \N^{\star}$, for all $f \in \CCCC^{3}\left(\R^{n}, \R\right) $, the process $\left(M_{t\wedge \check{\tau}^{K}}^{K, P_{f, n}} \right)_{t\geqslant 0}$ defined by 
\begin{equation*}
\begin{aligned}
M_{t\wedge \check{\tau}^{K}}^{K, P_{f, n}} &:= P_{f, n}\left(\mu_{t\wedge \check{\tau}^{K}}^{K} \right) - P_{f, n}\left(\mu_{0}^{K} \right) - \frac{K}{K{\color{black}\sigma_{K}^{2}}}\int_{0}^{t\wedge \check{\tau}^{K}}{\int_{\R}^{}{\mu(\dd x)\int_{\R}^{}{\mu(\dd y)}} } \\
& \hspace{0.5cm} \times b\left({\color{black}\sigma_{K}} \sqrt{K} x + z_{s}^{K}, {\color{black}\sigma_{K}} \sqrt{K} x + z_{s}^{K} \right)\left[P_{f \circ \tau_{-\frac{y-x}{K}}, n}\left(\mu - \frac{\delta_{x}}{K} + \frac{\delta_{y}}{K} \right) - P_{f, n}(\mu) \right]\dd s \\
& \hspace{1.5cm} + \frac{K}{K{\color{black}\sigma_{K}^{2}}}\int_{0}^{t\wedge \check{\tau}^{K}}{\int_{\R}^{}{\mu(\dd x)\theta\left({\color{black}\sigma_{K}} \sqrt{K} x + z_{s}^{K} \right)\int_{\R}^{}{m\left({\color{black}\sigma_{K}}\sqrt{K} x  + z_{s}^{K}, h \right) \dd h}}} \\ 
& \hspace{4.5cm} \times \left[P_{f \circ \tau_{-\frac{h}{K^{\frac{3}{2}}}}, n}\left(\mu - \frac{\delta_{x}}{K} + \frac{\delta_{x+\frac{h}{\sqrt{K}}}}{K} \right) - P_{f, n}(\mu) \right]\dd s
\end{aligned}
\label{Martingale_f_NON_bornee_n}
\end{equation*}
is a bounded martingale.
\label{Prop_Martingale_f_NON_bornee_n}
\end{Prop}
\begin{proof}
Thanks to the stopping time $\check{\tau}^{K}$, $\left(M_{t\wedge\check{\tau}^{K}}^{K, P_{f, n}} \right)_{t\geqslant 0}$ is bounded. Hence the martingality of this process follows from Lemma \ref{Lem_PB_Mg_Polynomes}. \qedhere
\end{proof}

\textit{(a) Moment of order $2$.} The following result will be useful in Section \ref{Section_4_Estimees_de_moments}. 

\begin{Lem} For all $K \in \N^{\star}$, the process $\left(M_{t\wedge \check{\tau}^{K}}^{K, P_{\id^{2}, 1}}\right)_{t\geqslant 0}$ defined in {\rm Proposition \ref{Martingale_f_NON_bornee_n}} satisfies for all $t \geqslant 0$ 
\begin{align*}
M_{t \wedge \check{\tau}^{K}}^{K, P_{\id^{2}, 1}} & = M_{2}\left(\mu^{K }_{t\wedge \check{\tau}^{K}} \right) - M_{2}\left(\mu_{0}^{K } \right) \\ 
& \hspace{1cm} + \frac{1}{K^{2}{\color{black}\sigma_{K}^{2}}}\int_{0}^{t\wedge \check{\tau}^{K}}{\left\{2 b\left(z_{s}^{K }, z_{s}^{K } \right) M_{2}\left(\mu_{s}^{K } \right) - \theta\left(z_{s}^{K} \right)m_{2}\left(z_{s}^{K} \right) \right\}\dd s } \\
& \hspace{1cm}   +  \frac{1}{K^{2}{\color{black}\sigma_{K}^{2}}} \int_{0}^{t\wedge \check{\tau}^{K}}{O\left(\frac{1}{K} + {\color{black}\sigma_{K}} \sqrt{K} M_{1}\left(\mu_{s}^{K} \right) + {\color{black}\sigma_{K}} K^{\frac{3}{2}}M_{3}\left(\mu_{s}^{K} \right) \right)\dd s} 
\end{align*} is a square integrable martingale with quadratic variation:
\begin{equation}
\begin{aligned}
\left\langle M^{K, P_{\id^{2}, 1}}\right\rangle_{t\wedge \check{\tau}^{K}} & = \frac{2}{K^{2}{\color{black}\sigma_{K}^{2}}}\int_{0}^{t\wedge \check{\tau}^{K}}{b\left(z_{s}^{K }, z_{s}^{K } \right)\left\{M_{4}\left(\mu_{s}^{K } \right) - M_{2}^{2}\left(\mu_{s}^{K } \right)  \right\} \dd s} \\ 
& \hspace{-2.5cm} + \frac{1}{K^{2}{\color{black}\sigma_{K}^{2}}}\int_{0}^{t\wedge \check{\tau}^{K}}{O\left(\frac{{\color{black}\sigma_{K}}\sqrt{K}}{K^{2}} + \frac{M_{2}\left(\mu_{s}^{K} \right)}{K} + \frac{{\color{black}\sigma_{K}}\sqrt{K}M_{3}\left(\mu_{s}^{K} \right)}{K} + {\color{black}\sigma_{K}}\sqrt{K}M_{5}\left(\mu_{s}^{K} \right)  \right)\dd s}.
\label{Eq_Variation_Quad_M2}
\end{aligned}
\end{equation}
\label{Cor_moment_ordre_2}
\end{Lem}

\begin{proof} \textit{Step 1. Approximation of the \textsc{Doob} decomposition.} Note that for all $x, y \in \R$, 
\[P_{\id^{2}\circ \tau_{-\frac{y-x}{K}}, 1}\left(\mu - \frac{\delta_{x}}{K} + \frac{\delta_{y}}{K} \right) - P_{\id^{2}, 1}(\mu) =  \frac{1}{K}\left(y^{2} - x^{2} \right) - \frac{1}{K^{2}}(y-x)^{2}.\]
From \textsc{Taylor}'s formula, it follows that
\begin{align*}
& \frac{K}{K{\color{black}\sigma_{K}^{2}}}\int_{0}^{t}{\dd s\int_{\R}^{}{\mu(\dd x)\int_{\R}^{}{\mu(\dd y)b\left({\color{black}\sigma_{K}} \sqrt{K} x + z_{s}^{K}, {\color{black}\sigma_{K}}\sqrt{K}y + z_{s}^{K} \right)} } } \\
& \hspace{4.55cm} \times \left[P_{\id^{2}\circ \tau_{-\frac{y-x}{K}}, 1}\left(\mu - \frac{\delta_{x}}{K} + \frac{\delta_{y}}{K} \right) - P_{\id^{2}, 1}(\mu) \right] \\
& \hspace{1cm} = \frac{1}{K^{2}{\color{black}\sigma_{K}^{2}}}\int_{0}^{t}{\left[-2b\left(z_{s}^{K}, z_{s}^{K} \right)M_{2}\left(\mu_{s}^{K} \right) + O\left({\color{black}\sigma_{K}} K^{\frac{3}{2}}M_{3}\left(\mu_{s}^{K} \right) \right) \right]\dd s} 
\end{align*}
and 
\begin{align*}
&\frac{K}{K{\color{black}\sigma_{K}^{2}}}\int_{0}^{t}{\dd s\int_{\R}^{}{\mu(\dd x)\int_{\R}^{}{ m\left({\color{black}\sigma_{K}} \sqrt{K} x + z_{s}^{K}, h \right)\theta\left({\color{black}\sigma_{K}}\sqrt{K}x + z_{s}^{K} \right)} } }\\ 
& \hspace{4.5cm} \times \left[P_{\id^{2}\circ \tau_{-\frac{h}{K^{3/2}}}, 1}\left(\mu - \frac{\delta_{x}}{K} + \frac{\delta_{x + \frac{h}{\sqrt{K}}}}{K} \right) - P_{\id^{2}, 1}(\mu) \right]\dd h \\
& \hspace{1cm} = \frac{1}{K^{2}{\color{black}\sigma_{K}^{2}}}\int_{0}^{t}{\left[\theta\left(z_{s}^{K} \right)m_{2}\left(z_{s}^{K} \right)  + O\left(\frac{1}{K} + {\color{black}\sigma_{K}}\sqrt{K}M_{1}\left(\mu_{s}^{K} \right)  \right) \right]\dd s} . %+\sigma KM_{2}\left(\mu_{s}^{K} \right)
\end{align*}
We deduce the first announced result. \\

\textit{Step 2. Quadratic variation.} In similar way to the proof of Proposition \ref{Prop_Generateur_Couple_Lent_Rapide}, we have 
\begin{align*}
& \left\langle M^{K, P_{\id^{2}, 1}}\right\rangle_{t} = \frac{K}{K{\color{black}\sigma_{K}^{2}}}\int_{0}^{t}{\dd s \int_{\R}^{}{\mu_{s}^{K}(\dd x)\int_{\R}^{}{\mu_{s}^{K}(\dd y)b\left({\color{black}\sigma_{K}} \sqrt{K} x + z_{s}^{K}, {\color{black}\sigma_{K}}\sqrt{K}y + z_{s}^{K} \right)} } } \\
& \hspace{6cm} \times \left[\frac{1}{K}\left(y^{2} - x^{2} \right) - \frac{1}{K^{2}}(y-x)^{2} \right]^{2}  \\
& \hspace{1.5cm} + \frac{K}{K{\color{black}\sigma_{K}^{2}}}\int_{0}^{t}{\dd s \int_{\R}^{}{\mu_{s}^{K}(\dd x)\int_{\R}^{}{m\left({\color{black}\sigma_{K}} \sqrt{K} x + z_{s}^{K}, h \right)\theta\left({\color{black}\sigma_{K}}\sqrt{K}x + z_{s}^{K} \right)} } } \\
& \hspace{6cm} \times \left[\frac{h^{2}}{K^{2}} - \frac{h^{2}}{K^{3}} + \frac{2xh}{K^{\frac{3}{2}}} \right]^{2} \dd h.
\end{align*}
The announced result follows from \textsc{Taylor}'s formula and straightforward computations.~\qedhere
\end{proof}

\textit{(b) Moments of degree $6$.} Let us denote for all $\ell \in \N$,  $\widetilde{M}_{\ell}(\mu) := \left\langle \id^{\ell}, \mu \right\rangle$.

\begin{Lem} Let us consider the functions $f, g : \R^{2} \to \R$ and $h : \R^{3} \to \R$ respectively defined by $f(u,v) := u^{4}v^{2}$, $g(u,v) := u^{3}v^{3}$ and $h(u,v,w) := u^{2}v^{2}w^{2}$.  For all $K \in \N^{\star}$, %for all $\sigma > 0$, 
the processes $\left(M_{t\wedge \check{\tau}^{K}}^{K, P_{\id^{6}, 1}}\right)_{t\geqslant 0}$, $\left(M_{t\wedge \check{\tau}^{K}}^{K, P_{f, 2}}\right)_{t\geqslant 0}$, $\left(M_{t\wedge \check{\tau}^{K}}^{K, P_{g, 2}}\right)_{t\geqslant 0}$ and $\left(M_{t\wedge \check{\tau}^{K}}^{K, P_{h, 3}}\right)_{t\geqslant 0}$ satisfiy for all $t \geqslant 0$
\begin{equation*}
\begin{aligned}
\bullet \ M_{t\wedge\check{\tau}^{K}}^{K, P_{\id^{6}, 1}} & = M_{6}\left(\mu_{t \wedge \check{\tau}^{K}}^{K} \right) - M_{6}\left(\mu_{0}^{K} \right) + \frac{3}{K^{2}{\color{black}\sigma_{K}^{2}}}\int_{0}^{t\wedge\check{\tau}^{K}}{\left\{b\left(z_{s}^{K}, z_{s}^{K}\right)\left[4M_{6}\left(\mu_{s}^{K} \right)  \right.  \right.} \\
& \hspace{2.75cm}  \left. \left. -10M_{4}\left(\mu_{s}^{K} \right)M_{2}\left(\mu_{s}^{K} \right)\right] - 5\theta\left(z_{s}^{K} \right)m_{2}\left(z_{s}^{K} \right)M_{4}\left(\mu_{s}^{K} \right)\right\}\dd s \\
& \hspace{-1.25cm} + \frac{1}{K^{2}{\color{black}\sigma_{K}^{2}}}\int_{0}^{t\wedge \check{\tau}^{K}}{O\left(\frac{M_{3}\left(\mu_{s}^{K} \right)}{\sqrt{K}} + \frac{1 + M_{4}\left(\mu_{s}^{K} \right)}{K} + {\color{black}\sigma_{K}}\sqrt{K}M_{5}\left(\mu_{s}^{K} \right) \right.}, \\
& \hspace{6.85cm} \left. + \frac{ M_{6}\left(\mu_{s}^{K}\right)}{K^{2}} + {\color{black}\sigma_{K}} K^{\frac{3}{2}}M_{7}\left(\mu_{s}^{K} \right) \right)\dd s \\
\bullet \ M_{t\wedge \check{\tau}^{K}}^{K, P_{f, 2}} & =  M_{4}\left(\mu_{t\wedge \check{\tau}^{K}}^{K} \right)M_{2}\left( \mu_{t\wedge \check{\tau}^{K}}^{K}\right) - M_{4}\left(\mu_{0}^{K} \right)M_{2}\left( \mu_{0}^{K}\right) \\
& \hspace{-0.5cm}+\frac{1}{K^{2}{\color{black}\sigma_{K}^{2}}}\int_{0}^{t\wedge \check{\tau}^{K}}{\left\{2b\left(z_{s}^{K}, z_{s}^{K} \right)\left[6 M_{4}\left(\mu_{s}^{K} \right)M_{2}\left(\mu_{s}^{K} \right) -   M_{6}\left(\mu_{s}^{K} \right) + 4\widetilde{M}_{3}^{2}\left(\mu_{s}^{K} \right) \right. \right.}   \\
& \hspace{1.85cm} \left.\left. - 6 M_{2}^{3}\left(\mu_{s}^{K} \right) \right] - \theta\left(z_{s}^{K} \right)m_{2}\left(z_{s}^{K} \right)\left[6M_{2}^{2}\left(\mu_{s}^{K} \right) + M_{4}\left(\mu_{s}^{K} \right) \right] \right\}\dd s \\
& \hspace{-0.5cm} + \frac{1}{K^{2}{\color{black}\sigma_{K}^{2}}}\int_{0}^{t\wedge \check{\tau}^{K}}{O\left(\frac{M_{6}\left(\mu_{s}^{K} \right)}{K^{3}} + {\color{black}\sigma_{K}} K^{\frac{3}{2}}M_{7}\left(\mu_{s}^{K} \right) \right)\dd s}, \\ 
\bullet \ M_{t\wedge \check{\tau}^{K}}^{K, P_{g, 2}} & = \widetilde{M}_{3}^{2}\left(\mu_{t\wedge \check{\tau}^{K}}^{K} \right) - \widetilde{M}_{3}^{2}\left(\mu_{0}^{K} \right) + \frac{1}{K^{2}{\color{black}\sigma_{K}^{2}}}\int_{0}^{t\wedge\check{\tau}^{K}}{\left\{b\left(z_{s}^{K}, z_{s}^{K} \right)\left[-2M_{6}\left(\mu_{s}^{K} \right)   \right. \right.} \\
& \hspace{2.5cm} + \left.\left. 14\widetilde{M}_{3}^{2}\left(\mu_{s}^{K} \right) + 12M_{4}\left(\mu_{s}^{K} \right)M_{2}\left(\mu_{s}^{K} \right) - 18M_{2}^{3}\left(\mu_{s}^{K} \right)\right] \right\}\dd s \\
& \hspace{-0.5cm} + \frac{1}{K^{2}{\color{black}\sigma_{K}^{2}}}\int_{0}^{t\wedge \check{\tau}^{K}}{O\left(\frac{M_{6}\left(\mu_{s}^{K} \right)}{K^{3}} + {\color{black}\sigma_{K}} K^{\frac{3}{2}}M_{7}\left(\mu_{s}^{K} \right) \right)\dd s}, 
\end{aligned}
\end{equation*}
\begin{equation*}
\begin{aligned}
\bullet \ M_{t\wedge \check{\tau}^{K}}^{K, P_{h, 3}} & = M_{2}^{3}\left(\mu_{t\wedge \check{\tau}^{K}}^{K} \right) - M_{2}^{3}\left(\mu_{0}^{K} \right) + \frac{1}{K^{2}{\color{black}\sigma_{K}^{2}}}\int_{0}^{t\wedge \check{\tau}^{K}}{\left\{b\left(z_{s}^{K}, z_{s}^{K} \right)\left[12 M_{2}^{3}\left(\mu_{s}^{K} \right) \right. \right.} \\
& \hspace{1.5cm} - \left. \left. 6M_{4}\left(\mu_{s}^{K} \right)M_{2}\left(\mu_{s}^{K} \right)  \right] - 3\theta\left(z_{s}^{K} \right)m_{2}\left(z_{s}^{K} \right)M_{4}\left(\mu_{s}^{K} \right)\right\}\dd s \\
& \hspace{-0.5cm} + \frac{1}{K^{2}{\color{black}\sigma_{K}^{2}}}\int_{0}^{t\wedge \check{\tau}^{K}}{O\left(\frac{M_{6}\left(\mu_{s}^{K} \right)}{K^{3}} + {\color{black}\sigma_{K}} K^{\frac{3}{2}}M_{7}\left(\mu_{s}^{K} \right) \right)\dd s}.
\end{aligned}
\label{Martingale_id_id6}
\end{equation*}
\label{Cor_moment_ordre_6}
\end{Lem}

\begin{proof} The approximation proposed is obtained in a similar way to the proof of Lemma \ref{Cor_moment_ordre_2} \qedhere 
\end{proof}

\subsection{Proof of Lemma \ref{Lem_PB_Mg_Polynomes}\label{Sous_Section_6_2_PB_Mg_polynomes}}
From Proposition \ref{Prop_Generateur_Couple_Lent_Rapide}, for the choice of test functions $\Phi(z, \mu) := P_{f,n}(\mu)$ given by (\ref{Eq_fonctions_test_F_f_n}) and noting that for all $x, y \in \R$, \[\left\langle f, \left(\tau_{-\frac{y-x}{K}} \, \sharp \, \left[ \mu - \frac{\delta_{x}}{K} + \frac{\delta_{y}}{K}\right] \right)^{n}\right\rangle = \left\langle f \circ \tau_{-\frac{y-x}{K}}, \left[\mu - \frac{\delta_{x}}{K} + \frac{\delta_{y}}{K} \right]^{n} \right\rangle,\]
we obtain that $\LL^{K}P_{f, n}(z,\mu) =  \frac{K}{K{\color{black}\sigma_{K}^{2}}}\left(\rm{\textbf{\rm{\textbf{(A)}}}}^{K} + {\rm{\textbf{\rm{\textbf{(B)}}}}}^{K}\right)$ where
\begin{align*}
{\rm{\textbf{\rm{\textbf{(A)}}}}}^{K} & := \int_{\R}^{}{\mu\left(\dd x \right)\int_{\R}^{}{\mu\left(\dd y \right)b\left({\color{black}\sigma_{K}} \sqrt{K} x + z, {\color{black}\sigma_{K}} \sqrt{K} y + z \right)} {\rm{\textbf{\rm{\textbf{(C)}}}}}_{x,y}^{K}}, \\
{\rm{\textbf{\rm{\textbf{(B)}}}}}^{K} & := \int_{\R}^{}{\mu\left(\dd x \right)\int_{\R}^{}{\theta\left({\color{black}\sigma_{K}}\sqrt{K} x + z \right)m\left(z + {\color{black}\sigma_{K}}\sqrt{K}x, h \right) }{\rm{\textbf{\rm{\textbf{(C)}}}}}_{x,x+h/\sqrt{K}}^{K} \,\dd h  }
\end{align*}
where ${\rm{\textbf{\rm{\textbf{(C)}}}}}_{x,y}^{K} := \left\langle f \circ \tau_{-\frac{y-x}{K}}, \left[\mu - \frac{\delta_{x}}{K} + \frac{\delta_{y}}{K} \right]^{n} \right\rangle - \left\langle f, \mu^{n} \right\rangle$. Note that, from \textsc{Taylor}'s formula, we obtain that 
\begin{align*}
{\rm{\textbf{\rm{\textbf{(C)}}}}}_{x,y}^{K} & = \sum\limits_{k\, = \, 1}^{n}{\binom{n}{k} \left\langle f, \mu^{n-k}\left[-\frac{\delta_{x}}{K} + \frac{\delta_{y}}{K} \right]^{k} \right\rangle} - \frac{y-x}{K}\sum\limits_{i \, = \, 1}^{n}{\left\langle \partial_{i}f, \left[\mu - \frac{\delta_{x}}{K} + \frac{\delta_{y}}{K}\right]^{n} \right\rangle}  \\
& \qquad  + \frac{(y-x)^{2}}{2K^{2}}\sum\limits_{i\, = \, 1}^{n}{\sum\limits_{j\, = \, 1}^{n}{\left\langle \partial^{2}_{ij}f, \left[\mu - \frac{\delta_{x}}{K} + \frac{\delta_{y}}{K}\right]^{n} \right\rangle}}  + O\left(\frac{\left|y-x \right|^{3}}{K^{3}} \right).
\end{align*}

By abuse of notation, we do not indicate the orders of the products of measure $\mu^{n-k}\left[-\frac{\delta_{x}}{K} + \frac{\delta_{y}}{K}\right]^{k}$. \\ 

\textit{Step 1. Decomposition and study of ${\rm{\textbf{\rm{\textbf{(A)}}}}}^{K}$.} Note that ${\rm{\textbf{\rm{\textbf{(A)}}}}}^{K} = \sum_{i\, = \, 1}^{5}{{\rm{\textbf{\rm{\textbf{(A)}}}}}_{i}^{K}}$ where
\begin{align*}
{\rm{\textbf{\rm{\textbf{(A)}}}}}_{1}^{K} & := b(z,z) \sum\limits_{k\, = \, 1}^{n}{\binom{n}{k} \int_{\R}^{}{\mu(\dd x)\int_{\R}^{}{\mu(\dd y)\left\langle f, \mu^{n-k}\left[-\frac{\delta_{x}}{K} + \frac{\delta_{y}}{K} \right]^{k} \right\rangle}}}, \\
{\rm{\textbf{\rm{\textbf{(A)}}}}}_{2}^{K} & := -\frac{b(z,z)}{K} \sum\limits_{i\, = \, 1}^{n}{\sum_{k\, = \, 0}^{n}{\binom{n}{k}\int_{\R}^{}{\mu(\dd x)\int_{\R}^{}{\mu(\dd y) (y-x) \left\langle \partial_{i}f, \mu^{n-k}\left[-\frac{\delta_{x}}{K} + \frac{\delta_{y}}{K} \right]^{k} \right\rangle}}}}, \\
{\rm{\textbf{\rm{\textbf{(A)}}}}}_{3}^{K} & := \frac{ b(z,z)}{2K^{2}} \sum\limits_{i, j\, = \, 1}^{n}{\sum_{k\, = \, 0}^{n}{\binom{n}{k}\int_{\R}^{}{\mu(\dd x)\int_{\R}^{}{\mu(\dd y) (y-x)^{2} \left\langle \partial_{ij}^{2}f, \mu^{n-k}\left[-\frac{\delta_{x}}{K} + \frac{\delta_{y}}{K} \right]^{k} \right\rangle}}}}, \\
{\rm{\textbf{\rm{\textbf{(A)}}}}}_{4}^{K} & := O\left(\frac{M_{3}\left(\mu\right)}{K^{3}}\right), \\ 
%O\left(\frac{1}{K^{3}} \int_{\R}^{}{\mu(\dd x)\int_{\R}^{}{\mu(\dd y) \left|y-x\right|^{3}}}\right), \\ 
%{\rm{\textbf{\rm{\textbf{(A)}}}}}_{5}^{K} & := O\left(\sigma \sqrt{K} \int_{\R}^{}{\mu(\dd x)\int_{\R}^{}{\mu(\dd y)}\big(|x| + |y| \big){\rm{\textbf{\rm{\textbf{(C)}}}}}_{x,y}^{K}} \right).
{\rm{\textbf{\rm{\textbf{(A)}}}}}_{5}^{K} & :=  \int_{\R}^{}{\mu(\dd x)\int_{\R}^{}{\mu(\dd y)}\left[b\left({\color{black}\sigma_{K}}\sqrt{K} x + z, {\color{black}\sigma_{K}}\sqrt{K} y + z \right) - b(z,z) \right]{\rm{\textbf{\rm{\textbf{(C)}}}}}_{x,y}^{K}}.
\end{align*}
We first study precisely the ${\rm{\textbf{\rm{\textbf{(A)}}}}}_{2}^{K}$ term by explaining the case $k = 0$, $k = 1$ and $k \geqslant 2$. We will just give the result for the other ${\rm{\textbf{\rm{\textbf{(A)}}}}}_{i}^{K}$, $i\in\{1,3,5\}$: the approach remains the same. By the centered assumption of $\mu$, note that for all $i\in \{1, \cdots, n \}$,  \[\int_{\R}^{}{\mu(\dd x)\int_{\R}^{}{\mu(\dd y)(y-x)\left\langle \partial_{i}f, \mu^{n} \right\rangle}} = 0.\] Then,  using again the centered assumption of $\mu$, we obtain that
\begin{align*}
& \binom{n}{1}\sum_{i\, = \, 1}^{n}{\int_{\R}^{}{\mu(\dd x)\int_{\R}^{}{\mu(\dd y)(y-x)\left\langle \partial_{i}f, \mu^{n-1}\left[-\frac{\delta_{x}}{K} + \frac{\delta_{y}}{K} \right] \right\rangle}}} \\
& \hspace{0.0cm} = \frac{1}{K}\sum_{i \, = \, 1}^{n}{\sum_{j \, = \, 1}^{n}{\int_{\R}^{}{\cdots\int_{\R}^{}{\int_{\R}^{}{\mu(\dd x) \int_{\R}^{}{\mu(\dd y)}} }} }} \\
& \hspace{0.25cm} \times \left[x \partial_{i}f\left(x_{1}, \cdots, x_{j-1}, x, x_{j+1} \cdots, x_{n} \right) - x \partial_{i}f\left(x_{1}, \cdots, x_{j-1}, y, x_{j+1} \cdots, x_{n} \right) \phantom{\frac{1}{K}} \right. \\
& \hspace{-0.35cm} \left. \phantom{\frac{1}{K}} - y \partial_{i}f\left(x_{1}, \cdots, x_{j-1}, x, x_{j+1} \cdots, x_{n} \right)+ y \partial_{i}f\left(x_{1}, \cdots, x_{j-1}, y, x_{j+1} \cdots, x_{n} \right) \right]\prod_{\substack{\ell \, = \, 1 \\ \ell \, \neq \, j}}^{n}{\mu(\dd x_{\ell})} \\
& \hspace{0.0cm} = \frac{2}{K}\sum_{i \, = \, 1}^{n}{\sum_{j \, = \, 1}^{n}{\int_{\R}^{}{\cdots \int_{\R}^{}{\int_{\R}^{}{\mu(\dd x)x\partial_{i}f\left(x_{1}, \cdots, x_{j-1}, x, x_{j+1}, \cdots, x_{n} \right)\prod_{\substack{\ell \, = \, 1 \\ \ell \, \neq \, j}}^{n}{\mu\left(\dd x_{\ell} \right)}}}} }} \\
& \hspace{0.0cm} =  \frac{2}{K}\left\langle \left(\nabla f \cdot \bm{1} \right)\left(\bullet \cdot \bm{1} \right), \mu^{n} \right\rangle.
\end{align*}
As $\nabla f$ is bounded, we deduce that for all $k \geqslant 2$, for all $i\in \{1, \cdots, n\}$, \[\int_{\R}^{}{\mu(\dd x)\int_{\R}^{}{\mu(\dd y)(y-x)\left\langle \partial_{i}f, \mu^{n-k}\left[-\frac{\delta_{x}}{K} + \frac{\delta_{y}}{K} \right]^{k} \right\rangle}} = O\left(\frac{M_{1}(\mu)}{K^{2}} \right).\]
Therefore, we deduce that 
\[{\rm{\textbf{\rm{\textbf{(A)}}}}}_{2}^{K} = -\frac{2b(z,z)}{K^{2}}\left\langle \left(\nabla f \cdot \bm{1} \right)\left(\bullet \cdot \bm{1} \right), \mu^{n} \right\rangle + O\left(\frac{M_{1}(\mu)}{K^{3}} \right).\]
In similar way and using Assumptions \textbf{(A)} we have that
\begin{align*}
{\rm{\textbf{\rm{\textbf{(A)}}}}}_{1}^{K} & = \frac{b(z,z)}{K^{2}}\sum_{i\, = \, 1}^{n}{\sum\limits_{\substack{j \, = \, 1 \\ j \, \neq \, i}}^{n}{\left[\left\langle \Phi_{i,j}f, \mu^{n-1} \right\rangle - \left\langle f, \mu^{n} \right\rangle \right]}} + O\left(\frac{1}{K^{3}} \right), \\
 {\rm{\textbf{\rm{\textbf{(A)}}}}}_{3}^{K} & = \frac{b(z,z)}{K^{2}}\sum_{i\, = \, 1}^{n}{\sum_{j \, = \, 1}^{n}{\left\langle K_{i,j}f, \mu^{n+1} \right\rangle} } + O\left(\frac{M_{2}\left(\mu \right)}{K^{3}} \right), \\
% {\rm{\textbf{\rm{\textbf{(A)}}}}}_{4}^{K} & = O\left(\frac{M_{3}\left(\mu \right)}{K^{3}} \right) \\
 {\rm{\textbf{\rm{\textbf{(A)}}}}}_{5}^{K} &= O\left(\frac{{\color{black}\sigma_{K}}\sqrt{K}}{K} \left[M_{1}\left(\mu \right) + M_{2}\left(\mu \right) \right] \right).
\end{align*}

\textit{Step 2. Decomposition and study of ${\rm{\textbf{\rm{\textbf{(B)}}}}}^{K}$.} For all $\alpha, h \in \R$, we denote $\theta_{m}\left(\alpha, h \right) := \theta\left(\alpha \right)m\left(\alpha, h \right)$. Note that ${\rm{\textbf{\rm{\textbf{(B)}}}}}^{K} = \sum_{i\, = \, 1}^{5}{{\rm{\textbf{\rm{\textbf{(B)}}}}}_{i}^{K}}$ where
\begin{align*}
{\rm{\textbf{\rm{\textbf{(B)}}}}}_{1}^{K} & := \sum\limits_{k\, = \, 1}^{n}{\binom{n}{k} \int_{\R}^{}{\mu(\dd x)\int_{\R}^{}{\theta_{m}(z, h)\left\langle f, \mu^{n-k}\left[-\frac{\delta_{x}}{K} + \frac{\delta_{x + \frac{h}{\sqrt{K}}}}{K} \right]^{k} \right\rangle \dd h}}}, \\
{\rm{\textbf{\rm{\textbf{(B)}}}}}_{2}^{K} & := -\frac{1}{K^{\frac{3}{2}}} \sum\limits_{i\, = \, 1}^{n}{\sum_{k\, = \, 0}^{n}{\binom{n}{k}\int_{\R}^{}{\mu(\dd x)\int_{\R}^{}{ \theta_{m}(z,h)h \left\langle \partial_{i}f, \mu^{n-k}\left[-\frac{\delta_{x}}{K} + \frac{\delta_{x+ \frac{h}{\sqrt{K}}}}{K} \right]^{k} \right\rangle \dd h}}}}, \\
{\rm{\textbf{\rm{\textbf{(B)}}}}}_{3}^{K} & := \frac{1}{2K^{3}} \sum\limits_{i, j\, = \, 1}^{n}{\sum_{k\, = \, 0}^{n}{\binom{n}{k}\int_{\R}^{}{\mu(\dd x)\int_{\R}^{}{ \theta_{m}(z,h)h^{2} \left\langle \partial_{ij}^{2}f, \mu^{n-k}\left[-\frac{\delta_{x}}{K} + \frac{\delta_{x + \frac{h}{\sqrt{K}}}}{K} \right]^{k} \right\rangle \dd h}}}}, \\
{\rm{\textbf{\rm{\textbf{(B)}}}}}_{4}^{K} & := O\left(\frac{1}{K^{\frac{9}{2}}} \int_{\R}^{}{\mu(\dd x)\int_{\R}^{}{\theta_{m}(z,h) \left|h\right|^{3} \dd h}}\right), \\ 
{\rm{\textbf{\rm{\textbf{(B)}}}}}_{5}^{K} & :=  \int_{\R}^{}{\mu(\dd x)\int_{\R}^{}{\left[\theta_{m}\left({\color{black}\sigma_{K}}\sqrt{K} x + z, h \right) - \theta_{m}(z,h) \right]{\rm{\textbf{\rm{\textbf{(C)}}}}}_{x,x + h/\sqrt{K}}^{K}}\, \dd h}.
\end{align*}
Note that, from \textsc{Taylor}'s formula and Assumptions \textbf{(A)}, we have that
\begin{align*}
\hspace{0.75cm}&  \binom{n}{1}\int_{\R}^{}{\mu(\dd x)\int_{\R}^{}{\theta_{m}(z,h)\left\langle f, \mu^{n-1}\left[-\frac{\delta_{x}}{K} + \frac{\delta_{x + \frac{h}{\sqrt{K}}}}{K} \right] \right\rangle\dd h}}  \\
& \hspace{1.5cm} = \frac{1}{K} \sum_{i\, = \, 1}^{n}{\int_{\R}^{}{\cdots \int_{\R}^{}{\int_{\R}^{}{\mu(\dd x_{i})\int_{\R}^{}{\theta_{m}(z,h)\left[\frac{h}{\sqrt{K}}\partial_{i}f\left(x_{1}, \cdots, x_{i}, \cdots, x_{n} \right) \right.}}}}}\\ 
& \hspace{2cm} \left. + \frac{h^{2}}{2K}\partial_{ii}^{2}f\left(x_{1}, \cdots, x_{i}, \cdots, x_{n} \right) + O\left(\frac{\left|h \right|^{3}}{K^{\frac{3}{2}}} \right) \right]\prod_{\substack{\ell \, = \, 1 \\ \ell \, \neq \, i}}^{n}{\mu\left(\dd x_{\ell}\right)} \\
& \hspace{1.5cm} =  \frac{\theta(z)m_{2}(z)}{K^{2}}\left\langle \frac{\Delta f}{2}, \mu^{n} \right\rangle + O\left(\frac{1}{K^{\frac{5}{2}}} \right). \hspace{8.6cm}
\end{align*}
In similar way to Step 1, we obtain that 
\[ {\rm{\textbf{\rm{\textbf{(B)}}}}}_{1}^{K}  = \frac{\theta(z)m_{2}(z)}{K^{2}}\left\langle \frac{\Delta f}{2}, \mu^{n} \right\rangle + O\left(\frac{1}{K^{\frac{5}{2}}} \right), \qquad {\rm{\textbf{\rm{\textbf{(B)}}}}}_{2}^{K}  =  O\left(\frac{1}{K^{3}} \right),\]
\[{\rm{\textbf{\rm{\textbf{(B)}}}}}_{3}^{K} =  O\left(\frac{1}{K^{3}} \right), \quad \  {\rm{\textbf{\rm{\textbf{(B)}}}}}_{4}^{K} = O\left(\frac{1}{K^{\frac{9}{2}}}\right) \quad \ {\rm{and}} \quad \ {\rm{\textbf{\rm{\textbf{(B)}}}}}_{5}^{K} = O\left(\frac{{\color{black}\sigma_{K}}\sqrt{K}}{K^{\frac{3}{2}}} M_{1}\left(\mu \right)  \right).\]
and using Assumptions (\ref{Hypothese_Gamme_sigma}) the announced result follows. \\

\textit{Step 3. Martingality.} By classical arguments, since $M^{K, P_{f, n}}$ is bounded, we obtain that $M^{K, P_{f, n}}$ is a square integrable martingale started at $0$.

\section{Moments estimates\label{Section_4_Estimees_de_moments}}
Lemmas \ref{Cor_moment_ordre_2} and \ref{Cor_moment_ordre_6} lead us to look for moment estimates. In Lemma \ref{Cor_moment_ordre_2} we have given the \textsc{Doob} semi-martingale decomposition of the second-order moment and note that involves the third-order moment in the error term. The presence of the higher order moment generates a difficulty to obtain a fine control of the second order moment. This difficulty is {\color{black} overcome} by introducing the stopping time $\check{\tau}^{K}$, given by (\ref{Temps_arret_check}). In Section \ref{Sous_Section_4_1_new_Moments_ordre_6}, we give estimates of the moment of order $6$ and some corollaries useful for Sections \ref{Section_5_Tension_LENT-RAPIDE} and \ref{Section_6_Caract_Gamma_Limite}. In Section \ref{Sous_Section_4_3_Inegalite_M2} we establish estimates in expectation and probability of the second order moment up to the stopping time $\check{\tau}^{K}$ and  we also control the martingale bracket of the \textsc{Doob} decomposition given in Lemma \ref{Cor_moment_ordre_2}. In Sections \ref{Sous_Section_4_4_Convergence_tau} and \ref{Sous_section_4_5_Sortie_domaine}, we prove that the stopping time $\tau^{K} := \check{\tau}^{K} \wedge \widehat{\tau}^{K}$, given by (\ref{Temps_arret_check_chapeau}), converges in probability to $+\infty$ when $K\to + \infty$ using coupling arguments between $M_{2}\left(\mu^{K} \right)$ and a biased random walk on $\N^{\star}$, reflected in $1$, and large deviations estimates for this random walk. 

\subsection{Estimates of the moment of order $6$\label{Sous_Section_4_1_new_Moments_ordre_6}}
{\color{black} In this section, we establish a technical lemma useful for Lemma \ref{Lem_Controle_M6_et_M4M2} below. The idea is to obtain an explicit upper bound of the moments of order $6$.

\begin{Lem} Given any continuous function $y(t)$ satisfying 
\begin{equation}
y(t) \leqslant y(s) + \int_{s}^{t}{F\left(y(r) \right)\dd r}, \qquad \forall s \leqslant t
\label{Eq_Nicolas_1}
\end{equation}
for some \textsc{Lipschitz} function $F : \R \to \R$, then $y(s) \leqslant z(s)$ for all $s \geqslant 0$, where $z(t)$ is the unique solution of 
\begin{equation}
z(t) = z(s) + \int_{s}^{t}{F\left(z(r) \right)\dd r}, \qquad \forall s \leqslant t
\label{Eq_Nicolas_2}
\end{equation}
such that $z(0) = y(0)$.
\label{Lem_Nicolas}
\end{Lem}

\begin{proof} First, it is sufficient to prove that, for all $\varepsilon >0$, $y(s) \leqslant z_{\varepsilon}(s)$ for all $s \geqslant 0$, where $z_{\varepsilon}$ is the unique solution of 
\begin{equation}
z_{\varepsilon}(t) = z_{\varepsilon}(s) + \int_{s}^{t}{\left[F\left(z_{\varepsilon}(r) \right) + \varepsilon \right]\dd r}, \qquad \forall s \leqslant t
\label{Eq_Nicolas_3}
\end{equation}
such that $z_{\varepsilon}(0) = y(0) + \varepsilon$. Indeed, it follows from (\ref{Eq_Nicolas_2}) and (\ref{Eq_Nicolas_3}) that, for all $t\geqslant 0$ and $\varepsilon >0$
\[\left|z(t) - z_{\varepsilon}(t) \right| \leqslant \varepsilon(1+t) + C_{\rm Lip}\int_{0}^{t}{\left|z(s) - z_{\varepsilon}(s) \right|\dd s}  \]
where $C_{\rm Lip}$ if the \textsc{Lipschitz} constant of $F$. \textsc{Gronwall}'s lemma then entails
\[\left|z(t) - z_{\varepsilon}(t) \right| \leqslant \varepsilon (1+t)\exp\left(C_{\rm Lip}t \right) \xrightarrow[ \varepsilon \to 0 ]{} 0\]
so that $y(t) \leqslant z_{\varepsilon}(t)$, for all $\varepsilon > 0$ indeed implies $y(t) \leqslant z(t)$. \\
\indent So, let $\varepsilon >0$ be fixed and let us prove that $y(t) \leqslant z_{\varepsilon}(t)$ for all $t\geqslant 0$. Define 
\[\vartheta_{\varepsilon} := \inf\left\{t\geqslant 0 \left| \phantom{1^{1^{1}}} \hspace{-0.55cm}  \right. y(t) \geqslant z_{\varepsilon}(t)  \right\}.\]
Our goal is to prove that $\vartheta_{\varepsilon} = + \infty$. \\
\indent By continuity of $y$, $\vartheta_{\varepsilon} >0$ and, if $\vartheta_{\varepsilon} < \infty$, $y\left(\vartheta_{\varepsilon}\right) = z_{\varepsilon}\left(\vartheta_{\varepsilon} \right)$. Assume by contradiction that $\vartheta_{\varepsilon} < \infty$. Then, using (\ref{Eq_Nicolas_1}) and (\ref{Eq_Nicolas_3}), for all $\delta \in \left(0, \vartheta_{\varepsilon} \right)$, \[\frac{z_{\varepsilon}\left(\vartheta_{\varepsilon}\right) - y\left(\vartheta_{\varepsilon} \right) - \left(z_{\varepsilon}\left(\vartheta_{\varepsilon} - \delta\right) - y\left(\vartheta_{\varepsilon}  - \delta\right) \right)}{{\color{black} \delta}} \geqslant \varepsilon + \frac{1}{\delta}\int_{{\color{black}\vartheta_{\varepsilon}}-\delta}^{{\color{black}\vartheta_{\varepsilon}}}{\left[F\left(z_{\varepsilon}(r) \right) - F(y(r))  \right]\dd r}. \]
By continuity of $z_{\varepsilon}$ and $y$, the last term of the right-hand side converges to $0$ when $\delta \to 0$. Hence the left-hand side is positive for all $\delta$ small enough. Since $z_{\varepsilon}\left(\vartheta_{\varepsilon}\right) = y\left(\vartheta_{\varepsilon}\right)$, this implies that 
\[z_{\varepsilon}\left(\vartheta_{\varepsilon} - \delta\right) - y\left(\vartheta_{\varepsilon}  - \delta\right) < 0 \]
for such $\delta$. This contradicts the definition of $\vartheta_{\varepsilon}$ as the first time $t$ such that $z_{\varepsilon}(t) \leqslant y(t)$. \qedhere
\end{proof}
}

\begin{Lem} There exists $K_{0} \in \N^{\star}$ large enough and two constants $C_{1}, C_{2}>0$ such that for all $K \geqslant K_{0}$ and $t\geqslant 0$, 
%\begin{align*}
%& \E\left(\left[\frac{3}{4}M_{6}\left(\mu_{t}^{K} \right) + 3M_{4}\left(\mu_{t}^{K} \right)M_{2}\left(\mu_{t}^{K} \right) + \frac{3}{2}M_{2}^{3}\left(\mu_{t}^{K} \right)\right]\II_{t\leqslant \check{\tau}^{K}}\right)   \leqslant \left( \frac{{\color{black} 21}}{4}C_{6}^{\star} - C \right)\exp\left(-\frac{3\underline{b}}{4K^{2}{\color{black}\sigma_{K}^{2}}}t \right) + C.
%\end{align*}
\begin{align*}
\E\left(M_{6}\left(\mu_{t}^{K} \right)\II_{t\leqslant \check{\tau}^{K}} \right) \leqslant C_{1}\exp\left(-\frac{9\underline{b}}{28K^{2}{\color{black}\sigma_{K}^{2}}}t \right) + C_{2}
\end{align*}
\label{Lem_Controle_M6_et_M4M2}
\end{Lem}

{\color{black}
\begin{proof} \textit{Step 1. Pathwise inequality.} Note that for all $t \geqslant 0$, 
\begin{equation}
M_{7}\left(\mu_{t\wedge \check{\tau}^{K}}^{K} \right) \leqslant \Diam\left(\Supp{\mu_{t\wedge \check{\tau}^{K}}^{K}}\right) M_{6}\left(\mu_{t\wedge \check{\tau}^{K}}^{K} \right) \leqslant \frac{ M_{6}\left(\mu_{t\wedge \check{\tau}^{K}}^{K} \right)}{{\color{black}\sigma_{K}} K^{\frac{3+\varepsilon}{2}}}.
\label{Eq_M7}
\end{equation}
Introducing for all $t\geqslant 0$, $A_{t}:= \frac{3}{4}M_{6}\left(\mu_{t}^{K} \right) + 7M_{4}\left( \mu_{t}^{K}\right)M_{2}\left( \mu_{t}^{K}\right) - 4\widetilde{M}_{3}^{2}\left(\mu_{t}^{K} \right) + \frac{3}{2}M_{2}^{3}\left(\mu_{t}^{K} \right)$, we obtain from Lemma \ref{Cor_moment_ordre_6}, (\ref{Eq_M7}) and Assumptions \textbf{(A)} that there exists a constant $C>0$ such that for all $t > s \geqslant 0$, 
\begin{align*}
A_{t\wedge \check{\tau}^{K}} & \leqslant A_{s\wedge \check{\tau}^{K}} - \frac{1}{K^{2}{\color{black}\sigma_{K}^{2}}}\int_{s\wedge \check{\tau}^{K}}^{t\wedge \check{\tau}^{K}}{\left\{\left(\underline{b} - \frac{C}{K^{\frac{\varepsilon}{2}}} \right)\left[3 M_{6}\left(\mu_{r}^{K} \right) + \frac{9}{2}M_{4}\left(\mu_{r}^{K} \right)M_{2}\left(\mu_{r}^{K} \right) + 6 M_{2}^{3}\left(\mu_{r}^{K} \right) \right] \right.} \\
& \hspace{2.5cm} - \left. \overline{\theta}\overline{m}_{2}\left[\frac{91}{4}M_{4}\left(\mu_{r}^{K} \right) + 42M_{2}^{2}\left(\mu_{r}^{K} \right) \right]\dd r\right\} + {\rm Mart}_{t\wedge \check{\tau}^{K}} - {\rm Mart}_{s\wedge \check{\tau}^{K}} 
\end{align*}
where $\left({\rm Mart}_{t} \right)_{t\geqslant 0}$ is a martingale. Note that, choosing $K_{0}$ large enough, for all $K \geqslant K_{0}$, $\underline{b} - CK^{-\varepsilon/2} > \underline{b}/2$. 
From the inequalities $M_{2}^{2}(\mu) \leqslant M_{4}(\mu)$ and $M_{4}(\mu) \leqslant M_{6}^{2/3}\left(\mu \right) \leqslant \alpha M_{6}(\mu) + 1/\alpha^{2}$ with $\alpha := 2 \underline{b}/259 \overline{\theta}\overline{m}_{2}$, we deduce that 
\[\frac{3\underline{b}}{2} M_{6}\left(\mu\right) - \overline{\theta}\overline{m}_{2}\left[\frac{91}{4}M_{4}\left(\mu \right) + 42M_{2}^{2}\left(\mu\right) \right] \geqslant \frac{3\underline{b}}{2} M_{6}\left(\mu\right) -  \frac{259\overline{\theta}\overline{m}_{2}}{4}M_{4}(\mu) \geqslant \underline{b} M_{6}\left(\mu\right) - \frac{259^{3}\overline{\theta}^{3}\overline{m}_{2}^{3}}{16 \underline{b}^{2}}. \]
Thus, for all $t > s \geqslant 0$, 
\begin{align*}
A_{t\wedge \check{\tau}^{K}} & \leqslant A_{s \wedge \check{\tau}^{K}} + {\rm Mart}_{t\wedge \check{\tau}^{K}} - {\rm Mart}_{s\wedge \check{\tau}^{K}} \\
&  \hspace{0.5cm} -\frac{\underline{b}}{2K^{2}{\color{black}\sigma_{K}^{2}}}\int_{s\wedge \check{\tau}^{K}}^{t\wedge\check{\tau}^{K}}{\left[2 M_{6}\left(\mu_{r}^{K} \right) + \frac{9}{2} M_{4}\left(\mu_{r}^{K} \right)M_{2}\left(\mu_{r}^{K} \right) + 6 M_{2}^{3}\left(\mu_{r}^{K} \right) - \left(\frac{259\overline{\theta}\overline{m}_{2}}{2\underline{b}}\right)^{3} \right] \dd r} \\
& \leqslant  A_{s \wedge \check{\tau}^{K}} - \frac{9\underline{b}}{28 K^{2}{\color{black}\sigma_{K}^{2}}}\int_{s\wedge \check{\tau}^{K}}^{t\wedge \check{\tau}^{K}}{\left(A_{r} - \frac{14}{9}\left(\frac{259\overline{\theta}\overline{m}_{2}}{2\underline{b}}\right)^{3}\right)\dd r} + {\rm Mart}_{t\wedge \check{\tau}^{K}} - {\rm Mart}_{s\wedge \check{\tau}^{K}} 
\end{align*}

\textit{Step 2. Conclusion.} Let us consider the stochastic process $\left(N\left(\mu_{t}^{K} \right) \right)_{t\geqslant 0}  $ defined by 
 \[N\left(\mu_{t}^{K} \right) := \left\{ \begin{array}{lll}
A_{t} & {\rm if} & s< t \leqslant \check{\tau}^{K}  \vspace{0.15cm} \\
\left(A_{\check{\tau}^{K}} - C_{0} \right)\exp\left(-\frac{3\underline{b}}{4K^{2}{\color{black}\sigma_{K}^{2}}}\left(t - \check{\tau}^{K} \right) \right) + C_{0} & {\rm if} & s < \check{\tau}^{K} < t \vspace{0.15cm} \\
0 & {\rm if} & \check{\tau}^{K} < s < t
\end{array}  \right. \] 
where $C_{0} := \frac{14}{9}\left(\frac{259\overline{\theta}\overline{m}_{2}}{2\underline{b}}\right)^{3}$. From Step 1, for all $t > s \geqslant 0$,  
\[ N\left(\mu_{t}^{K} \right) \leqslant N\left(\mu_{s}^{K} \right)  - \frac{9\underline{b}}{28K^{2}{\color{black}\sigma_{K}^{2}}}\int_{s}^{t}{\left[N\left(\mu_{r}^{K} \right) - C_{0} \right]\dd r} + {\rm Mart}_{t} - {\rm Mart}_{s}\]
which entails
\[y(t) \leqslant y(s) - \frac{9\underline{b}}{28K^{2}{\color{black}\sigma_{K}^{2}}}\int_{s}^{t}{\left(y(r) - C_{0} \right)\dd r} \]
where $y(t) := \E\left(N\left(\mu_{t}^{K} \right) \right)$.  From Lemma \ref{Lem_Nicolas}, we deduce that for all $t\geqslant 0$, $y(t) \leqslant z(t)$ where $z(t)$ is the unique solution of 
\[z(t) = z(s) - \frac{9\underline{b}}{28K^{2}{\color{black}\sigma_{K}^{2}}}\int_{s}^{t}{\left(z(r) - C_{0} \right)\dd r} \]
such that $z(0) = A_{0}$, i.e. $z(t) = (z(0) - C_{0})\exp\left(-\frac{9\underline{b}}{28K^{2}{\color{black}\sigma_{K}^{2}}}t \right) + C_{0}$.
 From (\ref{Def_distribution_recentre_dilate}) then (\ref{Hypothese_C_etoile}), note that 
\begin{equation*}
M_{6}\left(\mu_{0}^{K} \right) = \left\langle \id^{6}, \mu_{0}^{K} \right\rangle = \frac{1}{K^{3}{\color{black}\sigma_{K}^{6}}}\left\langle \left(\id - x_{0} \right)^{6}, \nu_{0}^{K} \right\rangle \leqslant C_{6}^{\star}.
\label{Eq_M_6_C_6_etoile}
\end{equation*}  
From \textsc{Cauchy-Schwarz}'s inequality $\widetilde{M}_{3}^{2}(\mu) \leqslant M_{4}(\mu)M_{2}(\mu)$, note that $z(0) \leqslant 21C_{6}^{\star}/4$ and so, for all $t\geqslant 0$, 
\begin{align*}
\E\left(M_{6}\left(\mu_{t}^{K} \right)\II_{t\leqslant \check{\tau}^{K}} \right)&\leqslant \frac{4}{3}\E\left(\left[\frac{3}{4}M_{6}\left(\mu_{t}^{K} \right) + 3M_{4}\left(\mu_{t}^{K} \right)M_{2}\left(\mu_{t}^{K} \right) + \frac{3}{2}M_{2}^{3}\left(\mu_{t}^{K} \right) \right]\II_{t\leqslant \check{\tau}^{K}} \right) \\
& \leqslant \frac{4}{3}z(t) \leqslant \frac{4}{3}\left(\frac{21}{4}\left(C_{6}^{\star} - C_{0} \right)\exp\left(-\frac{9\underline{b}}{28K^{2}{\color{black}\sigma_{K}^{2}}}t \right)  + C_{0} \right)
\end{align*}
and the announced result follows. \qedhere
\end{proof}
}

The following lemma will be useful in the proof of Theorem \ref{Thm_Kurtz_adapte}. 

\begin{Lem} 
For all $t \geqslant 0$, 
\[\sup_{K\in \N^{\star}}{\E\left(\int_{0}^{t\wedge \check{\tau}^{K}}{M_{6}\left(\mu_{s}^{K} \right)\dd s} \right)} < \infty.\]
\label{Lem_Controle_Integrale_Temps_M2}
\end{Lem}

\begin{proof} Let $t\geqslant 0$ be fixed. %Note that \[\int_{0}^{t}{M_{6}\left(\mu_{s}^{K} \right)\II_{s\wedge \check{\tau}^{K}} \dd s} \leqslant \frac{4}{3}\int_{0}^{t}{\left(\frac{3}{4} M_{6}\left(\mu_{s}^{K} \right) + 3M_{4}\left(\mu_{s}^{K} \right)M_{2}\left(\mu_{s}^{K} \right) + \frac{3}{2}M_{2}^{3}\left(\mu_{s}^{K} \right) \right)\II_{s\leqslant \check{\tau}^{K}}\dd s}.\] 
By \textsc{Fubini}'s theorem and Lemma \ref{Lem_Controle_M6_et_M4M2}, there exists two constants $C_{1}, C_{2}>0$ such that 
%\begin{align*}
%\E\left(\int_{0}^{t}{M_{6}\left(\mu_{s}^{K} \right)\II_{s\wedge \check{\tau}^{K}}\dd s} \right) & \leqslant \frac{4}{3}\left(\left(\frac{{\color{black} 21}}{4}C_{6}^{\star} - C\right)\int_{0}^{t}{\exp\left(-\frac{3\underline{b}}{4K^{2}{\color{black}\sigma_{K}^{2}}}s \right) \dd s }  + Ct \right)  \leqslant \frac{{\color{black} 21} C_{6}^{\star} + 4C}{3}t
%\end{align*}
{\color{black}
\[\E\left(\int_{0}^{t}{M_{6}\left(\mu_{s}^{K} \right)\II_{s\wedge \check{\tau}^{K}}\dd s} \right) \leqslant C_{1}\int_{0}^{t}{\exp\left(-\frac{9\underline{b}}{28K^{2}{\color{black}\sigma_{K}^{2}}}s \right)\dd s} + C_{2}t \leqslant (C_{1} + C_{2})t \]
}
which ends the proof. \qedhere
\end{proof}

The last lemma has the following consequence, used in Sections \ref{Section_5_Tension_LENT-RAPIDE} and \ref{Section_6_Caract_Gamma_Limite}. 
\begin{Cor}
For all $T >0$, the family \[\left(\int_{0}^{t\wedge \check{\tau}^{K}}{M_{5}\left(\mu_{s}^{K} \right)\dd s} \right)_{t\in[0, T], K \in \N^{\star}}\]
is uniformly integrable.
\label{Cor_Integrale_temps_M2_UI}
\end{Cor}

\begin{proof} Note that for all $K \in \N^{\star}$, for all $t \in [0, T]$, for all $A \in (0, +\infty)$, \[\II_{\int_{0}^{t}{M_{5}\left(\mu_{s}^{K} \right)\II_{s\leqslant \check{\tau}^{K}}}\dd s \geqslant A}\int_{0}^{t}{M_{5}\left(\mu_{s}^{K} \right)\II_{s\leqslant \check{\tau}^{K}}\dd s} \leqslant \frac{1}{\sqrt[5]{A}}\left(\int_{0}^{t}{M_{5}\left(\mu_{s}^{K} \right)\II_{s\leqslant \check{\tau}^{K}}\dd s}\right)^{6/5},\]
we deduce from \textsc{H\"{o}lder}'inequality that, for all $t \in [0, T]$, for all $A \in (0, + \infty)$, 
\begin{align*}
& \sup_{K\in \N^{\star}}{\E\left(\II_{\int_{0}^{t}{M_{5}\left(\mu_{s}^{K} \right)\II_{s\leqslant \check{\tau}^{K}}}\dd s \geqslant A}\int_{0}^{t}{M_{5}\left(\mu_{s}^{K} \right)\II_{s\leqslant \check{\tau}^{K}}\dd s} \right)} \\
& \hspace{5cm} \leqslant \frac{T}{\sqrt[5]{A}}\sup_{K\in \N^{\star}}{\E\left(\int_{0}^{t}{M_{5}^{6/5}\left(\mu_{s}^{K} \right)\II_{s\leqslant \check{\tau}^{K}}\dd s}\right)} \\
& \hspace{5cm} \leqslant \frac{T}{\sqrt[5]{A}}\sup_{K\in \N^{\star}}{\E\left(\int_{0}^{t}{M_{6}\left(\mu_{s}^{K} \right)\II_{s\leqslant \check{\tau}^{K}}\dd s}\right)}.
\end{align*}
From Lemma \ref{Lem_Controle_Integrale_Temps_M2}, we deduce that the right hand side of the previous inequality goes to $0$ when $A\to + \infty$ and so the announced result. \qedhere 
\end{proof}

\subsection{Some inequalities on the moment of order $2$\label{Sous_Section_4_3_Inegalite_M2}}
\begin{Lem} There exists $K_{0} \in \N^{\star}$ large enough such that for all $K \geqslant K_{0}$, the process $M_{2}\left(\mu^{K} \right)$ satisfies for all $t \geqslant 0$ the following inequalities\!\! :
\begin{equation}
\begin{aligned}
& M_{2}\left(\mu_{{\color{black} s \wedge \check{\tau}^{K}}}^{K} \right) - \frac{5 \overline{b}}{2K^{2}{\color{black}\sigma_{K}^{2}}}\int_{{\color{black} s \wedge \check{\tau}^{K}}}^{t\wedge \check{\tau}^{K}}{\left[M_{2}\left(\mu_{r}^{K} \right) - \frac{2\underline{\theta}\underline{m}_{2}}{5\overline{b}} + \frac{2\overline{\theta}\overline{m}_{2}}{5\overline{b}K} \right]\dd r} + M_{t\wedge \check{\tau}^{K}}^{K, P_{\id^{2}, 1}}  {\color{black} - M_{s\wedge \check{\tau}^{K}}^{K, P_{\id^{2}, 1}}} \\
& \hspace{0.5cm} \leqslant M_{2}\left(\mu_{t\wedge \check{\tau}^{K}}^{K} \right) \leqslant M_{2}\left(\mu_{{\color{black} s \wedge \check{\tau}^{K}}}^{K} \right) - \frac{\underline{b}}{K^{2}{\color{black}\sigma_{K}^{2}}}\int_{{\color{black} s \wedge \check{\tau}^{K}}}^{t\wedge \check{\tau}^{K}}{\left[M_{2}\left(\mu_{r}^{K} \right) - \frac{\overline{\theta}\overline{m}_{2}}{\underline{b}} \right]\dd r} + M_{t\wedge \check{\tau}^{K}}^{K, P_{\id^{2}, 1}} {\color{black} - M_{s\wedge \check{\tau}^{K}}^{K, P_{\id^{2}, 1}}}
\end{aligned}
\label{Inegalites_M2}
\end{equation}
where $M_{t\wedge \check{\tau}^{K}}^{K, P_{\id^{2}, 1}}$ is defined in {\rm Proposition \ref{Prop_Martingale_f_NON_bornee_n}}. Moreover, for $t\geqslant 0$, 
\begin{equation*}
\E\left(M_{2}\left(\mu_{t}^{K} \right)\II_{t\leqslant \check{\tau}^{K}} \right) \leqslant \E\left( M_{2}\left(\mu_{0}^{K} \right) - \frac{\overline{\theta}\overline{m}_{2}}{\underline{b}}\right)\exp\left(-\frac{\underline{b}}{K^{2}{\color{black}\sigma_{K}^{2}}}t \right) + \frac{\overline{\theta}\overline{m}_{2}}{\underline{b}}.
\label{Inegalites_M2_Esperances}
\end{equation*}
\label{Lem_Controle_M2_Trajectoire_Esperance}
\end{Lem}

\begin{proof} Note that for all $t\geqslant 0$, 
\begin{equation}
M_{3}\left(\mu_{t\wedge \check{\tau}^{K}}^{K} \right) \leqslant \Diam\left(\Supp{\mu_{t\wedge \check{\tau}^{K}}^{K}}\right) M_{2}\left(\mu_{t\wedge \check{\tau}^{K}}^{K} \right) \leqslant \frac{ M_{2}\left(\mu_{t\wedge \check{\tau}^{K}}^{K} \right)}{{\color{black}\sigma_{K}} K^{\frac{3+\varepsilon}{2}}}.
\label{Eq_M3}
\end{equation}
From Lemma \ref{Cor_moment_ordre_2},  (\ref{Eq_M3}) and Assumptions \textbf{(A)}, there exists two constants $C, \widetilde{C} >0$ such that for all $t > s \geqslant 0$
\begin{align*}
&  M_{2}\left(\mu_{{\color{black} s \wedge \check{\tau}^{K}}}^{K} \right) - \frac{1}{K^{2}{\color{black}\sigma_{K}^{2}}}\int_{{\color{black} s \wedge \check{\tau}^{K}}}^{t\wedge \check{\tau}^{K}}{\left(\left[2\overline{b} + \frac{C}{K^{\frac{\varepsilon}{2}}} \right] M_{2}\left(\mu_{r}^{K } \right) - \underline{\theta} \underline{m}_{2} + \frac{\overline{\theta}\overline{m}_{2}}{K} \right)\dd r} + M_{t\wedge \check{\tau}^{K}}^{K, P_{\id^{2}, 1}} {\color{black} - M_{s\wedge \check{\tau}^{K}}^{K, P_{\id^{2}, 1}}} \\
& \hspace{1.5cm} \leqslant M_{2}\left(\mu_{t\wedge \check{\tau}^{K}}^{K } \right) \\
& \hspace{0.15cm} \leqslant M_{2}\left(\mu_{{\color{black} s \wedge \check{\tau}^{K}}}^{K } \right) - \frac{1}{K^{2}{\color{black}\sigma_{K}^{2}}}\int_{{\color{black} s \wedge \check{\tau}^{K}}}^{\color{black} t\wedge \check{\tau}^{K}}{\left(\left[2\underline{b} - \frac{\widetilde{C}}{K^{\frac{\varepsilon}{2}}} \right]M_{2}\left(\mu_{r}^{K } \right) - \overline{\theta} \overline{m}_{2}  + \frac{\underline{\theta}\underline{m}_{2}}{K}\right)\dd r} + M_{t\wedge \check{\tau}^{K}}^{K, P_{\id^{2}, 1}} {\color{black} - M_{s\wedge \check{\tau}^{K}}^{K, P_{\id^{2}, 1}}}.
\end{align*}
Noting that for all $K \geqslant K_{0}$, $2\underline{b} - \widetilde{C}K^{-\varepsilon/2} > \underline{b}$ and $2\overline{b} + CK^{-\varepsilon/2} < 5\overline{b}/2$, the announced first result follows. In similar way to Step 2 of the proof of Lemma \ref{Lem_Controle_M6_et_M4M2}, we obtain the second part of the announced result. \qedhere
\end{proof}

\begin{Cor}
For all $t \geqslant 0$, $\E\left(M_{2}\left(\mu_{t}^{K} \right)\II_{t\leqslant \check{\tau}^{K}} \right) \leqslant \max\left\{ C_{2}^{\star} , \overline{\theta}\overline{m}_{2}{\color{black}/\underline{b}} \right\}$.
\label{Cor_Esp_M2_max_2_constantes}
\end{Cor}

\begin{proof} From (\ref{Def_distribution_recentre_dilate}) and (\ref{Hypothese_C_etoile}), note that \[M_{2}\left(\mu_{0}^{K} \right) =  \left\langle \id^{2}, \mu_{0}^{K} \right\rangle {\color{black} \leqslant C_{2}^{\star}}.\]
 The announced result follows from Lemma \ref{Lem_Controle_M2_Trajectoire_Esperance}. \qedhere
\end{proof}

The goal of the following result is to bound the martingale bracket $\left\langle M^{K, P_{\id^{2}, 1}}\right\rangle_{t\wedge \check{\tau}^{K}}$. To do this, we exploit the martingale approximation of the moment of order 6 in order to control the dominant term $M_{4}\left(\mu^{K }_{s} \right) - M_{2}^{2} \left(\mu^{K }_{s} \right)$ in (\ref{Eq_Variation_Quad_M2}).

\begin{Lem}
There exists a constant $C>0$ such that for all $K \in \N^{\star}$, for all $t\geqslant 0$ the martingale bracket  $\left\langle M^{K, P_{\id^{2}, 1}}\right\rangle_{t}$ satisfies 
\begin{equation*}
\E\left(\left\langle M^{K, P_{\id^{2}, 1}} \right\rangle_{t\wedge\check{\tau}^{K}} \right) \leqslant \frac{C}{K^{2}{\color{black}\sigma_{K}^{2}}}\E\left(t\wedge\check{\tau}^{K} \right).
\end{equation*}
\label{Crochet_martingale_M2}
\end{Lem}

\begin{proof} From Lemma \ref{Cor_moment_ordre_2}  and \textsc{H\"{o}lder}'s inequalities $M_{2}^{2}(\mu) \leqslant M_{4}(\mu) \leqslant M_{6}^{2/3} \leqslant 1 + M_{6}(\mu)$ and $M_{5}\left(\mu \right) \leqslant M_{6}\left(\mu \right)^{5/6} \leqslant 1 + M_{6}\left(\mu \right)$, we deduce that there exists a constant $C>0$ such that for all $t \geqslant 0$ 
\begin{align*}
\left\langle M^{K, P_{\id^{2}, 1}}\right\rangle_{t\wedge \check{\tau}^{K}} & \leqslant \frac{C}{K^{2}{\color{black}\sigma_{K}^{2}}}\int_{0}^{t\wedge \check{\tau}^{K}}{\left(1 + M_{6}\left(\mu_{s}^{K} \right)\right) \left(1 + \frac{1}{K} \right) \dd s}. 
% \leqslant \frac{4\overline{b}}{K^{2}{\color{black}\sigma_{K}^{2}}}\left(1 + \frac{C}{K} \right)\int_{0}^{t}{M_{6}\left(\mu_{s \wedge \check{\tau}^{K}}^{K} \right) \dd s}.
\end{align*}
The announced result follows from \textsc{Fubini}'s theorem and Lemma \ref{Lem_Controle_M6_et_M4M2}. \qedhere
\end{proof}

\subsection{Convergence to $+\infty$ of the stopping time $\tau^{K}$ and support concentration property\label{Sous_Section_4_4_Convergence_tau}}

The main result of this section is the following convergence result for the stopping time $\tau^{K}$, defined by (\ref{Temps_arret_check_chapeau}), when $K \to + \infty$. 

\begin{Prop} Under {\rm Assumption (\ref{Hypothese_Gamme_sigma})}, $\tau^{K }$ converges in probability to $+\infty$ when $K \to + \infty$.
\label{Prop_Convergence_tau_et_theta}
\end{Prop}

The proof is based on the next lemma proved in Section \ref{Sous_section_4_5_Sortie_domaine} below.  Thanks to Proposition \ref{Prop_Convergence_tau_et_theta}, we deduce the support concentration property of $\left(\nu_{t/K{\color{black}\sigma_{K}^{2}}}^{K} \right)$ given by (\ref{Eq_Thm_Concentration_Support}). Indeed, with probability which tends to $1$ when $K\to + \infty$, $\Diam\left(\mu_{t}^{K} \right) \leqslant \frac{1}{{\color{black}\sigma_{K}} K^{\frac{3+\varepsilon}{2}}}$ for all $t \in [0, T]$. Hence, \[\Diam\left(\nu_{t/K{\color{black}\sigma_{K}^{2}}}^{K} \right) = {\color{black}\sigma_{K}}\sqrt{K}\Diam\left(\mu_{t}^{K} \right) \leqslant \frac{1}{K^{1+\frac{\varepsilon}{2}}}  \] and (\ref{Eq_Thm_Concentration_Support}) follows.

\begin{Lem} Under {\rm Assumption (\ref{Hypothese_Gamme_sigma})}, for all $T \geqslant 0$, \[\lim_{K\to +\infty}{\P\left(\widehat{\tau}^{K} \leqslant \check{\tau}^{K} \wedge T \right)} = \lim_{K \to + \infty}{\P\left(\sup_{0\leqslant t\leqslant T\wedge \check{\tau}^{K}}{M_{2}\left(\mu_{t}^{K } \right)} \geqslant K^{\varepsilon} \right)} = 0.\]
\label{Lem_Proba_sup_M2_K_epsilon}
\end{Lem}

\begin{proof}[Proof of Proposition \ref{Prop_Convergence_tau_et_theta}]
Note that for all $T\geqslant 0$,
\begin{align*}
\P\left(\check{\tau}^{K} < T \wedge \widehat{\tau}^{K} \right) \leqslant \P\left(\exists t < T \wedge \widehat{\tau}^{K}, \Diam\left(\Supp \mu_{t}^{K} \right) > \frac{1}{{\color{black}\sigma_{K}} K^{\frac{3+\varepsilon}{2}}} \right).
\end{align*}
As for all $t \geqslant 0$, for all $K \in \N^{\star}$,  $\Diam \left(\Supp \mu_{t}^{K} \right)^{2} \leqslant \left(2\max{\left\{\left|x \right| \left| \phantom{1^{1^{1}}} \hspace{-0.6cm} \right. x \in \Supp \mu_{t}^{K} \right\}}\right)^{2} \leqslant 4K M_{2}\left(\mu_{t}^{K}\right)$, we deduce that 
\begin{align*}
& \P\left(\exists t < T \wedge \widehat{\tau}^{K},  \Diam\left(\Supp \mu_{t}^{K} \right) > \frac{1}{{\color{black}\sigma_{K}} K^{\frac{3+\varepsilon}{2}}} \right) \\ 
& \hspace{4cm} \leqslant \P\left(\exists t < T \wedge,  \widehat{\tau}^{K} M_{2}\left(\mu_{t}^{K} \right) > \frac{1}{4{\color{black}\sigma_{K}^{2}}K^{4+\varepsilon}}  \right).
\end{align*}
According to (\ref{Hypothese_Gamme_sigma}), we deduce that the following inequality $4{\color{black}\sigma_{K}^{2}}K^{4 + \varepsilon} \leqslant 1/K^{\varepsilon}$ holds true when $K \to + \infty$ and ${\color{black}\sigma_{K}} \to 0$. Hence, we deduce that
\[\P\left(\exists t < T \wedge \widehat{\tau}^{K}, M_{2}\left(\mu_{t}^{K} \right) > \frac{1}{4{\color{black}\sigma_{K}^{2}}K^{4+\varepsilon}} \right) \leqslant \P\left(\exists t < T \wedge \widehat{\tau}^{K}, M_{2}\left(\mu_{t}^{K} \right) > K^{\varepsilon}  \right) = 0. \]
Now, for all $T > 0$, $\left\{\check{\tau}^{K} \geqslant T \wedge \widehat{\tau}^{K} \right\} = \left\{\check{\tau}^{K} \geqslant \widehat{\tau}^{K} \right\} \cup \left\{\check{\tau}^{K} \geqslant T \right\} $, and so
\begin{align*}
\P\left(\left\{\check{\tau}^{K} \geqslant T \wedge \widehat{\tau}^{K} \right\} \cap \left\{\tau^{K} < T \right\} \right) & \leqslant \P\left(\left[\left\{\widehat{\tau}^{K} < T \right\} \cap \left\{\widehat{\tau}^{K} \leqslant \check{\tau}^{K} \right\}\right] \cup \left\{\widehat{\tau}^{K} < T \leqslant \check{\tau}^{K} \right\} \right) \\ 
& \leqslant \P\left(\widehat{\tau}^{K} \leqslant T \wedge \check{\tau}^{K} \right).
\end{align*}
Note that  for all $T > 0$,
\begin{align*}
\P\left(\tau^{K} < T \right) & \leqslant \P\left(\check{\tau}^{K} < T \wedge \widehat{\tau}^{K} \right) + \P\left(\widehat{\tau}^{K} \leqslant T \wedge \check{\tau}^{K} \right).
\end{align*}
Hence, from Lemma \ref{Lem_Proba_sup_M2_K_epsilon}, we have for all $T\geqslant 0$,  $\lim_{K\to +\infty}{\P\left(\tau^{K} < T \right)} = 0$ and the announced result follows. \qedhere
\end{proof}

\subsection{Proof of Lemma \ref{Lem_Proba_sup_M2_K_epsilon}\label{Sous_section_4_5_Sortie_domaine}} %Problem of exit from a attracting domain
Note that if $M_{2}\left(\mu_{s}^{K} \right)$ becomes larger than $\overline{\theta}\overline{m}_{2}/\underline{b}$, the drift term in the right-hand side of (\ref{Inegalites_M2}) is then negative preventing $M_{2}\left(\mu_{s}^{K} \right)$ to become excessively large unless the martingale $M^{K, P_{\id^{2}, 1}}$ has an exceptional path with large increments on a small time interval. We then expect to have large deviations estimates and bounds on the time of exit from attractive domains for $M_{2}\left(\mu^{K} \right)$. However, we cannot consider establishing directly a large deviation principle on $M_{2}\left(\mu^{K} \right)$, in particular because of the slow-fast limit. The approach considered below is based on coupling arguments between $M_{2}\left(\mu^{K} \right)$ and a simpler process for which large deviations estimates are known.  \\

Let us introduce $u_{0} := 0$ and for all $\ell \in \N^{\star}$ the real number $u_{\ell} := 3^{\ell} K^{\frac{\varepsilon}{2}}$ and the interval $\III_{\ell} :=\left[u_{\ell - 1}, u_{\ell + 1}\right)$. We will look at the process $M_{2}\left(\mu^{K} \right)$ at successive exit times of $\left(\III_{\ell}\right)_{\ell \in \N^{\star}}$. We set $L_{0}^{K} := 1$ when $M_{2}\left(\mu_{0}^{K} \right) \in [0, 2u_{1} + 1]$ and $L_{0}^{K} := \ell_{0} \in \N^{\star}$ when $M_{2}\left(\mu_{0}^{K} \right) \in \left(2u_{\ell_{0} - 1} + 1, 2u_{\ell_{0}} + 1 \right]$. We also set $T_{0}^{K} :=0$ and by induction on $k\in \N^{\star}$, if $\left(L_{i}^{K} \right)_{1\leqslant i \leqslant k}$ and $\left(T_{i}^{K} \right)_{1\leqslant i \leqslant k}$ are constructed, 
\[T_{k+1}^{K} := \inf\left\{t \geqslant T_{k}^{K} \left| \phantom{1^{1^{1^{1}}}} \hspace{-0.7cm} \right.  M_{2}\left(\mu_{t}^{K} \right) \notin \III_{L_{k}^{K}} \right\} \]
and \[ L_{k+1}^{K} := \min\left\{\ell \in \N \left| \phantom{1^{1^{1^{1}}}} \hspace{-0.7cm} \right. M_{2}\left(\mu_{T_{k+1}^{K}}^{K} \right) \leqslant 2u_{\ell} + 1 \right\}. \]
Note that from (\ref{Hypothese_C_etoile}), for $K$ large enough, $L_{0}^{K} = 1$ because $M_{2}\left(\mu_{0}^{K} \right) \leqslant C_{2}^{\star}$. Since $\left(\mu_{t}^{K} \right)_{t\geqslant 0}$ is a pure jump process, note that $\left(M_{2}\left(\mu_{t}^{K} \right) \right)_{t\geqslant 0}$ is also a pure jump process. The definition of $u_{\ell}$ is motivated by the fact that, for all $k \geqslant 0$, $L_{k+1}^{K} \leqslant L_{k}^{K} + 1$, which follows from the next lemma.
\begin{Lem}
Let $\vartheta_{1}$ be the first jump time of $\mu^{K}$. Then, for $K$ large enough \[M_{2}\left(\mu_{\vartheta_{1}}^{K} \right) \leqslant 2 M_{2}\left(\mu_{0}^{K} \right) + 1.\]
In particular, for all $t >0$, $\Delta M_{2}\left(\mu_{t}^{K} \right) := M_{2}\left(\mu_{t}^{K} \right) - M_{2}\left(\mu_{t^{-}}^{K} \right) \leqslant M_{2}\left(\mu_{t^{-}}^{K} \right) + 1$.
\label{Lem_Accroissements_M2}
\end{Lem}

\begin{proof}
Let us consider $\mu_{0}^{K} := \frac{1}{K}\sum_{k \, = \, 1}^{K}{\delta_{x_{k}}}$ and $y = \left(y_{k} \right)_{1\leqslant k \leqslant K}$ defined by  
\[y = \left\{\begin{array}{ll}
\left(x_{1}, \cdots, x_{i-1}, x_{i} + \frac{H}{\sqrt{K}}, x_{i+1}, \cdots, x_{K} \right) & {\rm on} \ E_{1} , \\
\left(x_{1}, \cdots, x_{j-1}, x_{i}, x_{j+1}, \cdots, x_{K}\right) & {\rm on} \ E_{2},
\end{array} \right.  \] where $H$ follows the mutation law $m(x_{i}, \dd h)$, $E_{1}$ is the event ``a mutation occurs at time $\vartheta_{1}$'' and $E_{2}$ is the event ``a resampling between $i$ and $j$ occurs at time $\vartheta_{1}$''. Note that $\mu_{\vartheta_{1}}^{K} = \tau_{-\left\langle \id, \frac{1}{K} \sum_{k\, = \, 1}^{K}{\delta_{y_{k}}} \right\rangle } \sharp \left(\frac{1}{K} \sum_{k\, = \, 1}^{K}{\delta_{y_{k}}}  \right)$. Hence, noting that for all $k \in \{1, \cdots, K \}$, \[\left|x_{k} \right| \leqslant \sqrt{K}\left(1 + M_{2}\left(\mu_{0}^{K} \right) \right)\]  and using Assumption \textbf{(A2)}, for $K$ large enough we have that
\begin{align*}
M_{2}\left(\mu_{\vartheta_{1}}^{K} \right) &  \leqslant M_{2}\left(\frac{1}{K}\sum_{i\, = \, 1}^{K}{\delta_{y_{i}}}  \right)  = \frac{1}{K}\sum_{i\, = \, 1}^{K}{y_{i}^{2}} \\
&  \leqslant  \left\{\begin{array}{ll}
\cfrac{1}{K}{\displaystyle{\sum\limits_{k\, = \, 1}^{K}{x_{k}^{2}}}} + \cfrac{2A_{m}x_{k}}{K^{\frac{3}{2}}} + \cfrac{A_{m}^{2}}{K^{2}} & \quad {\rm{on}} \quad \ E_{1}, \vspace{0.2cm} \\
\cfrac{1}{K}{\displaystyle{\sum_{k\, = \, 1}^{K}{x_{k}^{2}}}} + \cfrac{x_{j}^{2} - x_{i}^{2}}{K} & \quad {\rm{on}} \quad \ E_{2},
\end{array} \right. \\
& \leqslant \left\{\begin{array}{ll}
M_{2}\left(\mu_{0}^{K} \right) + \cfrac{2A_{m}\left(1 + M_{2}\left(\mu_{0}^{K} \right)\right)}{K}  + \cfrac{A_{m}^{2}}{K^{2}} & \quad {\rm{on}} \quad E_{1}, \vspace{0.2cm} \\
2 M_{2}\left(\mu_{0}^{K} \right) & \quad {\rm{on}} \quad E_{2},
\end{array} \right. \\
& \leqslant 2M_{2}\left(\mu_{0}^{K} \right) + 1. 
\end{align*}
\end{proof}

\indent Each of the previous steps will be called \emph{transitions of $M_{2}\left(\mu_{t}^{K} \right)$}. The idea of the proof of Lemma \ref{Lem_Proba_sup_M2_K_epsilon} is to estimate the transition probabilities of the sequence $\left(L_{k}^{K}\right)_{k \in \N}$ in order to construct a coupling between $\left(L_{k}^{K}\right)_{k \in \N}$ and a biased random walk  on $\N^{\star}$ reflected in $1$. \\ 
\indent At time $\widehat{\tau}^{K}$, $M_{2}\left(\mu_{\widehat{\tau}^{K}}^{K} \right) \in \III_{\left\lfloor \frac{\varepsilon}{2\log(3)}\log(K) \right\rfloor}$ where $\lfloor x \rfloor$ is the lower integer part of $x$. So, the number of transitions of $M_{2}\left(\mu_{t}^{K} \right)$ before $\widehat{\tau}^{K}$ is greater than or equal to $T_{k_{0}}^{K}$ where $k_{0}$ is the first integer such that $L_{k_{0}}^{K} = \left\lfloor \frac{\varepsilon}{2\log(3)}\log(K) \right\rfloor$. \\
\indent By using estimates on the number of steps that a biased random walk takes to reach $\left\lfloor \frac{\varepsilon}{2\log(3)}\log(K) \right\rfloor$ and estimates on the durations between two transitions $T_{k+1}^{K}-T_{k}^{K}$, we will deduce a lower bound, exponential in $K$, on $\widehat{\tau}^{K}$. An additional difficulty comes from the fact that the previous coupling argument is only valid up to time $\check{\tau}^{K}$. Hence we will construct a coupling that takes into account the possibility that $\check{\tau}^{K}$ happens during each transition step. The proof is divided into four steps: in Section \ref{Sous-sous-section_4_5_1_One-step_transitions}, we characterise the behaviour of the first transition step; in Section \ref{Sous-sous-section_4_5_2_Construction_of_the_coupling}, the proposed coupling is constructed; in Section \ref{Sous-sous-section_4_5_3_Attractive_domain_exit_estimates_for_random_walks}, we give estimates on the first exit time from an attractive domain for random walks; finally, we conclude in Section \ref{Sous-sous-section_4_5_4_Conclusion}.

\subsubsection{One-step transitions\label{Sous-sous-section_4_5_1_One-step_transitions}}
 In this section we look at only one transition: we suppose that $M_{2}\left(\mu_{0}^{K} \right) \in \left[2u_{\ell - 1} + 1, 2u_{\ell} + 1 \right)$ for $\ell \in \N^{\star}$ fixed and we look for bounds on the first transition probabilities of $M_{2}\left(\mu_{t}^{K} \right)$ and on time $T_{1}^{K}$.  The main result is the following.

 \begin{Lem}  For all $\ell \in \N$, for all $\mu_{0}^{K } \in \MM_{1, K}^{c, 2}(\R)$ such that $$M_{2}\left(\mu_{0}^{K } \right) \in \left[2u_{\ell-1}+1, 2u_{\ell} + 1 \right),$$  there exist a constant $C>0$, $K_{0} \in \N^{\star}$ such that for all $K \geqslant K_{0}$, we have
 \begin{equation}
 \P_{\mu_{0}^{K}}\left(\left[\left\{L_{1}^{K} = \ell-1 \right\} \cap \left\{\check{\tau}^{K} > T_{1}^{K} \geqslant K{\color{black}\sigma_{K}^{2}} \right\}\right] \cup \left\{T_{1}^{K} \geqslant \check{\tau}^{K}  \right\}\right) \geqslant  \frac{1}{2} + \eta_{K},
 \label{Eq_one-step_transition}
 \end{equation}
 where $\eta_{K} := \frac{1}{2} - \frac{C}{K^{\varepsilon/2}}$. 
  \label{Lem_One-step_transition}
\end{Lem} 

The proof of the previous result is based on Corollary \ref{Cor_Ineg_Trajectoires_M2} which is obtained as a straightforward consequence of Lemma \ref{Lem_Controle_M2_Trajectoire_Esperance}. 

\begin{Cor} Let $\ell \in \N^{\star}$ be fixed. There exists $K_{0} \in \N^{\star}$ large enough such that for all $K \geqslant K_{0}$, for all $\mu_{0}^{K} \in \MM_{1, K}^{c, 2}(\R)$ such that $M_{2}\left(\mu_{0}^{K}\right) \in \left[2u_{\ell - 1}+1, 2u_{\ell} + 1 \right)$, we have for all $t \leqslant T_{1}^{K}$ that
\begin{align*}
& 2u_{\ell-1} +1 - \frac{5\overline{b}u_{\ell + 1} + \underline{\theta}\underline{m}_{2}}{2K^{2}{\color{black}\sigma_{K}^{2}}}t\wedge \check{\tau}^{K} + M_{t\wedge \check{\tau}^{K}}^{K, P_{\id^{2}, 1}} \\ 
& \hspace{4cm} \leqslant M_{2}\left(\mu_{t\wedge \check{\tau}^{K}}^{K} \right) \leqslant 2u_{\ell} + 1 - \frac{\underline{b}u_{\ell - 1}}{K^{2}{\color{black}\sigma_{K}^{2}}}t\wedge \check{\tau}^{K} + M_{t\wedge \check{\tau}^{K}}^{K, P_{\id^{2}, 1}}.
\end{align*}
\label{Cor_Ineg_Trajectoires_M2}
\end{Cor}

\begin{proof}[Proof of Lemma \ref{Lem_One-step_transition}] By passing on the complementary of (\ref{Eq_one-step_transition}), it is equivalent to prove that
\[\P_{\mu_{0}^{K}}\left(\left[\left\{L_{1}^{K} = \ell + 1 \right\} \cup \left\{T_{1}^{K} < K{\color{black}\sigma_{K}^{2}} \right\} \right] \cap \left\{T_{1}^{K} < \check{\tau}^{K} \right\} \right)< \frac{1}{2} - \eta_{K}. \]
As for all events $A, B$, $\P\left(A \cup B \right) = \P\left(A\cap B^{c} \right) + \P(B)$, we have 
\begin{equation}
\begin{aligned}
& \P_{\mu_{0}^{K}}\left(\left[\left\{L_{1}^{K} = \ell + 1 \right\} \cup \left\{T_{1}^{K} < K{\color{black}\sigma_{K}^{2}} \right\} \right] \cap \left\{T_{1}^{K} < \check{\tau}^{K} \right\} \right) \\ 
& \hspace{1cm} = \P_{\mu_{0}^{K}}\left(\left\{T_{1}^{K} < \check{\tau}^{K} \right\} \cap \left\{L_{1}^{K} = \ell + 1 \right\} \cap \left\{T_{1}^{K} \geqslant K{\color{black}\sigma_{K}^{2}} \right\}  \right) \\ 
& \hspace{2cm} + \P_{\mu_{0}^{K}}\left(\left\{T_{1}^{K} < \check{\tau}^{K} \right\} \cap \left\{T_{1}^{K} < K{\color{black}\sigma_{K}^{2}} \right\} \right) \\
& \hspace{1cm} \leqslant \P_{\mu_{0}^{K}}\left(\left\{T_{1}^{K} < \check{\tau}^{K} \right\} \cap \left\{L_{1}^{K} = \ell + 1 \right\}   \right) + \P_{\mu_{0}^{K}}\left(\left\{T_{1}^{K} < \check{\tau}^{K} \right\} \cap \left\{T_{1}^{K} < K{\color{black}\sigma_{K}^{2}} \right\} \right).
\end{aligned}
\label{Ineq_One_step}
\end{equation}

\textit{Step 1. Control of the first right-hand term of (\ref{Ineq_One_step}).} Let us consider the martingale $\left({\rm Mart}_{t}^{K, +} \right)_{t\geqslant 0}$ defined by \[{\rm Mart}_{t}^{K, +} := \left\{\begin{array}{lll} 
M_{t}^{K, P_{\id^{2}, 1}} & {\rm{if}} & t\leqslant \check{\tau}^{K} \\
M_{\check{\tau}^{K}}^{K, P_{\id^{2}, 1}} + \frac{1}{K{\color{black}\sigma_{K}}}B_{t-\check{\tau}^{K}} & {\rm{if}} & t> \check{\tau}^{K}
\end{array}
\right. \] 
where $\left(B_{t} \right)_{t\geqslant 0}$ is a standard Brownian motion independent of $\left(\mu_{t}^{K}\right)_{t\geqslant 0}$. Note that from Lemma \ref{Crochet_martingale_M2}, there exists a constant $C>0$ such that for all $t\geqslant 0$ we have 
\begin{equation}
\E\left(\left\langle {\rm Mart}^{K, +} \right\rangle_{t} \right) \leqslant \frac{C}{K^{2}{\color{black}\sigma_{K}^{2}}}t.
\label{Eq_Mart_K_+}
\end{equation}
Let us consider the process $\left(M_{2}^{K, +}(t) \right)_{t\geqslant 0}$ defined by \[M_{2}^{K, +}(t) := \left\{\begin{array}{lll} 
M_{2}\left(\mu_{t}^{K} \right) & {\rm if} & t\leqslant \check{\tau}^{K} \\
M_{2}\left(\mu_{\check{\tau}^{K}}^{K} \right) - \frac{\underline{b}u_{\ell - 1}}{K^{2}{\color{black}\sigma_{K}^{2}}}\left(t - \check{\tau}^{K} \right) + \left({\rm Mart}_{t}^{K, +}  - {\rm Mart}_{\check{\tau}^{K}}^{K, +} \right) &  {\rm{if}} & t >  \check{\tau}^{K}
\end{array}
\right. .\]
Note that from Corollary \ref{Cor_Ineg_Trajectoires_M2}, for all $t \geqslant 0$, 
\begin{equation}
2u_{\ell - 1} + 1 - \frac{5\overline{b}u_{\ell + 1} + \underline{\theta}\underline{m}_{2}}{2K^{2}{\color{black}\sigma_{K}^{2}}}t + {\rm Mart}_{t}^{K, +} \leqslant M_{2}^{K, +}(t) \leqslant 2u_{\ell} + 1 - \frac{\underline{b}u_{\ell -1}}{K^{2}{\color{black}\sigma_{K}^{2}}}t + {\rm Mart}_{t}^{K, +}.
\label{Eq_M2_K_+}
\end{equation}
By \textsc{Girsanov}'s theorem, the process $\left(E_{t}^{K, +} \right)_{t\geqslant 0}$ defined for all $t\geqslant 0$ by \[E_{t}^{K, +} := \exp\left( {\rm Mart}_{t}^{K, +} - \frac{1}{2}\left\langle {\rm Mart}^{K, +} \right\rangle_{t} \right)\]
is a local martingale. Let $t_{0}^{K} := \frac{6 K^{2}{\color{black}\sigma_{K}^{2}}}{\underline{b}}$ be fixed and let us consider 
\begin{align*}
T_{1}^{K, +}  & := \inf\left\{t \geqslant 0 \left| \phantom{1^{1^{1^{1}}}} \hspace{-0.7cm} \right.  M_{2}^{K, +}\left(t \right) \not \in \III_{\ell} \right\}, \\
L_{1}^{K, +}  & := \min\left\{\ell \in \N \left| \phantom{1^{1^{1^{1}}}} \hspace{-0.7cm} \right. M_{2}^{K, +}\left(T_{1}^{K, +}\right)\leqslant 2u_{\ell} - 1  \right\}.
\end{align*}
 From (\ref{Eq_M2_K_+}) and denoting $\Gamma := \left\{\left\langle {\rm Mart}^{K, +}  \right\rangle_{T_{1}^{K, +}\wedge t_{0}^{K}} \leqslant \frac{2\underline{b}u_{\ell - 1}}{K^{2} {\color{black}\sigma_{K}^{2}}}T_{1}^{K,+}\wedge t_{0}^{K} \right\}$, we have 
\begin{align*}
1 & = \E_{\mu_{0}^{K}}\left(E_{T_{1}^{K, +}\wedge t_{0}^{K}}^{K, +} \right)  \geqslant \E_{\mu_{0}^{K}}\left(E_{T_{1}^{K, +}\wedge t_{0}^{K}}^{K, +}\II_{\Gamma} \right) \\
& \geqslant \exp\left(-2 u_{\ell} - 1 \right) \E_{\mu_{0}^{K}}\left(\exp\left(  M_{2}^{K, +}\left(T_{1}^{K, +}\wedge t_{0}^{K} \right) +   \frac{\underline{b}u_{\ell -1}}{K^{2}{\color{black}\sigma_{K}^{2}}}T_{1}^{K, +}\wedge t_{0}^{K} \right. \right. \\
& \hspace{8cm} - \left.\left. \frac{1}{2}\left\langle {\rm Mart}^{K, +} \right\rangle_{T_{1}^{K, +}\wedge t_{0}^{K}} \right)\II_{\Gamma}\right) \\
& \geqslant  \exp\left(-2  u_{\ell} -1 \right) \E_{\mu_{0}^{K}}\left(\exp\left(  M_{2}^{K, +}\left(T_{1}^{K, +} \wedge t_{0}^{K} \right) \right)\II_{\Gamma}\right) \\
& \geqslant \exp\left(-2  u_{\ell} -1 \right)\left[\exp\left(  u_{\ell+1} \right)\P_{\mu_{0}^{K}}\left(\left\{L_{1}^{K, +} = \ell + 1\right\} \cap \Gamma \cap \left\{T_{1}^{K, +} \leqslant t_{0}^{K} \right\}  \right) \right. \\
& \hspace{4cm} +\left. \P_{\mu_{0}^{K}}\left(\left\{L_{1}^{K, +} \leqslant \ell - 1\right\} \cap \Gamma \cap \left\{T_{1}^{K, +} \leqslant t_{0}^{K} \right\}  \right) \right].
\end{align*}
As for all events $A, B, C$, $\P\left(A\cap B \cap C \right) \geqslant \P(A) - \P\left(A \cap B^{c} \right) - \P\left(A \cap C^{c} \right)$, and denoting $q := \P_{\mu_{0}^{K}}\left(L_{1}^{K, +}  = \ell + 1\right) $ we have that
\begin{align*}
1  & \geqslant \exp\left(u_{\ell} -1 \right)\left[q - \P_{\mu_{0}^{K}}\left(\left\{ L_{1}^{K, +}  = \ell + 1 \right\} \cap \Gamma^{c}\right) \right. \\
& \hspace{4cm}-\left. \P_{\mu_{0}^{K}}\left(\left\{ L_{1}^{K, +}  = \ell + 1 \right\} \cap \left\{T_{1}^{K, +} > t_{0}^{K} \right\}\right) \right] \\
& \quad + \exp\left(-2u_{\ell} - 1 \right)\left[1-q - \P_{\mu_{0}^{K}}\left(\left\{ L_{1}^{K, +}  \leqslant \ell - 1 \right\} \cap \Gamma^{c}\right) \right.\\
& \hspace{4cm} \left.-\P_{\mu_{0}^{K}}\left(\left\{ L_{1}^{K, +}  \leqslant \ell - 1 \right\} \cap \left\{T_{1}^{K, +} > t_{0}^{K} \right\}\right) \right] \\
& \geqslant \left(\exp\left(u_{\ell} - 1 \right) - \exp\left(-2u_{\ell} -1 \right) \right)q + \exp\left(-2u_{\ell} -1 \right) \\
& \hspace{0.75cm} - \left( \exp\left(u_{\ell} -1 \right) + \exp\left(-2u_{\ell} -1 \right)\right)\left[\P_{\mu_{0}^{K}}\left(\Gamma^{c} \right)  + \P_{\mu_{0}^{K}}\left(T_{1}^{K, +} > t_{0}^{K} \right) \right].
\end{align*}
Then, considering $s_{0}^{K} := \frac{u_{\ell - 1}K^{2}{\color{black}\sigma_{K}^{2}}}{5\overline{b}u_{\ell + 1} + \underline{\theta}\underline{m}_{2}} < t_{0}^{K}$ be fixed, we deduce that 
\begin{align*}
q & \leqslant \frac{1-\exp\left(-2u_{\ell} - 1 \right)}{\exp\left(u_{\ell} -1 \right) - \exp\left(-2u_{\ell} -1 \right)}  \\
& \hspace{1cm} + \frac{\exp\left(u_{\ell}  \right) +\exp\left(-2u_{\ell}  \right)}{\exp\left(u_{\ell}  \right) - \exp\left(-2u_{\ell} \right)} \left[\P_{\mu_{0}^{K}}\left(\Gamma^{c}  \right) + \P_{\mu_{0}^{K}}\left(T_{1}^{K, +} > t_{0}^{K} \right) \right]\\
& \leqslant \exp\left(-u_{\ell} + 1\right) \times \frac{1 - \exp\left(-2u_{\ell} -1 \right)}{1 - \exp\left(-u_{\ell + 1} \right)} +  \frac{\exp\left(u_{\ell}  \right) + \exp\left(-2u_{\ell} \right)}{\exp\left(u_{\ell} \right) - \exp\left(-2u_{\ell} \right)}\\
& \hspace{1.5cm} \times \left[\P_{\mu_{0}^{K}}\left(\Gamma^{c} \cap \left\{T_{1}^{K, +} \geqslant s_{0}^{K} \right\} \right) + \P_{\mu_{0}^{K}}\left(T_{1}^{K, +} < s_{0}^{K} \right) + \P_{\mu_{0}^{K}}\left(T_{1}^{K, +} > t_{0}^{K} \right) \right].
\end{align*}
Denoting $\lceil a \rceil $ the upper integer part of $a \in \R$, we obtain on the one hand that there exists a constant $C_{1}>0$ such that
\begin{align*}
& \P_{\mu_{0}^{K}}\left(\Gamma^{c}\cap \left\{T_{1}^{K, +} \geqslant s_{0}^{K} \right\} \right) \\
& \hspace{0.0cm}  = \P_{\mu_{0}^{K}}\left(\left\{ \left\langle  {\rm Mart}^{K, +} \right\rangle_{T_{1}^{K,+}\wedge t_{0}^{K}} >  \frac{2 \underline{b}u_{\ell -1}}{K^{2} {\color{black}\sigma_{K}^{2}}} T_{1}^{K,+} \wedge t_{0}^{K} \right\} \cap \left\{ T_{1}^{K, +} \geqslant s_{0}^{K}\right\} \right) \\ 
&  \hspace{0.0cm} = \hspace{-0.2cm} \sum\limits_{ k \, = \, 1}^{\left\lceil \frac{t_{0}^{K}}{s_{0}^{K}} \right\rceil -1}{\hspace{-0.3cm}\P_{\mu_{0}^{K}}\left(\left\{\left\langle {\rm Mart}^{K, +} \right\rangle_{T_{1}^{K,+}\wedge t_{0}^{K}} >  \frac{2 \underline{b}u_{\ell-1}}{ K^{2} {\color{black}\sigma_{K}^{2}}} T_{1}^{K,+} \wedge t_{0}^{K} \right\} \cap \left\{k s_{0}^{K} \leqslant T_{1}^{K,+} \leqslant (k+1)s_{0}^{K} \right\}\right)} \\
& \hspace{1.5cm} +  \P_{\mu_{0}^{K}}\left(\left\{\left\langle {\rm Mart}^{K, +} \right\rangle_{T_{1}^{K,+}\wedge t_{0}^{K}} >  \frac{2 \underline{b}u_{\ell-1}}{ K^{2} {\color{black}\sigma_{K}^{2}}} T_{1}^{K,+} \wedge t_{0}^{K} \right\} \cap \left\{T_{1}^{K, +} \geqslant \lceil t_{0}^{K}/s_{0}^{K} \rceil s_{0}^{K} \right\} \right)  \\
&  \hspace{0.0cm} \leqslant \sum\limits_{k \, = \, 1}^{\left\lceil \frac{t_{0}^{K}}{s_{0}^{K}} \right\rceil -1}{\hspace{-0.3cm}\P_{\mu_{0}^{K}}\left(\left\langle  {\rm Mart}^{K, +} \right\rangle_{(k+1)s_{0}^{K}} >  \frac{2 \underline{b}u_{\ell -1}}{K^{2} {\color{black}\sigma_{K}^{2}}}k s_{0}^{K}\right) + \P_{\mu_{0}^{K}}\left(\left\langle  {\rm Mart}^{K, +} \right\rangle_{t_{0}^{K}} >  \frac{2 \underline{b}u_{\ell -1}}{K^{2} {\color{black}\sigma_{K}^{2}}}t_{0}^{K}\right)} \\
& \hspace{0.0cm} \leqslant K^{2}{\color{black}\sigma_{K}^{2}}\left(\sum\limits_{k \, = \, 1}^{\lceil t_{0}^{K}/s_{0}^{K} \rceil -1}{\frac{\E\left(\left\langle {\rm Mart}^{K, +}\right\rangle_{(k+1)s_{0}^{K}} \right)}{2\underline{b}u_{\ell -1}ks_{0}^{K}} + \frac{\E\left(\left\langle {\rm Mart}^{K, +}\right\rangle_{t_{0}^{K}} \right)}{2\underline{b}u_{\ell -1}t_{0}^{K}}}\right) \\ 
& \hspace{0.0cm} \leqslant \frac{C_{1}}{K^{\frac{\varepsilon}{2}}}
\end{align*}
where we use \textsc{Markov}'s inequality in the second inequality and (\ref{Eq_M2_K_+}) in the last inequality. On the other hand, we control $\P_{\mu_{0}^{K}}\left(T_{1}^{K, +} < s_{0}^{K} \right)$ as follows: 
\begin{align*}
\P_{\mu_{0}^{K}}\left(T_{1}^{K, +} < s_{0}^{K} \right)&  \leqslant \P_{\mu_{0}^{K}}\left(\left\{T_{1}^{K, +} < s_{0}^{K} \right\} \cap \left\{\sup_{0\leqslant t \leqslant s_{0}^{K}}{M_{2}^{K, +}(t)} \geqslant u_{\ell +1} \right\} \right) \\
& \hspace{1cm} + \P_{\mu_{0}^{K}}\left(\left\{T_{1}^{K, +} < s_{0}^{K} \right\} \cap \left\{\inf_{0\leqslant t \leqslant s_{0}^{K}}{M_{2}^{K, +}(t)} < u_{\ell -1} \right\} \right).
\end{align*}

Let us consider $K_{1} \in \N^{\star}$ independent of $\ell$ satisfying $K_{1}^{\varepsilon/2} \geqslant 2$. There exists a constant $C_{1}>0$ such that for all $K \geqslant K_{1}$, we have that 
\begin{align*}
& \P_{\mu_{0}^{K}}\left(\left\{T_{1}^{K, +} < s_{0}^{K} \right\}  \cap \left\{\sup_{0\leqslant t \leqslant s_{0}^{K}}{M_{2}^{K, +}(t)} \geqslant u_{\ell +1} \right\} \right) \\ 
& \hspace{1.5cm} \overset{}{\leqslant} \P_{\mu_{0}^{K}}\left(\left\{T_{1}^{K, +} < s_{0}^{K} \right\}  \cap \left\{\sup_{0\leqslant t \leqslant s_{0}^{K}}{\left\{2 u_{\ell} + 1 - \frac{\underline{b}u_{\ell - 1}}{K^{2}{\color{black}\sigma_{K}^{2}}}t + {\rm Mart}_{t}^{K, +}  \right\} \geqslant u_{\ell +1} } \right\} \right) \\
& \hspace{1.5cm} \leqslant \P_{\mu_{0}^{K}}\left(\left\{T_{1}^{K, +} < s_{0}^{K} \right\}  \cap \left\{\sup_{0\leqslant t \leqslant s_{0}^{K}}{{\rm Mart}_{t}^{K, +}} \geqslant u_{\ell} - 1 \right\} \right) \\
& \hspace{1.5cm} \leqslant \P_{\mu_{0}^{K}}\left(\sup_{0\leqslant t \leqslant s_{0}^{K}}{\left|{\rm Mart}_{t}^{K, +} \right|} \geqslant \frac{3^{\ell}K^{\frac{\varepsilon}{2}}}{2} \right) \\
& \hspace{1.5cm} \leqslant 4 {\color{black}\E_{\mu_{0}^{K}}}\left(\left\langle {\rm Mart}^{K, +} \right\rangle_{s_{0}^{K}} \right) \times \frac{4}{3^{2\ell}K^{\varepsilon}} \\
& \hspace{1.5cm} \leqslant  \frac{C_{1}}{K^{\varepsilon}}
\end{align*}
where we use (\ref{Eq_M2_K_+}) in the first inequality, \textsc{Doob}'s maximal inequality in the fourth inequality and (\ref{Eq_Mart_K_+}) in the last inequality. Now, let us consider $K_{2} \in \N^{\star}$ independent of $\ell$ satisfying $K_{2}^{\varepsilon/2} \geqslant 4$. Using again (\ref{Eq_M2_K_+}), \textsc{Doob}'s maximal inequality, (\ref{Eq_Mart_K_+}), and the definition of $s_{0}^{K}$, we deduce that there exists a constant $C_{2}>0$ such that for all  $K \geqslant K_{2}$, we have that
\begin{align*}
& \P_{\mu_{0}^{K}}\left(\left\{T_{1}^{K, +} < s_{0}^{K} \right\} \cap \left\{\inf_{0\leqslant t \leqslant s_{0}^{K}}{M_{2}^{K, +}(t)} < u_{\ell -1} \right\} \right) \\ 
& \hspace{0.0cm} \leqslant \P_{\mu_{0}^{K}}\left(\left\{T_{1}^{K, +} < s_{0}^{K} \right\} \cap \left\{\inf_{0\leqslant t \leqslant s_{0}^{K}}{\left\{2u_{\ell -1} - 1 - \frac{5\overline{b}u_{\ell + 1} + \underline{\theta}\underline{m}_{2}}{2K^{2}{\color{black}\sigma_{K}^{2}}}t + {\rm Mart}^{K, +}_{t} \right\}} < u_{\ell - 1} \right\} \right) \\
& \hspace{0.0cm} \leqslant \P_{\mu_{0}^{K}}\left(\left\{T_{1}^{K, +} < s_{0}^{K} \right\} \cap \left\{\inf_{0\leqslant t\leqslant s_{0}^{K}}{{\rm Mart}_{t}^{K, +}} < - \frac{u_{\ell - 1}}{2} + 1 \right\} \right) \\
& \hspace{0.0cm} \leqslant \P_{\mu_{0}^{K}}\left(\sup_{0\leqslant t\leqslant s_{0}^{K}}{\left|{\rm Mart}_{t}^{K, +}\right|} >  \frac{u_{\ell - 1}}{2} - 1  \right) \\
& \hspace{0.0cm} \leqslant \frac{C_{2}}{K^{\varepsilon}}.
\end{align*}
Hence, for all $K \geqslant \max\left\{K_{1}, K_{2} \right\}$, $ \P_{\mu_{0}^{K}}\left(T_{1}^{K, +} < s_{0}^{K} \right) \leqslant \frac{C_{1} + C_{2}}{K^{\varepsilon}}$. Finally, from (\ref{Eq_M2_K_+}) and \textsc{Markov}'s inequality and the definition of $t_{0}^{K}$, note that there exists a constant $C_{3} >0$ such that 
\begin{align*}
\P_{\mu_{0}^{K}}\left(T_{1}^{K, +} > t_{0}^{K} \right) & \leqslant \P_{\mu_{0}^{K}}\left(\left\{\sup_{0\leqslant t \leqslant t_{0}^{K}}{M_{2}^{K, +}(t)} < u_{\ell + 1} \right\} \cap \left\{\inf_{0\leqslant t \leqslant t_{0}^{K}}{M_{2}^{K, +}(t)} \geqslant u_{\ell -1} \right\} \right) \\
& \leqslant \P_{\mu_{0}^{K}}\left({\rm Mart}_{t_{0}^{K}}^{K, +} \geqslant u_{\ell - 1} - 2u_{\ell} - 1 + \frac{\underline{b}u_{\ell - 1}}{K^{2}{\color{black}\sigma_{K}^{2}}}t_{0}^{K} \right)  \\
& \leqslant \frac{\E_{\mu_{0}^{K}}\left(\left\langle  {\rm Mart}^{K, +}\right\rangle_{t_{0}^{K}}  \right)}{\left(u_{\ell-1} -1\right)^{2}} \\
& \leqslant \frac{C_{3}}{K^{\varepsilon}}.
\end{align*}
Hence, for all $K \geqslant \max{\left\{K_{1}, K_{2} \right\}}$, we deduce that $q \leqslant \varepsilon_{K}$ where 
\[\varepsilon_{K} := \exp\left(-u_{\ell} + 1\right) \times \frac{1 - \exp\left(-2u_{\ell} -1 \right)}{1 - \exp\left(-u_{\ell + 1} \right)} +  \frac{\exp\left(u_{\ell} \right) + \exp\left(-2u_{\ell}  \right)}{\exp\left(u_{\ell} \right) - \exp\left(-2u_{\ell} \right)} \times \frac{C_{1}}{K^{\frac{\varepsilon}{2}}} \]
and thus, for all $K \geqslant \max{\left\{K_{1}, K_{2} \right\}}$, \[\P_{\mu_{0}^{K}}\left(\left\{ T_{1}^{K} < \check{\tau}^{K} \right\} \cap \left\{L_{1}^{K} = \ell + 1 \right\} \right) = \P_{\mu_{0}^{K}}\left(\left\{ T_{1}^{K,+} < \check{\tau}^{K} \right\} \cap \left\{L_{1}^{K} = \ell + 1 \right\} \right) \leqslant q \leqslant \varepsilon_{K}.  \]

\textit{Step 2. Control of the second right-hand term of (\ref{Ineq_One_step}).} Similarly, let us consider $K_{3} \in \N^{\star}$ independent of $\ell$ satisfying $K_{3}^{\varepsilon/2} \geqslant 2$. Using Corollary \ref{Cor_Ineg_Trajectoires_M2}, \textsc{Doob}'s maximal inequality and Lemma \ref{Crochet_martingale_M2} in similar way to Step 1, there exists a constant $C_{4}>0$ such that for all $K \geqslant K_{3}$ we have that 
\begin{align*}
\P_{\mu_{0}^{K}}\left(\left\{T_{1}^{K} < \check{\tau}^{K} \right\} \cap \left\{\sup_{0\leqslant t \leqslant K{\color{black}\sigma_{K}^{2}} \wedge \check{\tau}^{K}}{M_{2}\left(\mu_{t}^{K} \right)}  \geqslant u_{\ell + 1} \right\} \right) & \leqslant \frac{16\E_{\mu_{0}^{K}}\left(\left\langle M^{K, P_{\id^{2}, 1}}\right\rangle_{K{\color{black}\sigma_{K}^{2}}\wedge \check{\tau}^{K}} \right)}{3^{2\ell}K^{\varepsilon}}   \\
& \leqslant \frac{C_{4}}{K^{1+\varepsilon}}.
\end{align*}
Let us consider $K_{4} \in \N^{\star}$ independent of $\ell$ satisfying the relation  $u_{\ell -1} > 1 + \frac{5\overline{b}u_{\ell + 1} + \underline{\theta}\underline{m}_{2}}{2 K_{4}}$.
As previously, we establish that there exists a constant $C_{5} >0$ such that for all $K \geqslant K_{4}$ we have that 
\begin{align*}
 \P_{\mu_{0}^{K}}\left(\left\{T_{1}^{K} < \check{\tau}^{K} \right\} \cap \left\{\inf_{0\leqslant t \leqslant K{\color{black}\sigma_{K}^{2}} \wedge \check{\tau}^{K}}{M_{2}\left(\mu_{t}^{K} \right)} < u_{\ell - 1} \right\} \right) & \leqslant \frac{4\E_{\mu_{0}^{K}}\left(\left\langle M^{K, P_{\id^{2}, 1}} \right\rangle_{K{\color{black}\sigma_{K}^{2}}\wedge \check{\tau}^{K}} \right)}{\left(u_{\ell -1} - 1 - \frac{5\overline{b}u_{\ell + 1} + \underline{\theta}\underline{m}_{2}}{2K} \right)^{2}} \\
 & \leqslant \frac{C_{5}}{K^{1+\varepsilon}}.
\end{align*}
To conclude this proof, there exists a constant $C >0$ such that for all $K \geqslant K_{0}$ where $K_{0} := \max_{i \in \{1, \cdots, 4 \}}{K_{i}}$, \[\P_{\mu_{0}^{K}}\left(\left[\left\{L_{1}^{K} = \ell + 1 \right\} \cup \left\{T_{1}^{K} < K{\color{black}\sigma_{K}^{2}} \right\} \right] \cap \left\{T_{1}^{K} < \check{\tau}^{K} \right\} \right)< \frac{1}{2} - \eta_{K}\]
for any $\eta_{K} \geqslant \frac{1}{2} - \epsilon_{K} - \frac{C_{4}+C_{5}}{K^{1+\varepsilon}}$. A convenient choice is given by $\eta_{K} := \frac{1}{2} - \frac{C}{K^{\varepsilon/2}}$ for $K$ large enough independent of $\ell$ which completes the proof. \qedhere
\end{proof}

\subsubsection{Construction of the coupling\label{Sous-sous-section_4_5_2_Construction_of_the_coupling}}

The goal of this section is to construct a coupling between $\left(L_{k}^{K} \right)_{k\in \N}$ and a biased random walk on $\N^{\star}$ and reflected in $1$. To do this, we will construct a sequence $\left(\xi_{k}^{K}\right)_{k \in \N}$ of i.i.d. random variables with values in $\{-1, 1\}$ as follows. \\ 

\textit{Step 1.} First of all, thanks to Lemma \ref{Lem_One-step_transition} we can construct the random variable $\xi_{0}^{K} \in \left\{-1, 1\right\}$ such that 
\[\left\{\xi_{0}^{K} = -1 \right\} \subset \left[\left\{L_{1}^{K} = L_{0}^{K} - 1 \right\} \cap \left\{K{\color{black}\sigma_{K}^{2}} \leqslant T_{1}^{K} < \check{\tau}^{K} \right\} \right] \cup \left\{T_{1}^{K} \geqslant \widehat{\tau}^{K} \right\}\]
and $\P\left(\xi_{0}^{K} = -1 \right) = \frac{1}{2} + \eta_{K}$. \\

\textit{Step 2.} Then, thanks to Lemma \ref{Lem_One-step_transition} again and after having applied the \textsc{Markov} property at time $T_{1}^{K}$, we can construct the random variable $\xi_{1}^{K} \in \left\{-1, 1\right\}$ such that, conditionally to $\FF_{T_{1}^{K}} \vee \sigma\left(\xi_{0}^{K} \right)$,
\[\left\{\xi_{1}^{K} = -1 \right\} \subset \left[\left\{L_{2}^{K} = L_{1}^{K} - 1 \right\} \cap \left\{K{\color{black}\sigma_{K}^{2}} + T_{1}^{K} \leqslant T_{2}^{K} < \check{\tau}^{K} \right\} \right] \cup \left\{T_{2}^{K} \geqslant \widehat{\tau}^{K} \right\} \]
and $\P\left(\xi_{1}^{K} = -1 \left| \phantom{1^{1^{1}}} \hspace{-0.55cm} \right. \FF_{T_{1}^{K}} \vee \sigma\left(\xi_{0}^{K} \right) \right) = \frac{1}{2} + \eta_{K}$. Note that %from the last probability, we have that  
this implies in particular that $\xi_{1}^{K}$ is independent of $\xi_{0}^{K}$. \\

\textit{Step 3.} By induction on $k\in \N$, if $\left(\xi_{i}^{K} \right)_{0\leqslant i \leqslant k}$ are constructed, then thanks to Lemma \ref{Lem_One-step_transition} again and after having applied the \textsc{Markov} property at time $T_{k+1}^{K}$, we can construct the random variable $\xi_{k+1}^{K} \in \left\{-1, 1\right\}$ such that, conditionally to $\FF_{T_{k+1}^{K}} \vee \sigma\left(\xi_{0}^{K}, \cdots, \xi_{k}^{K} \right)$, 
\[\left\{\xi_{k+1}^{K} = -1 \right\} \subset \left[\left\{L_{k+2}^{K} = L_{k+1}^{K} - 1 \right\} \cap \left\{K{\color{black}\sigma_{K}^{2}} + T_{k+1}^{K} \leqslant T_{k+2}^{K} < \check{\tau}^{K} \right\} \right] \cup \left\{T_{k+2}^{K} \geqslant \widehat{\tau}^{K} \right\} \]
and $\P\left(\xi_{1}^{K} = -1 \left| \phantom{1^{1^{1}}} \hspace{-0.55cm} \right. \FF_{T_{k+1}^{K}} \vee \sigma\left(\xi_{0}^{K}, \cdots, \xi_{k}^{K} \right) \right) = \frac{1}{2} + \eta_{K}$. \\ 

To conclude, we set
\begin{equation}
\forall K \in \N^{\star}, \quad \forall k\in \N, \qquad Z_{k+1}^{K} - Z_{k}^{K} := \xi_{k}^{K} \qquad \quad {\rm and} \qquad \quad Z_{0}^{K} = L_{0}^{K}
\label{Eq_Suite_Z_k_K}
\end{equation}
and the sequence $\left(Z_{k}^{K}\right)_{k\in \N}$ satisfies by construction the following lemma: 
\begin{Lem} The sequence $\left(Z_{k}^{K}\right)_{k\in \N}$, %constructed above and 
given by {\rm(\ref{Eq_Suite_Z_k_K})} is a biased simple random walk on $\N^{\star}$, reflected in $1$ %satisfying $Z_{0}^{K} = L_{0}^{K}$ 
 and for all $k \in \N$ as long as $T_{k+1}^{K} < \check{\tau}^{K}$, that 
\begin{align*}
%{\rm{\textit{\rm{\textbf{(i)}}}}} \qquad & \widetilde{Z}_{k+1}^{K} - \widetilde{Z}_{k}^{K} \geqslant M_{2}\left(\mu_{T_{k+1}^{K}}^{K} \right) - M_{2}\left(\mu_{T_{k}^{K}}^{K} \right), \\
{\rm{\textbf{\rm{\textbf{(1)}}}}} \qquad & Z_{k+1}^{K} - Z_{k}^{K} \geqslant L_{k+1}^{K} - L_{k}^{K} \\
{\rm{\textit{\rm{\textbf{(2)}}}}} \qquad  & T_{k+1}^{K} - T_{k}^{K} \geqslant K{\color{black}\sigma_{K}^{2}} \qquad {\rm when} \qquad Z_{k+1}^{K} - Z_{k}^{K} = -1.
\end{align*}
\label{Lem_Marche_Aleatoire_Z}
\end{Lem}

On the one hand, note that for all $k\in \N$, $L_{k}^{K} \leqslant Z_{k}^{K}$. On the other hand, note that Lemma \ref{Lem_Accroissements_M2} implies for all $k \in \N$, $L_{k+1}^{K} - L_{k}^{K} \in \left\{ \cdots, -3, -2, -1, 1 \right\} $ which justifies (1) of Lemma \ref{Lem_Marche_Aleatoire_Z}. In Figure \ref{Fig_M2_Marche_aleatoire}, we illustrate the coupling between $\left(L_{k}^{K} \right)_{0\leqslant k \leqslant 7}$ and  $\left(Z_{k}^{K}\right)_{0\leqslant k \leqslant 7}$. In this figure, we observe that $\xi_{k}^{K} = L_{k+1}^{K} - L_{k}^{K}$ for all $k \in \left\{0,1,2,4,5,6 \right\}$ and $\xi_{3}^{K} = -1 \geqslant -2 = L_{4}^{K} - L_{3}^{K}$. This illustrates that $Z_{k+1}^{K}-Z_{k}^{K}$ is always $1$ when $L_{k+1}^{K}-L_{k}^{K}$ is equal to $1$. Note that $Z_{k+1}^{K}-Z_{k}^{K}$ can be $1$ when $L_{k+1}^{K}-L_{k}^{K}$ is equal to $-1$. 

\begin{figure}[h]
\flushleft	
\definecolor{Mon_bleu_illustration}{rgb}{0.490,0.490,1} % "joli bleu 
\definecolor{ffqqtt}{rgb}{1,0,0.2}
\definecolor{ffqqqq}{rgb}{1,0,0}
\definecolor{ttzzqq}{rgb}{0.2,0.6,0}
\definecolor{xdxdff}{rgb}{0.49019607843137253,0.49019607843137253,1}
\definecolor{ududff}{rgb}{0.30196078431372547,0.30196078431372547,1}
\definecolor{ffwwqq}{rgb}{1,0.4,0}
%\hspace{-1.5cm}

\caption{For $K$ large enough, coupling between $\left(L_{k}^{K} \right)_{0\leqslant k \leqslant 7}$ and  $\left(Z_{k}^{K}\right)_{0\leqslant k \leqslant 7}$ up to time $T$ and before $\check{\tau}^{K}$ and where $M_{2}\left(\mu_{0}^{K} \right) \in [0, 2u_{1} + 1)$ so that $L_{0}^{K} = 1$ and $\widetilde{Z}_{k}^{K} := u_{Z_{k}^{K}} = 3^{Z_{k}^{K}}K^{\varepsilon/2}$ for all $k \in \N$.}
\label{Fig_M2_Marche_aleatoire}
\end{figure}

\subsubsection{Estimates of exit from an attractive domain for random walks\label{Sous-sous-section_4_5_3_Attractive_domain_exit_estimates_for_random_walks}}
\indent Let us consider $N^{K}$ the number of transitions before reaching $\left\lfloor \frac{\varepsilon}{2\log(3)}\log(K) \right\rfloor$ for the random walk $\left(Z_{k}^{K}\right)_{k\in \N}$. 

\begin{Rem} By {\rm Lemma \ref{Lem_Marche_Aleatoire_Z}}, $N^{K} \leqslant k_{0}$ where $k_{0}$ is  the first integer such that $L_{k_{0}}^{K} = \left\lfloor \frac{\varepsilon}{2\log(3)}\log(K) \right\rfloor$, so $\widehat{\tau}^{K} \geqslant T_{k_{0}}^{K} \geqslant T_{N^{K}}^{K}$. 
\end{Rem}
 
 The following lemma gives an estimate on the problem of exit from a domain for $Z^{K}$. For all $k \in \N^{\star}$, we denote by $\PPP_{k}$ the law of the \textsc{Markov} chain $Z^{K}$ given $Z_{0}^{K} = L_{0}^{K} = k$.

\begin{Lem} We have, \[\lim_{K\to + \infty}{\PPP_{1}\left(N^{K} \geqslant \exp\left(\frac{\varepsilon^{2}}{16\log(3)}\log^{2}(K) \right)  \right)} = 1. \]
\label{Lem_PGD_lineaire}
\end{Lem}

\begin{proof} For all $k \in \N^{\star}$, let us consider the stopping time $\TT_{k}^{K} := \inf\left\{ n \in \N \left| \phantom{1^{1^{1^{1}}}} \hspace{-0.7cm} \right. Z_{n}^{K} = k\right\}$ and we set $v_{k}^{K} := \PPP_{k}\left(\TT_{1}^{K} < \TT_{\left\lfloor \frac{\varepsilon}{2\log(3)}\log(K) \right\rfloor}^{K} \right)$. Note that $v_{1}^{K} = 1$. By a classical approach and setting $r_{K} :=  \frac{1+2\eta_{K}}{1-2\eta_{K}} >1$, we prove that 
\begin{align*}
v_{k}^{K} & = \frac{1-r_{K}^{-\left(\left\lfloor \frac{\varepsilon}{2\log(3)}\log(K) \right\rfloor - k\right)}}{1-r_{K}^{-\left(\left\lfloor \frac{\varepsilon}{2\log(3)}\log(K) \right\rfloor - 1\right)}} \\
& = 1 - \left(r_{K}^{k-1} - 1 \right)\exp\left(-\left(\left\lfloor \frac{\varepsilon}{2\log(3)}\log(K) \right\rfloor -1 \right)\log(r_{K}) \right).
\end{align*}
Note that $N^{K}$ is greater than the number of transitions from $2$ to $1$ that occur before $\TT_{\left\lfloor \frac{\varepsilon}{2\log(3)}\log(K) \right\rfloor}^{K}$. Hence, under $\PPP_{1}$, $N^{K} \geqslant X$ where $X$ is random variable with geometric law of parameter ${\color{black}q_{K}}:= \left(r_{K} - 1 \right)\exp\left(-\left(\left\lfloor \frac{\varepsilon}{2\log(3)}\log(K) \right\rfloor -1 \right)\log(r_{K}) \right)$. Since $r_{K} \sim_{K \to + \infty} \frac{K^{\varepsilon/2}}{C}$, with $C>0$ given by Lemma \ref{Lem_One-step_transition}, ${\color{black}q_{K} \to 0}$ when $K \to + \infty$ and using that for all $x >0$ small enough, $\log(1-x) \geqslant -2x$ we deduce that for all $K$ large enough,
\begin{align*}
 & \PPP_{1}\left(X \geqslant  \exp\left(\frac{\left\lfloor \frac{\varepsilon}{2\log(3)}\log(K) \right\rfloor -1}{2}\log\left(r_{K} \right) \right) \right) \\
& \hspace{2cm} = \exp\left(\exp\left(\frac{\left\lfloor \frac{\varepsilon}{2\log(3)}\log(K) \right\rfloor - 1}{2}\log\left(r_{K} \right) \right)\log\left(1-{\color{black}q_{K}} \right)  \right) \\
& \hspace{2cm} \geqslant \exp\left(-2\left(r_{K} - 1 \right)\exp\left(-\frac{\left\lfloor \frac{\varepsilon}{2\log(3)}\log(K) \right\rfloor - 1}{2}\log\left(r_{K} \right) \right)\right).
\end{align*}
Using again that $r_{K} \sim_{K \to + \infty} \frac{K^{\varepsilon/2}}{C}$, we deduce that for $K$ large enough, 
\[\frac{\lfloor \frac{\varepsilon}{2\log(3)} \log(K) \rfloor - 1}{2}\log\left(r_{K}\right) \geqslant \frac{\varepsilon^{2}}{16 \log(3)}\log^{2}\left(K \right) \]
and the announced result follows. \qedhere
\end{proof}

It is in the next Corollary \ref{Cor_PGD_lineaire} that we see the importance of Assumption (\ref{Hypothese_Gamme_sigma}) where ${\color{black}\sigma_{K}}$ must not be too small. 

\begin{Cor} For all $T >0$, 
 \[  \lim_{K \to + \infty}{\P\left(\frac{K{\color{black}\sigma_{K}^{2}}}{2}\left(N^{K} - \frac{\varepsilon}{2\log(3)}\log(K)\right)  < T \right)} = 0.\]
 \label{Cor_PGD_lineaire}
\end{Cor}

\begin{proof} From Lemma \ref{Lem_PGD_lineaire} and Assumption (\ref{Hypothese_Gamme_sigma}) we deduce in a straightfoward manner that \[\lim_{K \to + \infty}{\PPP_{1}\left(\frac{K{\color{black}\sigma_{K}^{2}}}{2}\left(N^{K} - \frac{\varepsilon}{2\log(3)}\log(K)\right)  < T \right)} = 0.\]
As $L_{0}^{K}$ depends on $\mu_{0}^{K}$  note that $L_{0}^{K} = Z_{0}^{K}$, is function of the initial condition $\mu_{0}^{K}$. However, by Assumption (\ref{Hypothese_C_etoile}) and \textsc{Markov}'s inequality, we have that \[\P\left(L_{0}^{K} >1 \right) \leqslant \P\left(M_{2}\left(\mu_{0}^{K} \right) \geqslant K^{\frac{\varepsilon}{2}} \right) \leqslant \frac{C_{2}^{\star}}{K^{\varepsilon}} \]
 which tends to $0$ when $K \to +\infty$ and the announced result follows. \qedhere
\end{proof}

\subsubsection{Conclusion\label{Sous-sous-section_4_5_4_Conclusion}} 
We denote respectively by $N_{+}^{K}$ and $N_{-}^{K}$ the number of upward transitions and downward transitions for $\left(Z_{k}^{K}\right)_{k\in \N}$ before reaching $\left\lfloor \frac{\varepsilon}{2\log(3)}\log(K) \right\rfloor$.  Let $T \geqslant 0$ be fixed. Note that 
 \begin{equation}
 \begin{aligned}
 \P\left(\widehat{\tau}^{K} \leqslant \check{\tau}^{K}\wedge T \right) & \leqslant \P\left(\left\{ \widehat{\tau}^{K} \leqslant \check{\tau}^{K} \wedge T \right\} \cap \left\{\check{\tau}^{K} > T_{N^{K}}^{K} \right\} \right) \\
 & \hspace{1cm} + \P\left(\left\{ \widehat{\tau}^{K} \leqslant  \check{\tau}^{K} \wedge T \right\} \cap \left\{\check{\tau}^{K} \leqslant T_{N^{K}}^{K} \right\} \right).
 \end{aligned}
 \label{Eq_Conclusion_Section_4}
 \end{equation}
 
\textit{Step 1. Control of the first right-hand term of (\ref{Eq_Conclusion_Section_4}).} As $Z_{0}^{K} = 1$ with probability converging to $1$ and $Z_{N^{K}}^{K} = \left\lfloor \frac{\varepsilon}{2\log(3)}\log(K) \right\rfloor$, we have 
 $N_{+}^{K} - N_{-}^{K} \leqslant \left\lfloor \frac{\varepsilon}{2\log(3)}\log(K) \right\rfloor$ and $N_{+}^{K} + N_{-}^{K} = N^{K}$. Thus, \[N_{-}^{K} \geqslant \frac{N^{K} - \left\lfloor \frac{\varepsilon}{2\log(3)}\log(K) \right\rfloor}{2} .\] 
 
 Hence, if $\check{\tau}^{K} > T_{N^{K}}^{K}$, then by Lemma \ref{Lem_Marche_Aleatoire_Z} (2),
 \[\widehat{\tau}^{K}  \geqslant T_{N^{K}}^{K} \geqslant \sum\limits_{i\, = \, 1}^{N^{K}}{\left(T_{i}^{K} - T_{i-1}^{K} \right)} \geqslant N_{-}^{K}K{\color{black}\sigma_{K}^{2}} \geqslant \frac{K{\color{black}\sigma_{K}^{2}}}{2}\left(N^{K} - \frac{\varepsilon}{2\log(3)}\log(K) \right). \]
 Therefore, \[\P\left(\left\{ \widehat{\tau}^{K} \leqslant \check{\tau}^{K} \wedge T \right\} \cap \left\{\check{\tau}^{K} > T_{N^{K}}^{K} \right\} \right) \leqslant \P\left(\frac{K{\color{black}\sigma_{K}^{2}}}{2}\left(N^{K} - \frac{\varepsilon}{2\log(3)}\log(K) \right) \leqslant T \right)\]
which tends to $0$ when $K \to + \infty$ according to Corollary \ref{Cor_PGD_lineaire}. \\

\textit{Step 2. Control of the second right-hand term of (\ref{Eq_Conclusion_Section_4}).}  If $\check{\tau}^{K} \leqslant T_{N^{K}}^{K}$, let us consider $\widetilde{N}^{K}$ the last index $k$ such that $T_{k}^{K} < \check{\tau}^{K}$.  We have $\widetilde{N}^{K} < N^{K}$ and so $\widetilde{N}^{K} +1 \leqslant \widetilde{N}^{K}$ which implies, thanks to Lemma \ref{Lem_Accroissements_M2}, that for all $t \leqslant T_{\widetilde{N}^{K}+1}$ and $K$ large enough, \[M_{2}\left(\mu_{t}^{K} \right) \leqslant 2u_{\lfloor\frac{\varepsilon}{2\log(3)}\log(K) \rfloor} +1 \leqslant 2K^{\frac{\varepsilon}{2}} + 1 < K^{\varepsilon}.\] As the jump times of $\nu^{K}$ are isolated, we deduce that $\widehat{\tau}^{K}>T_{\widetilde{N}^{K} + 1}^{K}$. Now, by definition of $\widetilde{N}^{K}$, we have $T_{\widetilde{N}^{K}+1}^{K} \geqslant \check{\tau}^{K}$ and thus $\widehat{\tau}^{K} > \check{\tau}^{K}$. Hence, we obtain that \[ \P\left(\left\{ \widehat{\tau}^{K} \leqslant \check{\tau}^{K} \wedge T \right\} \cap \left\{\check{\tau}^{K} \leqslant T_{N^{K}}^{K} \right\} \right) = \P\left(\varnothing \right) = 0.\]
Therefore, $\lim_{K\to+\infty}{\P\left(\widehat{\tau}^{K} \leqslant \check{\tau}^{K}\wedge T \right)} = 0$ which concludes the proof of Lemma \ref{Lem_Proba_sup_M2_K_epsilon}. 

\section{Tightness on the torus\label{Section_5_Tension_LENT-RAPIDE}} %Tightness of the occupation measures $\Gamma^{K}$ on the torus

The main result of this section is given by Theorem \ref{Thm_Kurtz_adapte}. This is a stochastic averaging result inspired by \textsc{Kurtz} {\color{blue} \cite[\color{black} Theorem 2.1]{Kurtz}} establishing a tightness result, in the torus case, of the sequence of laws of $\left(\left(z^{K}, \Gamma^{K} \right)\right)_{K \in \N^{\star}}$ where $\Gamma^{K}$ is the occupation measure of the fast component $\mu^{K}$. We will use criteria proposed by \textsc{Ethier-Kurtz} in {\color{blue} \cite[\color{black} Theorems 3.9.1 and 3.9.4]{Ethier_markov_1986}}.  In Sections \ref{Section_6_Caract_Gamma_Limite} and \ref{Section_7_Caract_Composante_LENTE} we identify the limit and we prove its uniqueness. \\  

Let introduce the torus $\T_{R} := \left[x_{0} - 2R, x_{0} + 2R \right]$ of length $4R$ with $R >0$ fixed %and large enough 
and $x_{0}$ is the value of the mean trait of $\nu_{0}^{K}$ as in Theorem \ref{Thm_CEAD_Jouet}. We define $\MM_{1}\left(\T_{R}\right)$ the set of probability measure on $\T_{R}$. We denote by $\MM_{m}\left(\MM_{1}\left(\R\right)\right)$ the set of measures $\Gamma$ on $\R_{+} \times \MM_{1}\left(\R\right)$ such that for all $t \geqslant 0$, $\Gamma\left([0, t] \times \MM_{1}\left(\R\right) \right) = t$. For any $t \geqslant 0$, we denote by $\MM_{m}^{t}\left(\MM_{1}\left(\R\right) \right)$ the set of measures $\Gamma \in \MM_{m}\left(\MM_{1}\left(\R\right)\right)$ restricted to $[0,t] \times \MM_{1}\left(\R\right)$. For all $T>0$, we denote by $\D\left([0,T], \T_{R}\right)$ the space of {\color{black} c\`{a}d-l\`{a}g} functions on $[0,T]$ with values in $\T_{R}$.
 
Let $b_{R} \in \CCCC^{2}_{b}\left(\T_{R}^{2}, \R\right)$ and $\theta_{R} \in \CCCC^{2}_{b}\left(\T_{R}, \R\right)$ be two functions satisfying $b_{R} = b$ on $[x_{0} - R, x_{0} + R]^{2}$ and $\theta_{R} = \theta$ on $[x_{0} - R, x_{0} + R]$.  Let $T>0$ be fixed. We define on $\T_{R}$, in similar way as in Section \ref{Section_2_IBM}, the processes $\nu^{K, R} := \left(\nu_{t}^{K,R} \right)_{t \in [0, T]}$, $z^{K, R} := \left( z_{t}^{K, R} \right)_{t\in [0, T]}$ and $\mu^{K, R} = \left(\mu_{t}^{K, R} \right)_{t\in [0, T]}$ as follows \[\nu_{t}^{K, R} := \frac{1}{K}\sum_{i\, = \, 1}^{K}{\delta_{x_{i}^{R}(t)}}, \quad  z_{t}^{K, R} := \left\langle \id, \nu_{t/K{\color{black}\sigma_{K}^{2}}}^{K, R}\right\rangle \quad  {\rm and}  \quad \mu_{t}^{K, R} := \left(h_{\frac{1}{{\color{black}\sigma_{K}} \sqrt{K}}} \circ \tau_{-z_{t}^{K}} \right) \sharp \, \nu_{t/K{\color{black}\sigma_{K}^{2}}}^{K} \] where for all $i \in \left\{1, \cdots, K\right\}$, $\nu_{t}^{K, R}$ is defined on $\MM_{1}\left(\T_{R} \right)$ from $b_{R}$ and $\theta_{R}$ as $\nu_{t}^{K}$ was defined from $b$ and $\theta$ and $z_{t}^{K, R} \in \T_{R}$. Note that $\mu_{t}^{K, R}$ takes values in $\MM_{1}\left(\T_{R}^{K} \right)$ where $\T_{R}^{K}$ is the torus corresponding to the interval $\left[\frac{1}{{\color{black}\sigma_{K}}\sqrt{K}}\left(x_{0} - 2R - z_{t}^{K} \right), \frac{1}{{\color{black}\sigma_{K}}\sqrt{K}}\left(x_{0} + 2R - z_{t}^{K} \right) \right]$. However, we will identify in the sequel $\MM_{1}\left(\T_{R}^{K} \right)$ as a subset of $\MM_{1}(\R)$ using the natural embedding  of $\T_{R}^{K}$ in $\R$. So, $\mu_{t}^{K, R} \in \MM_{1}(\R)$. We define also the stopping times 
\begin{align*}
\widehat{\tau}^{K,R} & := \inf\left\{t\geqslant 0 \left| \phantom{1^{1^{1^{1}}}} \hspace{-0.6cm} \right. M_{2}\left(\mu_{t}^{K, R} \right) \geqslant K^{\varepsilon} \right\},  \\
\check{\tau}^{K,R} & := \inf\left\{t\geqslant 0 \left| \phantom{1^{1^{1^{1}}}} \hspace{-0.6cm} \right. \Diam\left(\Supp \nu_{t}^{K, R} \right) > \frac{1}{{\color{black}\sigma_{K}} K^{\frac{3+\varepsilon}{2}}}\right\},
\end{align*}
and $\tau^{K, R} := \check{\tau}^{K,R} \wedge \widehat{\tau}^{K,R}$. We define on $\MM_{m}^{T}\left(\MM_{1}\left(\R \right) \right)$ the sequence of random measures $\left(\Gamma^{K, R} \right)_{K \in \N^{\star}}$ as follows
\begin{equation*}
\Gamma^{K, R}\left(\dd t, \dd \mu\right) = \delta_{\mu_{t}^{K, R}}\left(\dd \mu\right)\dd t.
\label{Gamma_K}
\end{equation*}
In the sequel, we study the limit as $K$ tends to $+\infty$ of the sequence of laws of $$\left(\left(z^{K, R}, \Gamma^{K, R} \right) \right)_{K \in \N^{\star}}$$ in $\MM_{1}\left(\D\left(\left[0, T\right], \T_{R} \right) \times \MM_{m}^{T}\left(\MM_{1}\left(\R\right) \right)\right)$.

\begin{Thm} Let $T, R>0$. The sequence of laws of $\left(\left(z^{K, R}, \Gamma^{K, R} \right) \right)_{K \in \N^{\star}}$ is tight in the set of probability
measures on $\D\left(\left[0, T\right], \T_{R} \right) \times \MM_{m}^{T}\left(\MM_{1}\left(\R\right) \right)$ and for any limiting value $\Q$ of this sequence, the canonical process $\left(\zeta^{R}, \Gamma^{R} \right)$ on $\D\left(\left[0, T\right], \T_{R} \right) \times \MM_{m}^{T}\left(\MM_{1}\left(\R \right) \right)$ satisfies that  for all $f \in \CCCC_{b}^{2}\left(\T_{R}, \T_{R} \right)$, 
\begin{equation}
N_{t}^{f, R} := f\left(\zeta_{t}^{R}\right) - f\left(\zeta_{0}^{R}\right) - \int_{0}^{t}{\int_{\MM_{1}\left(\R\right)}^{}{\LL_{{\rm SLOW}}f\left(\zeta_{s}^{R}, \mu \right)\Gamma^{R}\left(\dd s ,\dd \mu \right)}} = 0
\label{Martingale_N_t_f}
\end{equation}
$\Q-$a.s. Moreover, for all $t \in [0, T]$, 
\begin{equation}
\E_{\Q}\left(\int_{0}^{t}{\int_{\MM_{1}\left(\R\right)}^{}{M_{5}\left(\mu\right)\Gamma^{R}\left(\dd s, \dd \mu \right)}} \right) < \infty,
\label{Controle_E_M2_Gamma}
\end{equation}
so that the definition {\rm(\ref{Martingale_N_t_f})} makes sense.
\label{Thm_Kurtz_adapte}
\end{Thm}

This result is based on Propositions \ref{Prop_Generateur_Couple_Lent_Rapide} and \ref{Prop_Convergence_tau_et_theta}, Lemma \ref{Lem_Controle_Integrale_Temps_M2}. The proof is divided in eight steps. In Step 1, we establish the tightness of the family of laws of the stopped slow component $\left(z_{{\tiny \bullet} \wedge \tau^{K,R}}^{K, R} \right)_{K\in \N^{\star}}$. In Step 2, we establish a compact {\color{black}containment} condition for the stopped \emph{fast} component $\left(\mu_{{\tiny \bullet}\wedge \tau^{K,R}}^{K,R}\right)_{K \in \N^{\star}}$. In Step 3, we prove the tightness of the family of laws of the occupation measure. In Step 4, we deduce the tightness of the family of laws of the couple (slow, occupation measure fast). In Step 5, we prove uniform integrability results for a family $\left(N_{t\wedge \tau^{K,R}}^{f, K, R} \right)_{t\in [0, T], K \in \N^{\star}}$ constructed from $z^{K, R}$ and $\Gamma^{K, R}$ similarly as $N_{t}^{f,R}$ in (\ref{Martingale_N_t_f}). Step 6 is devoted to establish the convergence in distribution of $\left(N_{{\tiny \bullet}\wedge \tau^{K,R}}^{f, K, R} \right)_{K \in \N^{\star}}$ to $N^{f, R}$. In Steps 7 and 8, we prove that the limit $N^{f, R}$ is null $\Q-$a.s. 

The main modification of \textsc{Kurtz}'s setting of {\color{blue} \cite[\color{black} Theorem 2.1]{Kurtz}} is that we have to work with stopped times and need to be careful with moment estimates and uniform integrability properties. This led us to rewrite the proof. %\\

\begin{proof}[Proof of Theorem \ref{Thm_Kurtz_adapte}]
\textit{Step 1. Tightness of the family of laws of $\left( z_{{\tiny \bullet} \wedge \tau^{K,R}}^{K, R}\right)_{K\in \N^{\star}}$ on $\D\left([0, T], \T_{R} \right)$.}  Let $f \in \CCCC_{b}^{2}(\T_{R}, \T_{R})$. For all $K \in \N^{\star}$ and for all $t \in [0, T]$, let us consider the two processes $Y_{t}^{K, R} := f\left(z_{t\wedge \tau^{K,R}}^{K, R}\right)$ and $Z_{t}^{K, R}$ defined by the relation 
 \[Y_{t}^{K, R} = Y_{0}^{K, R} + \int_{0}^{t}{Z_{s}^{K, R}\dd s} + M_{t\wedge \tau^{K, R}}^{f, K, R}\] where $M_{t\wedge \tau^{K,R}}^{f, K, R}$ is the martingale given by (\ref{PB_Mg_LENT_Approx}) in the torus case. Note that, for all $s > \tau^{K, R}$, $Z_{s}^{K, R} = 0$. From Proposition \ref{Prop_Generateur_Lent_Decomposition}, (\ref{Eq_M3}) and Lemma \ref{Lem_Controle_Integrale_Temps_M2}, note that there exists a constant $C>0$ such that 
 \begin{align*}
 & \sup_{K\in \N^{\star}}{\E\left[\left(\int_{0}^{T}{\left|Z_{t}^{K, R}\right|^{2}\dd t}\right)^{\frac{1}{2}}\right]} \\
  & \hspace{1.5cm} \leqslant \sup_{K\in \N^{\star}}{\E\left[1 + \int_{0}^{T}{\left|Z_{t}^{K, R} \right|^{2}\dd t} \right]} \\
 & \hspace{1.5cm} \leqslant C \left\{ 1 + \sup_{K\in \N^{\star}}{\E\left[\int_{0}^{T\wedge \tau^{K,R}}{\left(M_{2}^{2}\left(\mu_{t}^{K, R} \right) + \frac{1}{K^{2}}\left[1 + M_{2}^{2}\left(\mu_{t}^{K, R} \right) \right] \right)\dd t} \right]}\right\} \\
 & \hspace{1.5cm} < \infty.
 \end{align*}
 Hence, from {\color{blue} \cite[\color{black} Theorem 3.9.4]{Ethier_markov_1986}}, the family of laws of  $\left(Y^{K, R} \right)_{K \in \N^{\star}}$, on $\D\left([0, T], \T_{R} \right)$, is tight. Let us observe that the compact {\color{black}containment} condition is satisfied by the stopped \emph{slow} component $\left(z_{{\tiny \bullet} \wedge \tau^{K,R}}^{K, R} \right)_{K \in \N^{\star}}$ since $\T_{R}$ is compact and \begin{equation}
 \forall T>0, \quad \inf_{K \in \N^{\star}}{\P\left( \sup_{0\leqslant t \leqslant T} z_{t\wedge \tau^{K,R}}^{K, R} \in \T_{R} \right)} = 1.
  \label{Compact_containment_condition_LENTE}
\end{equation}
 
  As $\CCCC_{b}^{2}\left(\T_{R}, \T_{R}\right)$ is a dense subset of $\CCCC_{b}^{0}\left(\T_{R}, \T_{R}\right)$ in the topology of the uniform norm, we deduce from (\ref{Compact_containment_condition_LENTE}) and {\color{blue} \cite[\color{black} Theorem 3.9.1]{Ethier_markov_1986}} that the family of laws of $\left( z_{{\tiny \bullet} \wedge \tau^{K,R}}^{K, R}\right)_{K\in \N^{\star}}$ is tight on $\D\left([0, T], \T_{R} \right) $.  \\
 
 \textit{Step 2. Compact {\color{black}containment} condition.} 
  \begin{Lem} Let $R >0$. For all $T \geqslant 0$, the family of laws of the marginal random variables of the stopped \emph{fast} process $\left(\mu_{t\wedge \tau^{K, R}}^{K, R}  \right)_{t \in [0, T], K \in \N^{\star}}$ is tight on $\MM_{1}\left(\R\right)$ i.e. \[ \forall \eta >0, \ \exists \, D_{{\color{black}\eta}}^{R, T} \subset \MM_{1}\left(\R\right) \ {\rm compact}, \ \forall t \in \left[0,T\right], \ \forall K \in \N^{\star}, \quad \P\left(\mu_{t\wedge \tau^{K,R}}^{K, R} \in D_{{\color{black}\eta}}^{R, T} \right) \geqslant 1-\eta.\]
\label{Lemme_criteres_tension}
\end{Lem}
  
\begin{proof}
Let $\eta >0$ be fixed. From Proposition \ref{Prop_Convergence_tau_et_theta}, there exists $K_{0}\in \N^{\star}$ large enough such that for all $K\geqslant K_{0}$, for all $t\in [0, T]$, $\P\left(t > \tau^{K, R} \right) \leqslant \frac{\eta}{2}$.  We consider the $\R-$valued sequence $\left(a_{q, \eta} \right)_{q\in \N}$ satisfying for all $q \in \N, a_{q, \eta} > 1$, and $\sum_{q\in \N^{\star}}^{}{\frac{q}{a_{q, \eta}^{2}}} < \frac{\eta}{2\overline{M}_{2}}$ where $\overline{M}_{2} := \max\left\{C_{2}^{\star}, \overline{\theta}\overline{m}_{2}{\color{black}/\underline{b}} \right\}$ is a uniform upper bound of $\E\left(M_{2}\left(\mu_{t\wedge \tau^{K,R}}^{K, R}  \right) \right)$ given by Corollary \ref{Cor_Esp_M2_max_2_constantes}. Let $\left(\kappa_{q, \eta} \right)_{q\in \N}$ be a sequence of compact {\color{black} intervals} on $\R$, increasing for inclusion, of the form $\left[-a_{q, \eta}, a_{q, \eta}\right]$. Let $t \in [0, T]$.  For all $q \in \N^{\star}$, and $K \geqslant K_{0}$, note that \[M_{2}\left(\mu_{t}^{K, R} \right)\II_{t\leqslant \tau^{K,R}} %= \int_{\T_{R/\sigma\sqrt{K}}}^{}{x^{2}\mu_{t}^{K, R}\II_{t\leqslant \tau^{K}}\left(\dd x\right)} 
\geqslant \II_{t\leqslant \tau^{K,R}}\int_{\kappa_{q, \eta}^{c}}^{}{x^{2}\mu_{t}^{K, R}\left(\dd x\right)} \geqslant a_{q, \eta}^{2}\mu_{t\wedge \tau^{K, R}}^{K, R} \left(\kappa_{q, \eta}^{c} \right)\II_{t\leqslant \tau^{K, R}}, \]
and we deduce that $\left\{\mu_{t\wedge \tau^{K, R}}^{K, R} \left(\kappa_{q, \eta}^{c} \right) > \frac{1}{q} \right\} \cap \left\{t\leqslant \tau^{K, R} \right\} \subset \left\{M_{2}\left(\mu_{t\wedge \tau^{K, R}}^{K, R}  \right) > \frac{a_{q, \eta}^{2}}{q} \right\} \cap \left\{t\leqslant \tau^{K,R} \right\} $. Hence, 
\begin{align*}
& \P\left(\left\{\exists q \in \N^{\star}, \ \mu_{t\wedge \tau^{K, R}}^{K, R} \left(\kappa_{q, \eta}^{c} \right) > \frac{1}{q}  \right\}\right) \\
 & \hspace{2cm} \leqslant \P\left(\tau^{K,R} < t \right) + \sum_{q\in \N^{\star}}^{}{\P\left(\left\{\mu_{t\wedge \tau^{K, R}}^{K, R} \left(\kappa_{q, \eta}^{c} \right) > \frac{1}{q}\right\} \cap \left\{t \leqslant \tau^{K,R} \right\}  \right)} \\
& \hspace{2cm} \leqslant \P\left(\tau^{K,R} < t \right) + \sum_{q\in \N^{\star}}^{}{\P\left( M_{2}\left(\mu_{t}^{K, R} \right)\II_{t\leqslant \tau^{K,R}} > \frac{a_{q, \eta}^{2}}{q} \right)} \\
& \hspace{2cm} \leqslant \P\left(\tau^{K,R} < t \right) + \sum_{q\in \N^{\star}}^{}{\frac{\E\left(M_{2}\left(\mu_{t}^{K, R} \right) \II_{t\leqslant \tau^{K,R}}\right)}{\frac{a_{q, \eta}^{2}}{q}}} \\
& \hspace{2cm} \leqslant \eta.
\end{align*} 
Therefore, we have proved that 
\[\forall \eta >0, \ \forall t \in [0, T], \ \forall K \geqslant K_{0}, \quad \P\left(\mu_{t\wedge \tau^{K, R}}^{K, R} \in \KK_{\eta} \right)\geqslant 1 - \eta,\]
where $\KK_{\eta} :=\left\{\mu\in \MM_{1}(\R) \left|\phantom{1^{1^{1^{1}}}} \hspace{-0.6cm} \right. \forall q \in \N^{\star}, \mu\left(\kappa_{q, \eta}^{c} \right) \leqslant \frac{1}{q} \right\}$ which is compact by \textsc{Prohorov}'s theorem.\qedhere 
\end{proof}

 \textit{Step 3. Tightness of the family of laws of $\left(\Gamma^{K, R} \right)_{K\in \N^{\star}}$ on $\MM_{m}^{T}\left(\MM_{1}\left(\R\right) \right)$.} Let $\eta > 0$ be fixed. From Proposition \ref{Prop_Convergence_tau_et_theta}, there exists $K_{0} \in \N^{\star}$ such that for all $K \geqslant K_{0}$, for all $t \in [0, T]$, $\P\left(t \geqslant \tau^{K,R}\right) \leqslant \frac{\eta}{2}$. Consider $D_{{\color{black}\frac{\eta}{2}}}^{R, T}$ the compact set in Lemma \ref{Lemme_criteres_tension}. 
  It follows that for all $t \in [0, T]$, $K \geqslant K_{0}$, 
 \begin{align*}
 \E\left(\Gamma^{K,R}\left([0, t] \times D_{{\color{black}\frac{\eta}{2}}}^{R, T} \right) \right) & \geqslant \E\left(\Gamma^{K,R}\left([0, t] \times D_{{\color{black}\frac{\eta}{2}}}^{R, T} \right)\II_{t<\tau^{K,R}} \right) \\
 & = \int_{0}^{t}{\P\left(\left\{\mu_{s\wedge \tau^{K, R}}^{K, R}  \in D_{{\color{black}\frac{\eta}{2}}}^{R, T} \right\} \cap \{t<\tau^{K,R} \} \right)\dd s} \\
 & \geqslant \int_{0}^{t}{\P\left(\mu_{s\wedge \tau^{K, R}}^{K, R} \in D_{{\color{black}\frac{\eta}{2}}}^{R, T} \right)\dd s}  - \int_{0}^{t}{\P\left(t\geqslant\tau^{K,R}\right)\dd s} \\
 & \geqslant t\left(1 - \eta\right).
 \end{align*}
 Therefore, the tightness of the family of laws of $\left(\Gamma^{K, R} \right)_{K\in \N^{\star}}$ follows from {\color{blue} \cite[\color{black} Lemma 1.3]{Kurtz}}. \\ 
 
 \textit{Step 4. Tightness of the family of laws of $\left(\left( z^{K, R}, \Gamma^{K, R}\right) \right)_{K\in \N^{\star}}$ on $\D\left([0, T], \T_{R} \right) \times \MM_{m}^{T}\left(\MM_{1}\left(\R\right)\right)$.} From Steps 1 and 3 and \textsc{Prohorov}'s theorem, we deduce that the family of laws of the couple $\left(\left( z_{{\tiny \bullet} \wedge \tau^{K,R}}^{K, R},  \Gamma^{K, R}\right) \right)_{K\in \N^{\star}}$ is relatively compact in $\MM_{1}\left(\D\left([0, T], \T_{R} \right)\right.  \times$ $\MM_{m}^{T}\left.\left(\MM_{1}\left(\R \right)\right) \right)$. Hence, there exists a probability measure $\Q$ on the canonical space $\D\left([0, T], \T_{R} \right) \times \MM_{m}^{T}\left(\MM_{1}\left(\R\right)\right)$  and an increasing function $n : \N^{\star} \to \N^{\star}$ such that the subsequence of laws of $\left(\left( z_{{\tiny \bullet} \wedge \tau^{n\left(K\right), R}}^{n\left(K\right), R}, \Gamma^{n\left(K\right), R}\right)\right)_{K \in \N^{\star}}$ converges weakly to the limiting value  $\Q$ when $K \to + \infty$.  Thanks to Proposition \ref{Prop_Convergence_tau_et_theta},  we deduce that the family of laws of $\left(\left( z^{n(K), R}, \Gamma^{n(K), R}\right) \right)_{K\in \N^{\star}}$ converges weakly to $\Q$  when $K \to + \infty$ and therefore that the family of laws of $\left(\left( z^{K, R}, \Gamma^{K, R}\right) \right)_{K\in \N^{\star}}$ is relatively compact, thus tight on $\D\left([0, T], \T_{R} \right) \times \MM_{m}^{T}\left(\MM_{1}\left(\R\right)\right) $ by \textsc{Prohorov}'s theorem. \\
 
 \textit{Step 5. Uniform integrability.} For all $f \in \CCCC_{b}^{2}\left(\T_{R}, \T_{R}\right)$, $K \in \N^{\star}$, let us consider $\left(N_{t}^{f, K, R}\right)_{t \in [0, T]}$ the stochastic process defined by 
\begin{equation*}
N_{t}^{f, K, R} := f\left(z_{t}^{K, R} \right) - f\left(z_{0}^{K, R} \right) - \int_{0}^{t}{\int_{\MM_{1}\left(\R\right)}^{}{\LL_{\rm SLOW}f\left(z_{s}^{K, R}, \mu \right)\Gamma^{K, R}\left(\dd s, \dd \mu\right)} }. 
\label{Processus_N_t_f_Def}
\end{equation*} 
 From (\ref{Generateur_LENT}), we have for all $f \in \CCCC_{b}^{2}\left(\T_{R}, \T_{R}\right)$, there exists a constant $C>0$ such that 
 \[ \forall t \in [0, T], \ \forall K \in \N^{\star}, \quad \left|N_{t\wedge \tau^{K,R}}^{f, K, R} \right| \leqslant  C\left( 1 + \int_{0}^{t}{M_{2}\left(\mu_{s}^{K, R} \right)\II_{s\leqslant \tau^{K, R}}\dd s} \right), \]
Hence, the uniform integrability of $\left(N_{{\tiny \bullet}\wedge \tau^{K,R}}^{f, K, R} \right)_{K \in \N^{\star}}$ follows from Corollary \ref{Cor_Integrale_temps_M2_UI}. \\
  
 \textit{Step 6. Proof of (\ref{Controle_E_M2_Gamma}) and almost sure convergence of $\left(\widetilde{N}_{{\tiny \bullet} \wedge \tau^{K,R}}^{f, K, R} \right)_{K\in \N^{\star}}$ to $\widetilde{N}^{f, R}$.} From Step 4, Proposition \ref{Prop_Convergence_tau_et_theta}  and \textsc{Skorohod}'s representation theorem, there exists an increasing function $\bar{n}: \N^{\star} \to \N^{\star}$ and a probability space on which we define, the random variable $\widetilde{\tau}^{\bar{n}\left(K\right), R}$, the families $\left(\widetilde{z}^{\bar{n}\left(K\right), R}\right)_{K\in \N^{\star}}$ and $\left(\widetilde{\Gamma}^{\bar{n}\left(K\right), R} \right)_{K \in \N^{\star}}$ and $\widetilde{\zeta}^{R}, \widetilde{\Gamma}^{R}$ copies of $\left(z^{\bar{n}(K), R}  \right)_{K\in \N^{\star}}$, $\left(\Gamma^{\bar{n}(K), R} \right)_{K\in \N^{\star}}$, $\zeta^{R}$, $\Gamma^{R}$ under $\Q$ such that the sequence $\left(\left(\widetilde{z}^{\bar{n}\left(K\right), R}, \widetilde{\Gamma}^{\bar{n}\left(K\right), R}, \widetilde{\tau}^{\bar{n}\left(K\right), R}\right)\right)_{K\in \N^{\star}}$ converges a.s. to $\left(\widetilde{\zeta}^{R}, \widetilde{\Gamma}^{R}, + \infty\right)$ when $K \to + \infty$. Note that for all $t \in [0, T]$, 
  \[\widetilde{\Gamma}^{\bar{n}\left(K\right), R}\left(\left[0, t\wedge \widetilde{\tau}^{\bar{n}\left(K\right), R} \right] \times \MM_{1}\left(\R\right) \right) = t\wedge\widetilde{\tau}^{\bar{n}\left(K\right), R} \xrightarrow[ K\longrightarrow  + \infty ]{{\rm a.s.}} t = \widetilde{\Gamma}^{\color{black}R}\left([0, t] \times \MM_{1}\left(\R\right) \right).\]
  
From {\color{blue} \cite[\color{black} Lemma 1.5 (b)(c)(d)]{Kurtz}} and Corollary \ref{Cor_Integrale_temps_M2_UI}, we have \[\int_{0}^{t\wedge \widetilde{\tau}^{K, R}}{\int_{\MM_{1}(\R)}^{}{M_{5}\left(\mu \right)\widetilde{\Gamma}^{K, R}\left(\dd s, \dd \mu \right)}}   \xrightarrow[ K\longrightarrow  + \infty ]{{\rm a.s.}} \int_{0}^{t}{\int_{\MM_{1}(\R)}^{}{M_{5}\left(\mu \right)\widetilde{\Gamma}^{\color{black} R}\left(\dd s, \dd \mu \right)}}   \]
and (\ref{Controle_E_M2_Gamma}) follows.  From {\color{blue} \cite[\color{black} Lemma 1.5 (b)(c)(d)]{Kurtz}}, we also deduce that for all $f\in \CCCC_{b}^{2}\left(\T_{R}, \T_{R} \right)$, for all $t \in [0, T]$, a.s. 
 \begin{align*}
& \lim_{K\to + \infty}{\int_{0}^{t\wedge \widetilde{\tau}^{\bar{n}\left(K \right), R}}{\int_{\MM_{1}\left(\R \right)}^{}{\LL_{\rm SLOW}f\left(\widetilde{z}_{s}^{\bar{n}\left(K\right), R}, \mu \right)\widetilde{\Gamma}^{\bar{n}\left(K\right), R}\left(\dd s, \dd \mu \right)}}} \\
& \hspace{3cm} = \int_{0}^{t}{\int_{\MM_{1}\left(\R\right)}^{}{\LL_{\rm SLOW}f\left(\widetilde{\zeta}^{R}, \mu \right)\widetilde{\Gamma}^{R}\left(\dd s, \dd \mu \right)}}
\end{align*}
and thus the sequence $\left(\widetilde{N}_{{\tiny \bullet} \wedge \widetilde{\tau}^{\bar{n}\left(K\right), R}}^{f, \bar{n}\left(K\right), R} \right)_{K\in \N^{\star}}$ converges a.s. to the process $\left(\widetilde{N}_{t}^{f, R} \right)_{t\in[0, T]}$ where $\widetilde{N}_{t}^{f, K, R}$ and $\widetilde{N}_{t}^{f, R}$ are respectively defined for all $K \in \N^{\star}$ by 
\begin{multline*}
\widetilde{N}_{t}^{f, K, R} := f\left(\widetilde{z}_{t}^{K, R}\right) - f\left(\widetilde{z}_{0}^{K, R}\right) - \int_{0}^{t}{\int_{\MM_{1}\left(\R\right)}^{}{\LL_{{\rm SLOW}}f\left(\widetilde{z}_{t}^{K, R}, \mu \right)\widetilde{\Gamma}^{K, R}\left(\dd s ,\dd \mu \right)}}, \\
  t \leqslant T \wedge \widetilde{\tau}^{K, R}
\end{multline*}
  \begin{align*}
  \hspace{-0.35cm}\widetilde{N}_{t}^{f, R} & := f\left(\widetilde{\zeta}_{t}^{R}\right) - f\left(\widetilde{\zeta}_{0}^{R}\right) - \int_{0}^{t}{\int_{\MM_{1}\left(\R\right)}^{}{\LL_{{\rm SLOW}}f\left(\widetilde{\zeta}_{t}^{R}, \mu \right)\widetilde{\Gamma}^{R}\left(\dd s ,\dd \mu \right)}}, \quad t \leqslant T.
  \end{align*}
  
  \textit{Step 7. $N^{f, R}$ is a martingale.} Let us consider the filtration $\left(\widetilde{\FF}_{t}^{K, R} \right)_{t\in [0, T]}$ defined by $\widetilde{\FF}_{t}^{K, R} :=  \sigma\left(\widetilde{z}_{s}^{K, R}, \widetilde{\Gamma}^{K, R}\left([0, s] \times H  \right)\left| \phantom{1^{1^{1^{1}}}} \hspace{-0.7cm} \right.  s\leqslant t, H \in \BB\left(\MM_{1}\left( \R\right) \right) \right)$, $\left(f_{i} \right)_{1\leqslant i \leqslant q}$, $q\in \N^{\star}$ bounded \textsc{Lipschitz} functions from $\T_{R}$ to $\T_{R}$ and $0<t_{1} \leqslant \cdots \leqslant t_{q} \leqslant s < t$. Let us denote for all $K\in \N^{\star}$, for all $t \leqslant \widetilde{\tau}^{K,R} \wedge T$,  $\widetilde{M}_{t}^{f, K, R} := \widetilde{N}_{t}^{f, K, R} + \widetilde{\EE}_{t}^{f, K, R}$  where $\widetilde{M}^{f, K, R}$ is constructed from $\widetilde{z}^{K, R}$ and $\widetilde{\mu}^{K, R}$ as in (\ref{PB_Mg_LENT_Approx}) and $\widetilde{\EE}_{t}^{f, K, R}$ is an error term.  Note that $\widetilde{M}_{t}^{f, K, R}$ is a $\left(\widetilde{\FF}_{t}^{K, R}\right)_{t \in [0, T]}-$martingale as in Proposition \ref{Prop_Generateur_Couple_Lent_Rapide}. Hence, 
  \begin{align*}
  \E\left(\widetilde{M}^{f, \bar{n}\left(K\right), R}_{t\wedge\widetilde{\tau}^{\bar{n}\left(K\right),R}}\prod_{i\, = \, 1}^{q}{f_{i}\left(\widetilde{M}^{f, \bar{n}\left(K\right), R}_{t_{i}\wedge\widetilde{\tau}^{\bar{n}\left(K\right),R}} \right)} \right)  & = \E\left(\widetilde{M}^{f, \bar{n}\left(K\right), R}_{s\wedge\widetilde{\tau}^{\bar{n}\left(K\right), R}}\prod_{i\, = \, 1}^{q}{f_{i}\left(\widetilde{M}^{f, \bar{n}\left(K\right), R}_{t_{i}\wedge\widetilde{\tau}^{\bar{n}\left(K\right), R}} \right)} \right).
\end{align*}   
 From  Proposition \ref{Prop_Generateur_Lent_Decomposition} and (\ref{Eq_M3}), $\widetilde{\EE}_{t\wedge \widetilde{\tau}^{K,R}}^{f, K, R} = \frac{1}{K}O\left(\int_{0}^{t}{\left[1 + M_{2}\left(\widetilde{\mu}_{s}^{K, R} \right) \right]\II_{s\leqslant \widetilde{\tau}^{K, R}}\dd s} \right)$ and then satisfy the condition $\lim_{K\to + \infty}{\E\left(\sup_{0\leqslant t \leqslant T}{\left|\EE_{t\wedge \widetilde{\tau}^{K, R}}^{f, K, R} \right|} \right)} = 0$. As for all $i \in \{ 1, \cdots, q\}$, $f_{i}$ is \textsc{Lipschitz}, there exists a constant $C>0$ such that 
\begin{align*}
& \E\left(\left|f_{i}\left(\widetilde{M}^{f, \bar{n}\left(K\right), R}_{t_{i}\wedge\widetilde{\tau}^{\bar{n}\left(K\right), R}} - \EE_{t_{i}\wedge\tau^{\bar{n}\left(K\right), R}}^{f, \bar{n}\left(K\right), R}   \right) - f_{i}\left(\widetilde{M}^{f, \bar{n}\left(K\right), R}_{t_{i}\wedge\widetilde{\tau}^{\bar{n}\left(K\right), R}} \right) \right| \right) \\
& \hspace{3cm} \leqslant \frac{C}{\bar{n}\left(K\right)}\E\left(\int_{0}^{t_{i}  }{\left[1 + M_{2}\left(\mu_{r}^{\bar{n}\left(K\right), R} \right) \right]\II_{r \leqslant \widetilde{\tau}^{\bar{n}\left(K\right), R}} \dd r} \right)
\end{align*}
where the term of the right hand side of the previous inequality goes to $0$ when $K \to + \infty$ according to Lemma \ref{Lem_Controle_Integrale_Temps_M2}. We deduce that for all $u \in \{s, t\}$, \[\lim_{K\to+\infty}{\hspace{-0.2cm}\E\left(\widetilde{M}^{f, \bar{n}\left(K\right), R}_{u\wedge\widetilde{\tau}^{\bar{n}\left(K\right), R}}\prod_{i\, = \, 1}^{q}{f_{i}\left(\widetilde{M}^{f, \bar{n}\left(K\right), R}_{t_{i}\wedge\widetilde{\tau}^{\bar{n}\left(K\right), R}} \right)} \right)} = \lim_{K\to + \infty}{\hspace{-0.2cm}\E\left(\widetilde{N}^{f, \bar{n}\left(K\right), R}_{u\wedge\widetilde{\tau}^{\bar{n}\left(K\right), R}}\prod_{i\, = \, 1}^{q}{f_{i}\left(\widetilde{N}^{f, \bar{n}\left(K\right), R}_{t_{i}\wedge\widetilde{\tau}^{\bar{n}\left(K\right), R}} \right)} \right)}.\]
From Steps 5 and 6, we deduce that \[\E\left(\widetilde{N}^{f, R}_{t}\prod_{i\, = \, 1}^{q}{f_{i}\left(\widetilde{N}^{f, R}_{t_{i}} \right)} \right) = \E\left(\widetilde{N}^{f, R}_{s}\prod_{i\, = \, 1}^{q}{f_{i}\left(\widetilde{N}^{f, R}_{t_{i}} \right)} \right).\]
Hence, \[\E\left(N{}^{f, R}_{t}\prod_{i\, = \, 1}^{q}{f_{i}\left(N{}^{f, R}_{t_{i}} \right)} \right) = \E\left(N{}^{f, R}_{s}\prod_{i\, = \, 1}^{q}{f_{i}\left(N{}^{f, R}_{t_{i}} \right)} \right) .\]  Since the last property is true for all $q\in \N^{\star}$, $t_{1} \leqslant \cdots \leqslant t_{q} \leqslant s < t$ and for all bounded \textsc{Lipschitz} functions $\left(f_{i} \right)_{1\leqslant i \leqslant q}$, the monotone class theorem ensures us that \[ \E\left(N_{t}^{f, R} \left| \phantom{1^{1^{1^{1}}}} \hspace{-0.7cm} \right. \sigma\left(N_{u}^{f, R} \left| \phantom{1^{1^{1
}}} \hspace{-0.55cm} \right. u\leqslant s \right) \right) = N_{s}^{f, R}.\] Hence the announced result. \\

\textit{Step 8. Nullity of $N^{f, R}$.} On the one hand, from \textsc{Itô}'s formula {\color{blue} \cite[\color{black} Theorem 32 of Chapter II]{Protter}}, for all $f \in \CCCC^{2}_{b}\left(\T_{R}, \T_{R}\right)$ and $t \in [0, T]$
\begin{equation}
\begin{aligned}
f^{2}\left(\zeta_{t}^{R} \right) & = f^{2}\left(\zeta_{0}^{R} \right)  + 2 \int_{0}^{t}{\int_{\MM_{1}(\R)}^{}{f\left(\zeta_{s}^{R} \right)\LL_{\rm SLOW}f\left(\zeta^{R}_{s}, \mu \right)\Gamma^{R}\left(\dd s, \dd \mu \right)}  } \\
& \quad + 2\int_{0}^{t}{f\left(\zeta_{s}^{R} \right)\dd N_{s}^{f, R}} + \left\langle N^{f, R}\right\rangle_{t} + \sum_{0<s\leqslant t}^{}{\left(f\left(\zeta_{s}^{R}\right) -  f\left(\zeta_{s^{-}}^{R}\right) \right)^{2}}.
\end{aligned}
\label{Ito_A}
\end{equation}
On the other hand, applying (\ref{Martingale_N_t_f}) with $f^{2} \in \CCCC_{b}^{2}\left(\T_{R}, \T_{R} \right)$, we obtain that for all $t \in [0, T]$
\begin{equation}
f^{2}\left(\zeta_{t}^{R} \right) = f^{2}\left(\zeta_{0}^{R} \right) + \int_{0}^{t}{\int_{\MM_{1}(\R)}^{}{\LL_{\rm SLOW}f^{2}\left(\zeta_{s}^{R}, \mu \right)\Gamma^{R}\left(\dd s, \dd \mu \right)}} + N_{t}^{f^{2}, R}.
\label{Ito_B}
\end{equation}
Comparing (\ref{Ito_A}) and (\ref{Ito_B}) leads for all $t \in [0, T]$ to \[N_{t}^{f^{2}, R} - 2\int_{0}^{t}{f\left(\zeta_{s}^{R} \right)\dd N_{s}^{f, R}} = \left\langle N^{f, R}\right\rangle_{t} + \sum_{0<s\leqslant t}^{}{\left(f\left(\zeta_{s}^{R}\right) -  f\left(\zeta_{s^{-}}^{R}\right) \right)^{2}} \]
and thus by {\color{blue} \cite[\color{black} Theorem 4.1]{LeGall}} that $\Q-$a.s. for all $t \in [0, T]$, \[\left\langle N^{f, R}\right\rangle_{t} = - \sum_{0<s\leqslant t}^{}{\left(f\left(\zeta_{s}^{R}\right) -  f\left(\zeta_{s^{-}}^{R}\right) \right)^{2}} \leqslant 0 \] 
  so that $\left\langle N^{f, R}\right\rangle_{t} = 0$. Hence, $\Q-$a.s. $N^{f, R} = 0$ which completes the proof. \qedhere
\end{proof}

\section{Characterisation of the occupation measure limit on the torus\label{Section_6_Caract_Gamma_Limite}} 
Consider a probability measure $\Q$ on $\D\left([0, T],\T_{R}\right) \times \MM_{m}^{T}\left(\MM_{1}\left(\R \right) \right)$ and the canonical process $\left(\zeta^{R}, \Gamma^{R} \right)$ as in Theorem \ref{Thm_Kurtz_adapte}. The following lemma gives us a desintegration result of the occupation measure $\Gamma^{R}$ that we characterise below. 

\begin{Lem} Let $T\geqslant 0$ be fixed. With the notations of {\rm Theorem \ref{Thm_Kurtz_adapte}}, there exists a random probability measure-valued process $\left(\gamma_{t}^{R} \right)_{t \in [0, T]}$ that is predictable in $(\omega, t)$ and such that for all bounded measurable function $\psi : [0,T] \times \MM_{1}\left(\R \right) \to \T_{R}$, 
\begin{equation}
\int_{0}^{t}{\int_{\MM_{1}\left(\R \right)}^{}{\psi(s, \mu})\Gamma^{R}\left(\dd s, \dd \mu \right)} = \int_{0}^{t}{\int_{\MM_{1}\left(\R \right)}^{}{\psi(s, \mu)\gamma^{R}_{s}(\dd \mu)\dd s}}.
\label{Eq_Desintegration_Gamma}
\end{equation}
\label{Lem_Desintegration}
\end{Lem}

\begin{proof} The desintegration result of $\Gamma^{R}$ follows directly from {\color{blue} \cite[\color{black} Lemma 1.4]{Kurtz}}. \qedhere 
\end{proof}

\begin{Cor} Let $T\geqslant 0$ be fixed. With the notations of {\rm Theorem \ref{Thm_Kurtz_adapte}}, we have that $\zeta^{R} \in \CCCC^{0}\left([0,T], \T_{R} \right)$ is differentiable of derivative in $L^{1}(\R)$ $\Q-$a.s.
\label{Cor_Continuity}
\end{Cor}

\begin{proof}
Applying (\ref{Martingale_N_t_f}) with $f = \id \in \CCCC_{b}^{1}\left(\T_{R}, \T_{R} \right)$, we deduce from Theorem \ref{Thm_Kurtz_adapte}, (\ref{Generateur_LENT}) and (\ref{Eq_Desintegration_Gamma}), $\Q-$a.s., for all $t \in [0, T]$
\begin{align*}
\zeta_{t}^{R} = \zeta_{0}^{R} + \int_{0}^{t}{\left(\int_{\MM_{1}(\R)}^{}{M_{2}\left(\mu \right)\gamma_{s}^{R}\left(\dd \mu \right)}\right)\partial_{1}\Fit\left(\zeta_{s}^{R}, \zeta_{s}^{R} \right)\dd s}.
\end{align*}
From Assumptions \textbf{(A)} and (\ref{Controle_E_M2_Gamma}), the integrand of the previous time integral is in $L^{1}(\R)$ $\Q-$a.s. Hence, the announced result follows from the fundamental theorem of calculus.
\end{proof}

We now want to characterise the limiting value $\Gamma^{R}\left(\dd t, \dd \mu\right) = \gamma_{t}^{R}\left(\dd \mu\right)\dd t$ under $\Q$. 

\begin{Prop}
With the notations of {\rm Theorem \ref{Thm_Kurtz_adapte}}, for a.e. $t \in [0, T]$, $\Q-$a.s., $\gamma_{t}^{R} = \pi^{\lambda\left(\zeta^{R} \right)}$ where $\pi^{\lambda}$ is the unique invariant probability measure of the centered \textsc{Fleming-Viot} process with resampling rate $\lambda$ (see {\color{blue} \rm {\cite[\color{black} Section 4]{Champagnat_Hass_FVr_2022}}}).
\label{Prop_Caract_Gamma_Limite}
\end{Prop}

It is here that we exploit ergodicity properties for the fast limit component. The proof of Proposition \ref{Prop_Caract_Gamma_Limite}, given in Section \ref{Sous_Section_6_1_Proof_Main_result} is based on the following technical lemma giving a characterisation of $\pi^{\lambda}$ and proved in Section \ref{Sous_Section_6_3_Proba_Invariante_Dawson}. To state this lemma, let us first recall from (\ref{Eq_fonctions_test_F_f_n}) the definition of polynomials in $\mu$:
\[P_{f, n}\left(\mu \right) := \left\langle f, \mu^{n} \right\rangle := \int_{\R}^{}{\cdots \int_{\R}^{}{f\left(x_{1}, \cdots, x_{n} \right)\mu\left(\dd x_{1} \right)\cdots \mu\left(\dd x_{n} \right)} } \]
with $n \in \N^{\star}$, $\mu \in \MM_{1}^{c}(\R)$,  $f \in \CCCC^{3}_{b}\left(\R^{n}, \R \right)$. For all $n \in \N^{\star}$, for any function $f : \R^{n} \to \R$ whose second derivatives exist, we denote by $\Hess(f) := \left(\partial_{ij}^{2}f \right)_{1\leqslant i, j \leqslant n}$ the Hessian matrix of $f$. For all $n \in \N^{\star}$, let us denote by $\CCCC^{2}_{\left\|\cdot\right\|}\left(\R^{n}, \R \right)$ the set
 \begin{align*}
& \left\{ f \in \CCCC^{2}\left(\R^{n}, \R\right) \left| \phantom{1^{1^{1^{1}}}} \hspace{-0.7cm} \right. \exists\, C >0, \ \forall x \in \R^{n}, \right.  \\
&  \hspace{4cm} \left.  \left|f(x) \right| + \left\|\nabla f(x) \right\|_{\infty} + \left\|\Hess (f)(x)  \right\|_{\infty} \leqslant C \left(1 + \left\|x \right\|_{\infty}^{2} \right)  \right\}.
 \end{align*}

From (\ref{Def_Gene_FVc_polynome}), we can see that if $f \in \CCCC^{4}_{b}\left(\R^{n}, \R \right)$, $\LL_{\rm{FVc}}^{\lambda}P_{f, n}\left(\mu \right)$ is a polynomial in $\mu$ of the form $P_{J, n+1}\left(\mu\right)$ for some function $J \in \CCCC^{2}_{\left\|\cdot \right\|}\left(\R^{n+1}, \R\right)$. We recall from {\color{blue} \cite[\color{black} Proposition 2.11]{Champagnat_Hass_FVr_2022}} that if $\mu \in \MM_{1}^{c,5}(\R)$, then $\sup_{0\leqslant t\leqslant T}{\E_{\mu}\left(M_{5}\left(X_{t} \right) \right)} < \infty$ where $\left(X_{t} \right)_{t\geqslant0} $ denotes the centered \textsc{Fleming-Viot} process with resampling rate $\lambda$. Since, by (\ref{Def_Gene_FVc_polynome}), for all $\mu \in \MM_{1}^{c, 4}(\R)$, \[\left|\LL_{\rm FVc}P_{J, n+1}(\mu) \right| \leqslant C \left(1 + M_{4}\left(\mu \right) \right),\] for some constant $C>0$, we can apply the martingale problem (\ref{PB_Mg_FVc_Polynomes}) to the function $P_{J, n+1}(\mu)$ using classical localisation techniques to obtain that the process $\left(\widehat{M}_{t}^{P_{J, n+1}} \right)_{t\geqslant 0}$ defined by 
\begin{equation}
\widehat{M}_{t}^{P_{J, n+1}} := P_{J, n+1}\left(X_{t} \right) -  P_{J, n+1}\left(X_{0} \right) - \int_{0}^{t}{\LL_{\rm FVc}^{\lambda}P_{J, n+1}\left(X_{s} \right)\dd s}
\label{PB_Mg_FVc_Polynomes_J}
\end{equation}
is a $\P_{\mu}-$martingale for all $\mu \in \MM_{1}^{c,5}(\R)$.

\begin{Lem}
 {\color{black} Let $\lambda \in \R$ be fixed}. If $\gamma \in \MM_{1}\left(\MM_{1}(\R)\right)$ satisfies $\int_{\MM_{1}(\R)}^{}{M_{4}\left(\mu \right)\gamma\left(\dd \mu \right)} < \infty$ and
\begin{align*}
%\textbf{\rm{\textbf{(a)}}} \quad & \forall \ell \in \N^{\star}, \quad \int_{\MM_{1}(\R)}^{}{M_{2 \ell}\left(\mu \right)\gamma\left(\dd \mu \right)} < \infty \\
%\textbf{\rm{\textbf{(b)}}} \quad & 
\forall n \in \N^{\star}, \ \forall f \in \CCCC_{\| \cdot \|}^{2}\left(\R^{n}, \R\right), \quad \int_{\MM_{1}(\R)}^{}{\LL_{\rm FVc}^{\lambda}{P_{f, n}}\left(\mu \right)\gamma\left( \dd \mu\right)} = 0,
%\textbf{\rm{\textbf{(b)}}} \quad & \forall n \in \N^{\star},\ \forall t\geqslant 0, \ \forall f_{t} \in \CCCC^{2}\left(\R^{n}, \R\right) \ {\rm satisfying} \forall x \in \R^{n}, \ f_{t}(x) \leqslant  , \quad \int_{\MM_{1}(\R)}^{}{\LL_{\rm FVc}^{\lambda}{P_{f, n}}\left(\mu \right)\gamma\left( \dd \mu\right)} = 0,
\end{align*}
then, $\gamma = \pi^{\lambda}$. 
\label{Lem_proba_invariante}
\end{Lem}

\subsection{Proof of Proposition \ref{Prop_Caract_Gamma_Limite}\label{Sous_Section_6_1_Proof_Main_result}}
Let us define for all $\ell \in \N$, $\CCCC_{K}^{\ell}\left(\R^{n}, \R \right)$ the space of real functions of class $\CCCC^{\ell}\left(\R^{n}, \R\right)$ with compact support. From Lemma \ref{Lem_PB_Mg_Polynomes} and (\ref{Temps_arret_check_chapeau}), for all $t\in [0, T]$, for all $n \in \N^{\star}$ we have for all $f \in \CCCC_{K}^{3}\left(\R^{n}, \R \right)$ that 
\begin{align*}
K^{2}{\color{black}\sigma_{K}^{2}}M_{t\wedge\tau^{K, R}}^{K, P_{f, n}} & = \sum_{i \, = \, 1}^{3}{{\rm{\textbf{\rm{\textbf{(A)}}}}}^{K, R}_{i}\left(t\wedge\tau^{K,R}\right)} 
\end{align*}
where for all $t \leqslant \tau^{K,R} \wedge T $
\begin{align*}
{\rm{\textbf{\rm{\textbf{(A)}}}}}^{K, R}_{1}(t) & :=  K^{2}{\color{black}\sigma_{K}^{2}}\left(P_{f, n}\left( \mu_{t}^{K, R} \right) - P_{f, n}\left(\mu_{0}^{K, R} \right) \right), \\
{\rm{\textbf{\rm{\textbf{(A)}}}}}^{K, R}_{2}(t)  & := - \int_{0}^{t}{\int_{\MM_{1}\left(\R\right)}^{}{\theta\left(z_{s}^{K, R} \right)m_{2}\left(z_{s}^{K, R} \right)\LL_{\rm FVc}^{\lambda\left(z_{s}^{K, R} \right)}P_{f, n}\left(\mu \right)\Gamma^{K, R}\left(\dd s, \dd \mu \right)} }, \\
{\rm{\textbf{\rm{\textbf{(A)}}}}}^{K, R}_{3}(t)  & :=  O\left(\frac{1}{\sqrt{K}} + {\color{black}\sigma_{K}} K^{\frac{3}{2}+\varepsilon} + \int_{0}^{t}{\frac{M_{3}\left(\mu_{s}^{K, R} \right)}{K}\, \dd s} \right),
\end{align*} 
is a martingale. Since ${\color{black}\sigma_{K}} K \to 0$ by Assumption (\ref{Hypothese_Gamme_sigma}), $ {\rm{\textbf{\rm{\textbf{(A)}}}}}^{K, R}_{1}\left(t\wedge\tau^{K,R}\right) \to 0$ when $K \to +\infty$. From Corollary \ref{Cor_Integrale_temps_M2_UI}, the sequence $\left({\rm{\textbf{\rm{\textbf{(A)}}}}}^{K, R}_{2}\left(t\wedge\tau^{K} \right) \right)_{t\in [0, T], K\in \N^{\star}}$ is uniformly integrable and converges in law, when $K \to + \infty$, to 
\begin{equation*}
M_{t}^{P_{f, n}} := -\int_{0}^{t}{\theta\left(\zeta_{s}^{R} \right)m_{2}\left(\zeta_{s}^{R} \right)\int_{\MM_{1}(\R)}^{}{\LL_{\rm FVc}^{\lambda\left(\zeta_{s}^{R} \right)}P_{f, n}\left(\mu \right)\Gamma^{R}\left(\dd s, \dd \mu \right)}}.
\label{Notation_Mg_limite_A3}
\end{equation*} 
Note that, from \textsc{Cauchy-Schwarz}'s inequality and %Lemma \ref{Lem_Controle_M6_et_M4M2}, 
{\color{black} Lemma \ref{Lem_Controle_Integrale_Temps_M2}}, there exists a constant $C>0$ such that  
\[\E\left(\sup_{0\leqslant t \leqslant T}{\int_{0}^{t\wedge \tau^{K, R}}{\frac{M_{3}\left(\mu_{s}^{K,R} \right)}{K}} }\dd s \right)^{2} \leqslant \frac{T}{K^{2}}\E\left(\int_{0}^{T\wedge \tau^{K, R}}{M_{6}\left(\mu_{s}^{K, R} \right)\dd s}  \right) \leqslant {\color{black}\frac{CT^2}{K^{2}}}. \]
Hence, $\lim_{K\to + \infty}{\sup_{0\leqslant t \leqslant T}{{\rm{\textbf{\rm{\textbf{(A)}}}}}^{K, R}_{3}(t\wedge \tau^{K, R})}} = 0$ in $L^{2}(\R)$. In particular, we deduce that a subsequence of $\left( \sup_{0\leqslant t \leqslant T}{{\rm{\textbf{\rm{\textbf{(A)}}}}}^{K, R}_{3}(t\wedge \tau^{K, R})} \right)_{K\in \N^{\star}}$ converges almost surely to $0$ and the family $\left( \sup_{0\leqslant t \leqslant T}{{\rm{\textbf{\rm{\textbf{(A)}}}}}^{K, R}_{3}(t\wedge \tau^{K, R})} \right)_{K\in \N^{\star}}$ is uniformly integrable along this subsequence {\color{blue} \cite[\color{black} Theorem 13.7]{Williams_1991}}. Using the same method based on \textsc{Skorohod}'s representation theorem as in the proof of Theorem \ref{Thm_Kurtz_adapte}, we deduce that the process $\left( M_{t}^{P_{f, n}}\right)_{t\geqslant 0}$ is a $\Q-$martingale.  As it  is also a continuous and finite variation process by Lemma \ref{Lem_Desintegration}, it must hence be  $\Q-$a.s. null {\color{blue} \cite[\color{black} Theorem 4.1]{LeGall}}. Hence, using Lemma \ref{Lem_Desintegration} again,  we have proved that
\begin{equation*}
\forall f \in \CCCC_{K}^{3}\left(\R^{n}, \R \right), \ \Q-{\rm a.s.}, \ \dd t-{\rm a.e.}, \quad  \int_{\MM_{1}(\R)}^{}{\LL_{\rm FVc}^{\lambda\left(\zeta_{t}^{R} \right)}P_{f, n}\left(\mu \right)\gamma_{t}^{R}\left(\dd \mu \right)} = 0.
\label{Relation_egale_0}
\end{equation*}

The space $\CCCC^{2}_{K}\left(\R^{n}, \R \right)$ equipped with the norm $$\left\|f \right\|_{W_{0}^{2, \infty}} = \left\|f \right\|_{\infty} +  \left\|\nabla f \right\|_{\infty} +  \left\|\Hess(f) \right\|_{\infty}$$ is separable. So, we can choose a dense countable family $\BB \subset \CCCC^{3}_{K}\left(\R^{n}, \R\right)$ such that \[\forall f \in \CCCC^{2}_{K}\left(\R^{n}, \R \right), \ \exists \, \left(f_{q} \right)_{q\in \N^{\star}} \in \BB^{\N^{\star}}, \quad f_{q} \xrightarrow[q\longrightarrow  + \infty ]{\|\cdot\|_{W_{0}^{2, \infty}}} f. \]
Then, \[ \Q-{\rm a.s.}, \ \dd t-{\rm a.e.}, \ \forall q \in \N^{\star},  \quad  \int_{\MM_{1}(\R)}^{}{\LL_{\rm FVc}^{\lambda\left(\zeta_{t}^{R} \right)}P_{f_{q}, n}\left(\mu \right)\gamma_{t}^{R}\left(\dd \mu \right)} = 0.\]
From (\ref{Controle_E_M2_Gamma}) and Lemma \ref{Lem_Desintegration}, we have that  $\int_{\MM_{1}(\R)}^{}{M_{4}\left(\mu \right)\gamma_{t}^{R}(\dd \mu)} < \infty$. As for all $q\in \N^{\star}$, 
\[\LL_{\rm FVc}^{\lambda\left(\zeta_{t}^{R} \right)}P_{f_{q}, n}\left(\mu \right) \leqslant C_{1}\left(1 + M_{2}\left(\mu \right) \right) \left\|f_{q} \right\|_{W_{0}^{2,\infty}}. \]
for some constant $C_{1}>0$, we obtain by the dominated convergence theorem that 
\begin{equation}
\Q-{\rm a.s.}, \ \dd t-{\rm a.e.}, \ \forall f \in \CCCC_{K}^{2}\left(\R^{n}, \R \right),  \quad  \int_{\MM_{1}(\R)}^{}{\LL_{\rm FVc}^{\lambda\left(\zeta_{t}^{R} \right)}P_{f, n}\left(\mu \right)\gamma_{t}^{R}\left(\dd \mu \right)} = 0.
\label{Relation_egale_0_bis}
\end{equation}
Let us consider $f \in \CCCC^{2}_{\|\cdot\|}\left(\R^{n}, \R \right)$ and for all $q \in \N^{\star}$, $x \mapsto \chi_{q}(x) = \exp\left(-\frac{1}{q^{2} - \left\|x \right\|_{\infty}^{2}} \right)\II_{\left\|x \right\|_{\infty} < q}$ of class $\CCCC^{\infty}\left(\R^{n}, \R \right)$ with compact support. Then, for all $q \in \N^{\star}$, $f\chi_{q} \in \CCCC_{K}^{2}\left(\R^{n}, \R\right)$. Noting that, $\left|\int_{\MM_{1}(\R)}^{}{\LL_{\rm FVc}^{\lambda\left(\zeta_{t}^{R} \right)}P_{f\chi_{q}, n}\left(\mu \right)\gamma_{t}^{R}\left(\dd \mu \right)} \right|$ is dominated by $C_{2} \int_{\MM_{1}(\R)}^{}{\left(1 + M_{4}\left(\mu \right) \right)\gamma_{t}^{R}(\dd \mu)}$ for some constant $C_{2}>0$, and since $\int_{\MM_{1}(\R)}^{}{M_{4}\left(\mu \right)\gamma_{t}^{R}(\dd \mu)} < \infty$, we deduce from (\ref{Relation_egale_0_bis}) applied to $f\chi_{q}$, by the dominated convergence theorem, that 
\[\Q-{\rm a.s.}, \ \dd t-{\rm a.e.}, \ \forall f \in \CCCC^{2}_{\|\cdot\|}\left(\R^{n}, \R \right),  \quad  \int_{\MM_{1}(\R)}^{}{\LL_{\rm FVc}^{\lambda\left(\zeta_{t}^{R} \right)}P_{f, n}\left(\mu \right)\gamma_{t}^{R}\left(\dd \mu \right)} = 0. \]
Therefore, it follows from Lemma \ref{Lem_proba_invariante} that $\Q-$a.s, $\dd t-$a.e., $\gamma_{t}^{R}\left(\dd \mu\right) =  \pi^{\lambda\left(\zeta_{t}^{R} \right)}\left(\dd \mu\right)$ which concludes the proof.

\subsection{Proof of Lemma \ref{Lem_proba_invariante}\label{Sous_Section_6_3_Proba_Invariante_Dawson}}
In Section \ref{Sous_sous_section_6.3.1_Extension_FVc_dualite} we extend some duality results for the centered \textsc{Fleming-Viot} process, obtained in {\color{blue} \cite{Champagnat_Hass_FVr_2022}}, which are be useful to prove Lemma \ref{Lem_proba_invariante} in Section \ref{Sous_sous_section_6.3.2_Preuve_effective}. 

\subsubsection{Extension of the duality result for the centered \textsc{Fleming-Viot} process\label{Sous_sous_section_6.3.1_Extension_FVc_dualite}}

Let us recall that the dual process $\left(\xi_{t} \right)_{t\geqslant 0}$ of the centered \textsc{Fleming-Viot} process $\left(X_{t} \right)_{t\geqslant 0}$ with resampling rate $\lambda$, on the state space $\bigcup_{n\in \N^{\star}}{\CCCC^{2}\left(\R^{n}, \R \right)}$, obtained in {\color{blue} \cite[\color{black} Section 3.2]{Champagnat_Hass_FVr_2022}} is defined as below.  \\

Let us consider $M := \left(M(t) \right)_{t\geqslant 0}$ a \textsc{Markov}'s birth and death process in $\N$ whose transition rates $q_{i,j}$ from $i$ to $j$ are given by: 
\[{\rm{\textbf{\rm{\textbf{(1)}}}}} \ \ q_{n, n+1} = \lambda n^{2} \qquad \quad {\rm{\textbf{\rm{\textbf{(2)}}}}} \ \ q_{n,n-1} = \lambda n(n-1) \qquad \quad {\rm{\textbf{\rm{\textbf{(3)}}}}} \ \ q_{i,j} = 0 \ {\rm{otherwise}}.\]

For all $M(0) \in \N^{\star}$, $\xi_{0} \in \CCCC^{2}_{b}\left(\R^{M(0)}, \R \right)$ and $\lambda >0$, we define 
\begin{multline}
\xi_{t}  := T_{\lambda}^{\left(M\left(\tau_{n} \right) \right)}\left(t-\tau_{n}\right)\Lambda_{n} T_{\lambda}^{\left(M\left(\tau_{n-1} \right) \right)}\left(\tau_{n}-\tau_{n-1}\right)\Lambda_{n-1} \cdots  \Lambda_{1}  T_{\lambda}^{\left(M\left(0 \right) \right)}\left(\tau_{1}\right)\xi_{0}, \\
   \tau_{n} \leqslant t < \tau_{n+1}, \ n \in \N , \label{Processus_dual_explicite}
\end{multline}
where $\left(\tau_{n}\right)_{n\in \N}$ is the sequence of jump times of %the birth-death process 
$M$ with $\tau_{0} = 0$, $\left(T_{\lambda}^{(n)}(t) \right)_{t\geqslant 0}$ is the semi-group of operator associated to the generator $B^{(n)}_{\lambda}$ given by (\ref{Operateur_B_n}) and where $\left(\Lambda_{n}\right)_{n\in \N}$ is a sequence of random operators.  These are conditionally independent given $M$ and satisfy for all $k \in \N$, $n\geqslant 1$ and $1\leqslant i \neq j \leqslant n$, 
\begin{equation*}
\P\left(\Lambda_{k} = \Phi_{i,j} \left| \phantom{1^{1^{1^{1}}}} \hspace{-0.7cm} \right. \left\{M\left(\tau_{k}^{-} \right) = n , M\left(\tau_{k} \right) = n-1\right\} \right)  = \frac{1}{n(n-1)} 
\label{Saut_Phi_i_j}
\end{equation*}
and for all $n \geqslant 1$ and $1\leqslant i, j \leqslant n$, 
\begin{equation*}
\P\left(\Lambda_{k} = K_{i,j} \left| \phantom{1^{1^{1^{1}}}} \hspace{-0.7cm} \right. \left\{M\left(\tau_{k}^{-} \right) = n , M\left(\tau_{k} \right) = n+1\right\} \right)  = \frac{1}{n^{2}},
 \label{Saut_K_i_j}
\end{equation*} 
where $\Phi_{i,j}$ and $K_{i,j}$ are respectively defined in (\ref{Expression_Phi_ij}) and (\ref{Operateur_K_ij}). Note that if $M(0) = 1$, the dual process can only jump from $\xi_{0}$ to $K_{ij}\xi_{0}$. As for all $t\geqslant 0$, $\xi_{t} \in \CCCC^{2}\left(\R^{M(t)}, \R \right)$, $K_{i,j}\xi_{t}$ is well defined. Moreover, the dual process $\left(\xi_{t} \right)_{t\geqslant 0}$ with initial condition $\xi_{0}$ is constructed on the same probability space and independently of the centered \textsc{Fleming-Viot} process $\left(X_{t} \right)_{t\geqslant 0}$ with resampling rate $\lambda$ and initial condition $\mu \in \MM_{1}^{c,2}(\R)$. We shall denote by $\P_{\left(\mu, \xi_{0} \right)}$, the law of the couple $\left(\left(X_{t}, \xi_{t} \right)\right)_{t\geqslant 0}$ on this probability space. \\

We denote by $S_{t}$ the number of jumps of the process $M$ on $[0,t]$. We start with a result giving bounds on the dual process, which is an extension of similar estimates obtained in {\color{blue} \cite[\color{black} Lemma 6.4]{Champagnat_Hass_FVr_2022}}.

\begin{Lem} For all $\xi_{0} \in \CCCC^{2}_{b}\left(\R^{M(0)}, \R\right)$ there exists $C_{0}: \bigcup_{k \in \N^{\star}}{(0, +\infty)^{k-1} \times \{k \}} \to \R_{+}$, locally bounded,  such that 
\[C_{0}\left(\tau_{1}, 1 \right) := \left\|\xi_{0} \right\|_{\infty} + \frac{3}{M(0)}\left\| \Hess\left(\xi_{0} \right) \right\|_{\infty} \]
and, for all $\left(t_{j} \right)_{j\in \N} \in (0,+\infty)^{\N}$, $k\mapsto C_{0}\left(\left(t_{i} \right)_{0\leqslant i \leqslant k-1}, k \right)$ is non-decreasing and satisfying 
\begin{equation}
\begin{aligned}
& \forall k \in \N^{\star}, \ \forall t \leqslant \tau_{k}, \ \forall x \in \R^{M(t)}, \\
& \hspace{4cm} \left| \xi_{t}(x)\right| \leqslant C_{0}\left(\left(\tau_{i+1} - \tau_{i} \right)_{0\leqslant i \leqslant k-1}, k \right)\left(1 + \left\|x \right\|_{\infty}^{2S_{t}} \right). 
\end{aligned}
\label{Eq_dual_Lemma_6_4_FVc}
\end{equation}
\label{Bornes_dual}
\end{Lem}

\begin{proof}  By induction on $k\in \N^{\star}$, we prove the property 
\begin{align*}
(\QQ_{k}) : \quad & \forall x \in \R^{M\left(\tau_{k} \right)}, \quad \left|\xi_{\tau_{k}}(x) \right| \leqslant C_{0}(\left(\tau_{i+1} - \tau_{i}\right)_{0\leqslant i\leqslant k-1}, k) \left(1 + \left\|x \right\|_{\infty}^{2k}  \right).
\end{align*}
\textit{Step 1. Initial case: computation of $C_{0}\left(\tau_{1}, 1 \right)$.}  Let us recall some notations of {\color{blue} \cite[\color{black} Theorem 6.1]{Champagnat_Hass_FVr_2022}}. For all $n\in \N^{\star}$, we denote by $\bm{1} \in \R^{n}$, the vector whose coordinates are all $1$.  For all $n \in \N^{\star}$, $t \geqslant 0$,  $x \in \R^{n}$ and $\lambda >0$, let us consider $g_{t,x}^{\lambda}$ the density of the Gaussian distribution $\NN^{(n)}\left(m_{t, x}^{\lambda}, \Sigma_{t}^{\lambda} \right)$ where $m_{t,x}^{\lambda} := x - \frac{\left(1 - \exp\left(-2\lambda n t \right) \right)}{n}(x \cdot \bm{1})\bm{1}$ and $\Sigma_{t}^{\lambda} := P\sigma_{t}^{\lambda}P^{-1}$ with
\[\sigma_{t}^{\lambda} := \begin{pmatrix}
\frac{1-\exp\left(-4\lambda nt\right)}{4\lambda n} & 0 & \dots  & \dots & 0 \vspace{0.1cm} \\
0 & t & 0 & \dots & 0 \\
\vdots & 0 & \ddots & \ddots & \vdots \\
\vdots & \vdots & \ddots & \ddots & 0 \\
0 & 0 & \dots & 0 & t
\end{pmatrix} \quad {\rm and} \quad P := \begin{pmatrix}
\frac{1}{\sqrt{n}} & \frac{1}{\sqrt{2}} & \cdots & \cdots & \frac{1}{\sqrt{n(n-1)}} \\
\vdots & - \sqrt{\frac{1}{2}} & \ddots & \cdots & \frac{1}{\sqrt{n(n-1)}} \\
\vdots & 0 & \ddots & \ddots & \vdots \\
\vdots & \vdots & \ddots & \ddots & \frac{1}{\sqrt{n(n-1)}} \\
\frac{1}{\sqrt{n}} & 0 & \cdots & 0 & - \sqrt{\frac{n-1}{n}}
\end{pmatrix}.\] 
Note that for all $i,j \in \{1, \cdots, n \}$, $\partial_{x_{i}}m_{t,x}^{\lambda} = \epsilon_{i} -  \frac{1 - \exp\left(-2\lambda n t \right)}{n}\bm{1}  $ and $\partial_{x_{i}x_{j}}^{2}m_{t,x}^{\lambda} = 0$  where $\left(\epsilon_{1},\cdots, \epsilon_{n}\right)$ is the canonical basis of $\R^{n}$. For all $\lambda >0$, the key identity for the sequel is  
\begin{equation}
\forall f \in L^{\infty}\left(\R^{n} \right), \ \forall t \geqslant 0, \ \forall x \in \R^{n}, \quad   T_{\lambda}^{(M(0))}(t)f(x) = \left(f \ast g_{t,0}^{\lambda}\right)\left(m_{t,x}^{\lambda} \right)
\label{Eq_clef_Semi_groupe}
\end{equation}
where $\ast$ stands for the convolution product (see (46) in {\color{blue} \cite[\color{black} Theorem 6.1]{Champagnat_Hass_FVr_2022}}). \\ 

On the one hand, note that for all $t<\tau_{1}$, $M(t) = M(0)$. In this case, for all $x \in \R^{M(0)}$, $\lambda >0$, $\xi_{t}(x) = T_{\lambda}^{\left(M(0)\right)}(t)\xi_{0}(x)$ and so, from {\color{blue} \cite[\color{black} Theomem 6.1]{Champagnat_Hass_FVr_2022}}, $\left|\xi_{t}(x) \right| \leqslant \left\|\xi_{0} \right\|_{\infty}$. On the other hand, at time $\tau_{1}$, we make a partition of cases according to whether the dual process loses or gains a variable. Let $i, j \in \left\{1, \cdots, M(0) \right\}$ be fixed.  

\begin{itemize}
\item \textit{Case $\Lambda_{1} = \Phi_{i,j}$.} In this case, $M(\tau_{1}) = M(0) - 1$ and we deduce from %the explicit expression 
(\ref{Processus_dual_explicite}) that  \[\forall x \in \R^{M(0)-1}, \qquad \xi_{\tau_{1}}(x) = \Phi_{i,j}T_{\lambda}^{\left(M(0) \right)}\left(\tau_{1} \right)\xi_{0}(x).\] By (\ref{Expression_Phi_ij}), we deduce that for all $x \in \R^{M(0)-1}$, $\left|\xi_{\tau_{1}}(x) \right| \leqslant \left\|\xi_{0} \right\|_{\infty}$.  \\
\item \textit{Case $\Lambda_{1} = K_{i,j}$.} In this case, $M(\tau_{1}) = M(0) + 1$ and we deduce from %the explicit expression 
(\ref{Processus_dual_explicite}) and (\ref{Operateur_K_ij}) then from {\color{blue} \cite[\color{black} Theorem 6.1 \textbf{(3)}]{Champagnat_Hass_FVr_2022}} and properties of the convolution product that for all $x \in \R^{M(0) + 1}$,
\begin{align*}
\xi_{\tau_{1}}(x) &  = K_{i,j}T_{\lambda}^{\left(M(0) \right)}\left(\tau_{1} \right)\xi_{0}(x) \\ 
& = \partial_{x_{i}x_{j}}T_{\lambda}^{\left(M(0) \right)}\left(\tau_{1} \right)\xi_{0}\left(\widetilde{x}\right)x_{M(0)+1}^{2} \\
& = \left(\partial_{x_{j}}m_{\tau_{1},\widetilde{x}}^{\lambda} \right)^{t}\left[\left(\Hess\left(\xi_{0}\right) \ast g_{\tau_{1},0}^{\lambda}\right)\left(m_{\tau_{1},\widetilde{x}}^{\lambda}\right)\partial_{x_{i}}m_{t,\widetilde{x}}^{\lambda}\right]x_{M(0)+1}^{2}
\end{align*}
where $\widetilde{x} = \left(x_{1}, \cdots, x_{M(0)} \right) \in \R^{M(0)}$. As, \[\left(\Hess\left(\xi_{0}\right) \ast g_{\tau_{1},0}^{\lambda}\right)\left(m_{\tau_{1},\widetilde{x}}^{\lambda}\right) = \left(\int_{\R^{M(0)}}^{}{\partial_{ij}^{2}\xi_{0}(u)g_{\tau_{1}, 0}^{\lambda}\left(m_{\tau_{1}, \widetilde{x}}^{\lambda} - u \right)\dd u} \right)_{1\leqslant i, j \leqslant M(0)}\]
we obtain that $\left\|\left(\Hess\left(\xi_{0}\right) \ast g_{\tau_{1},0}^{\lambda}\right)\left(m_{\tau_{1},\widetilde{x}}^{\lambda}\right) \right\|_{\infty} \leqslant \left\|\partial_{ij}^{2}\xi_{0} \right\|_{\infty}$. Noting that $$\left|\left(\partial_{x_{j}}m_{\tau_{1},\widetilde{x}}^{\lambda} \right)^{t}\partial_{x_{i}}m_{t,\widetilde{x}}^{\lambda} \right| \leqslant \frac{3}{M(0)},$$ we deduce that for all $ x \in \R^{M(0) + 1}$,
\[\left|\xi_{\tau_{1}}(x) \right| \leqslant \frac{3}{M(0)}\sup_{1\leqslant i,j \leqslant M(0)}{\left\|\partial_{ij}^{2}\xi_{0} \right\|_{\infty}}x_{M(0)+1}^{2} \leqslant \frac{3}{M(0)}\left\|\Hess(\xi_{0}) \right\|_{\infty}\left\|x \right\|_{\infty}^{2} \]
and $(\QQ_{1})$ follows. 
\end{itemize}

\textit{Step 2. Inductive Step.}  We assume that, for $k \in \N \, \backslash \, \{0,1 \}$, $(\QQ_{k-1})$ is satisfied and prove that $(\QQ_{k})$ is also. We make again a partition of cases according to whether the dual process loses or gains a variable. Let $i,j \in \left\{1, \cdots, M\left(\tau_{k-1} \right) \right\}$ be fixed.
\begin{itemize}
 \item \textit{Case $\Lambda_{k} = \Phi_{i,j}$ at the $k^{\rm{th}}$ jump.} In this case, $M(\tau_{k}) = M(\tau_{k-1}) - 1$ and we deduce from (\ref{Processus_dual_explicite}) that for all $x \in \R^{M\left(\tau_{k-1} \right)-1}$, $\xi_{\tau_{k}}(x) = \Phi_{i,j}T_{\lambda}^{\left(M(\tau_{k-1})\right)}(\tau_{k}-\tau_{k-1})\xi_{\tau_{k-1}}(x)$. 
By using (\ref{Expression_Phi_ij}) and  $\left(\QQ_{k-1} \right)$, we deduce from {\color{blue} \cite[\color{black} Corollary 6.2 \textbf{(1)}]{Champagnat_Hass_FVr_2022}} that for all $x \in \R^{M\left(\tau_{k-1} \right)-1}$, 
\begin{equation*}
\begin{aligned}
& \left|\Phi_{i,j}T_{\lambda}^{\left(M\left(\tau_{k-1} \right) \right)}\left(\tau_{k} - \tau_{k-1} \right)\xi_{\tau_{k-1}}(x) \right| \\
& \hspace{1cm} \leqslant C_{2}\left(\tau_{k} - \tau_{k-1}, M\left(\tau_{k-1}\right) \right){C}_{0}\left(\left(\tau_{i+1} - \tau_{i} \right)_{0\leqslant i \leqslant k-1}, k \right)\left(1 + \left\|x \right\|_{\infty}^{2(k-1)} \right),
\end{aligned}
\label{Ineq_Phi_ij_dual_Lemme_6_4}
\end{equation*}
where $C_{2}{C}_{0}$ is locally bounded. As $t\mapsto C_{2}\left(t, M\left(\tau_{k-1} \right)\right)$ is non-decreasing, we deduce $\left(\QQ_{k} \right)$ in that case. \\
\item \textit{Case $\Lambda_{k} = K_{i,j}$ at the $k^{\rm{th}}$ jump.}
In this case, $M(\tau_{k}) = M(\tau_{k-1}) + 1$.  From (\ref{Processus_dual_explicite}),  (\ref{Operateur_K_ij}) and {\color{blue} \cite[\color{black} Theorem 6.1 \textbf{(3)}]{Champagnat_Hass_FVr_2022}}, we have for all $x \in \R^{M\left(\tau_{k-1} \right) + 1}$,
\begin{align*}
\left|\xi_{\tau_{k}}(x) \right| & = \left|K_{i,j}T_{\lambda}^{\left(M\left(\tau_{k-1} \right) + 1 \right)}\left(\tau_{k} - \tau_{k-1} \right)\xi_{\tau_{k-1}}(x) \right| \\
& = \left|\left(\partial_{x_{j}}m^{\lambda}_{\tau_{k} - \tau_{k-1},\widetilde{x}} \right)^{t}\left[\left(\xi_{\tau_{k-1}} \ast \Hess\left(g_{\tau_{k} - \tau_{k-1}, 0}^{\lambda}\right) \right)(m^{\lambda}_{\tau_{k} - \tau_{k-1}, \widetilde{x}})\partial_{x_{i}}m^{\lambda}_{\tau_{k} - \tau_{k-1}, \widetilde{x}}\right] \right| \\ 
& \hspace{1cm} \times x_{M\left(\tau_{k-1} \right) + 1}^{2},
\end{align*}
where $\widetilde{x} = \left(x_{1}, \cdots, x_{M\left(\tau_{k-1} \right)} \right)^{t} \in \R^{M\left(\tau_{k-1} \right)}$. From $\left(\QQ_{k-1} \right)$ and {\color{blue} \cite[\color{black} Corollary 6.2 \textbf{(2)}]{Champagnat_Hass_FVr_2022}}, we deduce that 
\begin{equation*}
\begin{aligned}
& \left|K_{i,j}T_{\lambda}^{\left(M\left(\tau_{k-1} \right) + 1 \right)}\left(\tau_{k} - \tau_{k-1} \right)\xi_{\tau_{k-1}}(x) \right| \\ 
%& \hspace{1cm} \leqslant \widehat{C}_{0}\left(k, \left(\tau_{i+1} - \tau_{i} \right)_{1\leqslant i \leqslant k-1}, M\left(\tau_{k-1} \right) \right)\left(1 + \left\|x\right\|_{\infty}^{2k} \right).
& \hspace{1cm} \leqslant C_{3}\left(\tau_{k} - \tau_{k-1}, M\left(\tau_{k-1}\right) \right) C_{0}\left(\left(\tau_{i+1} - \tau_{i} \right)_{0\leqslant i \leqslant k-1}, k \right)   \left(1 + \left\|x\right\|_{\infty}^{2k} \right),
\end{aligned}
\label{Ineq_K_ij_dual_Lemme_6_4}
\end{equation*} 
where $C_{3}{C}_{0}$ is locally bounded and $\left(\QQ_{k} \right)$ follows in that case. 
\end{itemize}
We conclude by the principle of induction.  \\

\textit{Step 3. Proof of (\ref{Eq_dual_Lemma_6_4_FVc}) for $t < \tau_{k}$.} Note that from (\ref{Processus_dual_explicite}), for all $k \in \N$ and $t \in \, \left(\tau_{k}, \tau_{k+1} \right)$,   \[\forall x \in \R^{M\left(t \right)}, \qquad \xi_{t}(x) = T_{\lambda}^{\left(M\left(\tau_{k} \right) \right)}\left(t - \tau_{k} \right)\xi_{\tau_{k}}(x), \]
so the announced result follows from {\color{blue} \cite[\color{black} Corollary 6.2 \textbf{(1)}]{Champagnat_Hass_FVr_2022}}. \qedhere
\end{proof}

 Let us consider for all $k \in \N^{\star}$, $\ell, m \in \N$ the stopping times 
\begin{align*}
\vartheta_{k, \ell} & := \inf{\left\{ t \geqslant 0\left. \phantom{1^{1^{1^{1}}}} \hspace{-0.6cm} \right| S_{t}  \geqslant k \quad {\rm{or}} \quad \exists s \in [0,t], \ \left\langle  \xi_{s}, X_{t-s}^{M(s)}  \right\rangle \geqslant \ell \right\}}, \\
\vartheta_{k, m}'  & := \inf{\left\{ t \geqslant 0\left. \phantom{1^{1^{1^{1}}}} \hspace{-0.6cm} \right| S_{t}  \geqslant k \quad {\rm{or}} \quad C_{0}\left(\left(\tau_{i+1} - \tau_{i} \right)_{0\leqslant i \leqslant k-1}, k \right) \geqslant m \right\}}, \\
\vartheta_{k, \ell, m}' & := \inf{\left\{ t \geqslant 0\left. \phantom{1^{1^{1^{1}}}} \hspace{-0.6cm} \right| S_{t}  \geqslant k \quad {\rm{or}} \quad \exists s \in [0,t], \ \left\langle  \xi_{s}, X_{t-s}^{M(s)}  \right\rangle \geqslant \ell \right. } \\
& \hspace{2.95cm} \left. \phantom{1^{1^{1^{1}}}}  \quad {\rm{or}} \quad C_{0}\left(\left(\tau_{i+1} - \tau_{i} \right)_{0\leqslant i \leqslant k-1}, k \right) \geqslant m \right\}.
\end{align*}

As $\left(M(t), \xi_{t} \right)_{t\geqslant 0}$ is independent of $\left(X_{t} \right)_{t\geqslant 0}$, note that for all $k \in \N^{\star}$, $m\in \N$, $\vartheta_{k, m}'$ is independent of $\left(X_{t} \right)_{t\geqslant 0}$. Recall that for all $t\geqslant 0$, $\xi_{t}\in \CCCC^{2}\left(\R^{M(t)}, \R \right)$. For all $n \geqslant M(t)$, for all $x \in \R^{n}$, we denote by $\xi_{t}^{(n)}(x) = \xi_{t}^{(n)}\left(x_{1}, \cdots, x_{n} \right) := \xi_{t}\left(x_{1}, \cdots, x_{M(t)}\right)$ %the vector of $\R^{n}$ whose first coordinates are $\xi_{t}(x)$ and whose last coordinates are $1$, 
so that for all $\mu \in \MM_{1}^{c, 2}(\R)$, \[\left\langle \xi_{t}^{(n)}, \mu^{n} \right\rangle = \left\langle \xi_{t}, \mu^{M\left(t \right)} \right\rangle.\] 

 The next lemma is an extension of the duality identity proved in {\color{blue} \cite{Champagnat_Hass_FVr_2022}}.

\begin{Lem} Given $\left(X_{t}\right)_{t\geqslant 0}$, $\left(\xi_{t}\right)_{t\geqslant 0}$ as above with $X_{0} := \mu \in \MM_{1}^{c, 2k}(\R)$, $k \in \N^{\star}$ and $\xi_{0} \in \CCCC_{b}^{2}\left(\R^{M(0)}, \R\right)$, %has all its moments finite, 
we have that $\P_{\left(\mu, \xi_{0}\right)}-$a.s., for all $t \geqslant 0$, for all $m \in \N$, 
\begin{equation}
\begin{aligned}
\E_{\left(\mu, \xi_{0}\right)}\left(\left\langle \xi_{0}, X_{t\wedge \vartheta_{k, m}'}^{M(0)} \right\rangle \right) = \left\langle \E_{\xi_{0}}\left(\xi_{t\wedge \vartheta_{k, m}'}^{\left(M(0) + k\right)}\exp\left(\lambda \int_{0}^{t\wedge \vartheta_{k, m}'}{M^{2}(u)\dd u} \right) \right), \mu^{M(0) + k}\right\rangle.
\end{aligned}
\label{Eq_Lem_dualite_new}
\end{equation}
\label{Lem_dualite_new}
\end{Lem}

Compared to the duality identity of {\color{blue} \cite{Champagnat_Hass_FVr_2022}}, the key point of this lemma is that the right-hand side of the last equation, is a polynomial in $\mu$. 

\begin{proof} Let us consider $t\geqslant 0$, $k \in \N^{\star}$, $\ell, m \in \N$. As $\vartheta_{k, \ell, m}' \leqslant \vartheta_{k, \ell}$, the weak duality identity of {\color{blue} \cite[\color{black} Theorem 3.4]{Champagnat_Hass_FVr_2022}} implies that: 
\begin{equation*}
\begin{aligned}
& \E_{\left(\mu, \xi_{0}\right)}\left(\left\langle \xi_{0}, X_{t\wedge \vartheta_{k, \ell, m}'}^{M(0)} \right\rangle \right) \\ 
& \hspace{1.5cm} = \E_{\left(\mu, \xi_{0}\right)}\left(\left\langle \xi_{t\wedge \vartheta_{k, \ell, m}'}, \mu^{M\left(t\wedge \vartheta_{k, \ell, m}' \right)} \right\rangle \exp\left(\lambda \int_{0}^{t\wedge \vartheta_{k, \ell, m}'}{M^{2}(u)\dd u} \right) \right) \\
& \hspace{1.5cm} = \E_{\left(\mu, \xi_{0}\right)}\left(\left\langle \xi_{t\wedge \vartheta_{k, \ell, m}'}^{\left(M(0) + k\right)}, \mu^{M(0) + k} \right\rangle \exp\left(\lambda \int_{0}^{t\wedge \vartheta_{k, \ell, m}'}{M^{2}(u)\dd u} \right) \right). 
\end{aligned}
\end{equation*}
Now, from Lemma \ref{Bornes_dual} and the definition of $\vartheta_{k, \ell, m}'$, note that for all $x \in \R^{M(0) + k}$, 
\begin{equation}
\begin{aligned}
& \E_{\left(\mu, \xi_{0}\right)}\left( \xi_{t\wedge \vartheta_{k, \ell, m}'}^{\left(M(0) + k\right)}(x) \exp\left(\lambda \int_{0}^{t\wedge \vartheta_{k, \ell, m}'}{M^{2}(u)\dd u} \right) \right) \\
& \hspace{5cm} \leqslant m\left(1 + \left\|x \right\|_{\infty}^{2k} \right)\exp\left(\lambda t \left(M(0) + k \right)^{2} \right).
\end{aligned}
\label{Eq_theta_k_l_m_domination}
\end{equation}
Since $\mu \in \MM_{1}^{c, 2k}(\R)$, we deduce from \textsc{Fubini}'s theorem that 
\begin{align*}
& \E_{\left(\mu, \xi_{0}\right)}\left(\left\langle \xi_{0}, X_{t\wedge \vartheta_{k, \ell, m}'}^{M(0)} \right\rangle \right) \\
& \hspace{2cm} = \left\langle \E_{\left(\mu, \xi_{0}\right)}\left( \xi_{t\wedge \vartheta_{k, \ell, m}'}^{\left(M(0) + k\right)} \exp\left(\lambda \int_{0}^{t\wedge \vartheta_{k, \ell, m}'}{M^{2}(u)\dd u} \right) \right), \mu^{M(0) + k}  \right\rangle. 
\end{align*}

On the one hand, from {\color{blue} \cite[\color{black} Section 6.2.1]{Champagnat_Hass_FVr_2022}}, $\lim_{\ell \to + \infty}{\vartheta_{k, \ell, m}'} = \vartheta_{k, m}'$ $\P_{\left(\mu, \xi_{0} \right)}-$a.s., and since $\left(X_{t}\right)_{t\geqslant 0}$ has continous paths for the topology of weak convergence, we have $\P_{\left(\mu, \xi_{0} \right)}-$a.s.
\[\lim_{\ell \to + \infty}{\left\langle \xi_{0}, X_{t\wedge \vartheta_{k, \ell, m}'}^{M(0)} \right\rangle} =  \left\langle \xi_{0}, X_{t\wedge \vartheta_{k, m}'}^{M(0)} \right\rangle.\]
Therefore, we deduce from the \textsc{Lebesgue} dominated convergence theorem, that for all $\xi_{0} \in \CCCC^{2}_{b}\left(\R^{M(0)}, \R\right)$, 
\[\lim_{\ell \to + \infty}{\E_{\left(\mu, \xi_{0}\right)}\left(\left\langle \xi_{0}, X_{t\wedge \vartheta_{k, \ell, m}'}^{M(0)} \right\rangle \right)} = \E_{\left(\mu, \xi_{0}\right)}\left(\left\langle \xi_{0}, X_{t\wedge \vartheta_{k, m}'}^{M(0)} \right\rangle \right).\]
On the other hand, Lemma \ref{Bornes_dual} and the dominated convergence theorem imply that for all $x \in \R^{M(0) + k}$,
\begin{align*}
& \lim_{\ell \to + \infty}{\E_{\left(\mu, \xi_{0}\right)}\left( \xi_{t\wedge \vartheta_{k, \ell, m}'}^{\left(M(0) + k\right)}(x) \exp\left(\lambda \int_{0}^{t\wedge \vartheta_{k, \ell, m}'}{M^{2}(u)\dd u} \right) \right)} \\
& \hspace{2cm}  = \E_{\left(\mu, \xi_{0}\right)}\left( \xi_{t\wedge \vartheta_{k, m}'}^{\left(M(0) + k\right)}(x) \exp\left(\lambda \int_{0}^{t\wedge \vartheta_{k, m}'}{M^{2}(u)\dd u} \right) \right)
\end{align*} 
and the limit satisfies the inequality (\ref{Eq_theta_k_l_m_domination}).
As $\mu \in \MM_{1}^{c, 2k}(\R)$, we deduce again from the dominated convergence theorem that 
\begin{align*}
& \lim_{\ell \to + \infty}{\left\langle \E_{\left(\mu, \xi_{0}\right)}\left(\xi_{t\wedge \vartheta_{k, \ell, m}'}^{\left(M(0) + k\right)}\exp\left(\lambda \int_{0}^{t\wedge \vartheta_{k, \ell, m}'}{M^{2}(u)\dd u} \right) \right),\mu^{M(0) + k}\right\rangle} \\ 
& \hspace{2cm} = \left\langle \E_{\left(\mu, \xi_{0}\right)}\left(\xi_{t\wedge \vartheta_{k, m}'}^{\left(M(0) + k\right)}\exp\left(\lambda \int_{0}^{t\wedge \vartheta_{k, m}'}{M^{2}(u)\dd u} \right) \right), \mu^{M(0) +k}\right\rangle
\end{align*}
and the announced result follows from the fact that $\vartheta_{k, m}'$ is independent of $\left(X_{t} \right)_{t\geqslant 0}$, which implies 
\begin{align*}
& \E_{\left(\mu, \xi_{0}\right)}\left(\xi_{t\wedge \vartheta_{k, m}'}^{\left(M(0) + k\right)}(x)\exp\left(\lambda \int_{0}^{t\wedge \vartheta_{k, m}'}{M^{2}(u)\dd u} \right) \right) \\ 
& \hspace{3cm} = \E_{\xi_{0}}\left(\xi_{t\wedge \vartheta_{k, m}'}^{\left(M(0) + k\right)}(x)\exp\left(\lambda \int_{0}^{t\wedge \vartheta_{k, m}'}{M^{2}(u)\dd u} \right) \right).
\end{align*}
\end{proof}
From Lemma \ref{Bornes_dual}, we deduce that we can choose $m \in \N$ large enough such that $\vartheta'_{1, m} = \tau_{1}$.
In this case, the function inside the brackets in the left-hand side of (\ref{Eq_Lem_dualite_new}) is defined, for all $t\geqslant 0$ and for all $x \in \R^{M(0) + 1}$, by
\begin{equation*}
\psi_{t}(x):= \E_{\left(\mu, \xi_{0} \right)}\left(\xi_{t\wedge\tau_{1}}^{\left(M(0)+ 1 \right)}(x)\exp\left(\lambda\int_{0}^{t\wedge \tau_{1}}{M^{2}(u)\dd u} \right) \right). \label{Eq_Def_Psi_t}
\end{equation*}
For all $t \geqslant 0$, $n \in \N^{\star}$, $f \in \CCCC^{2}_{b}\left(\R^{n}, \R\right)$ and $\mu \in \MM_{1}^{c, 2}(\R)$ let us denote \[P_{\rm{FVc}}(t)P_{f, n}\left(\mu\right) := \E_{\mu}\left(P_{f, n}\left(X_{t}\right) \right)\] the semi-group of the centered \textsc{Fleming-Viot} process $(X_{t})_{t\geqslant 0}$ with resampling rate $\lambda$. \\ 

The following result is the main result of this section. It gives an extension of Lemma \ref{Bornes_dual} to a function which appears naturally in the proof of Lemma \ref{Lem_proba_invariante}. 
\begin{Prop}
For all $t \geqslant 0$, for all $\mu \in \MM_{1}^{c, 2}(\R)$, for all $\xi_{0} \in \CCCC^{2}_{b}\left(\R^{M(0)}, \R \right)$ \[P_{\rm{FVc}}(t)P_{\xi_{0}, M(0)}(\mu) = P_{V_{t}, M(0) + 1}\left(\mu \right) := \left\langle V_{t}, \mu^{M(0) + 1} \right\rangle \]
where 
\begin{equation}
V_{t} := \exp\left(\alpha_{0} t \right)\psi_{t} - \alpha_{0}\int_{0}^{t}{\exp\left(\alpha_{0} s \right)\psi_{s}\dd s}
\label{Eq_Def_V_t}
\end{equation}
and $\alpha_{0} := \lambda M(0)\left[2M(0)-1\right]$. In addition, for all $t\geqslant 0$, there exist a constant $C_{t}>0$ such that for all $x \in \R^{M(0) + 1}$, 
\begin{equation*}
\left|V_{t}(x)\right| \leqslant C_{t}\left( 1 + \|x \|^{2}_{\infty}\right).
\label{Eq_V_t_inf_1_plus_norme_x}
\end{equation*}
\label{Prop_Semi_groupe_T_FVc}
\end{Prop}

\begin{proof}
Let $t \geqslant 0$, $\mu \in \MM_{1}^{c, 2}(\R)$, $\xi_{0} \in \CCCC_{b}^{2}\left(\R^{M(0)}, \R \right)$ be fixed. Note that from Lemma \ref{Bornes_dual}, there exists a constant $\widetilde{C}_{t}$, depending only on $\xi_{0}$ and $t$ such that for all $x \in \R^{M(0)+1}$, 
\begin{equation}
\left|\psi_{t}(x) \right| \leqslant C_{0}\left(\tau_{1},1\right)\left(1+ \left\|x \right\|_{\infty}^{2} \right)\E_{\left(\mu, \xi_{0} \right)}\left(\exp\left(\lambda t M^{2}(0) \right) \right) \leqslant \widetilde{C}_{t}\left(1 + \left\|x \right\|_{\infty}^{2} \right). 
\label{Eq_Borne_phi_t}
\end{equation}

Now, from Lemma \ref{Lem_dualite_new} and the fact that $\tau_{1}$ is independent of $\left(X_{t} \right)_{t \geqslant 0}$, we obtain that 
\begin{align*}
\left\langle \psi_{t}, \mu^{M(0) + 1} \right\rangle = \E_{\left(\mu, \xi_{0} \right)}\left(P_{\xi_{0}, M(0)}\left(X_{t\wedge \tau_{1}} \right) \right) & = \E_{\xi_{0}}\left(P_{\rm{FVc}}\left(t\wedge \tau_{1} \right)P_{\xi_{0}, M(0)}\left(\mu \right) \right).
\end{align*}
As $\tau_{1}$ follows an exponential law of parameter $\alpha_{0}$ we have that 
\begin{equation}
\begin{aligned}
& \left\langle \psi_{t}, \mu^{M(0) + 1} \right\rangle  = P_{\rm{FVc}}\left(t\right)P_{\xi_{0}, M(0)}\left(\mu \right)\exp\left(-\alpha_{0} t \right) \\ 
& \hspace{4cm} + \alpha_{0}\int_{0}^{t}{P_{\rm{FVc}}\left(s\right)P_{\xi_{0}, M(0)}\left(\mu \right) \exp\left(-\alpha_{0} s \right)\dd s}.
\end{aligned}
\label{Eq_psi_t_Lem_6.6}
\end{equation}

Thanks to (\ref{Eq_Borne_phi_t}) and since $\mu$ has its moment of order 2 finite, the \textsc{Fubini} theorem ensures us that \[\left\langle  \int_{0}^{t}{\exp\left(\alpha_{0} s \right) \psi_{s} \dd s}, \mu^{M(0) + 1} \right\rangle = \int_{0}^{t}{\exp\left(\alpha_{0} s \right)\left\langle \psi_{s}, \mu^{M(0) + 1}  \right\rangle \dd s}. \]
From (\ref{Eq_psi_t_Lem_6.6}) we deduce that 
 \begin{equation}
 \begin{aligned}
& \left\langle  \int_{0}^{t}{\exp\left(\alpha_{0} s \right) \psi_{s} \dd s}, \mu^{M(0) + 1} \right\rangle \\
& \hspace{0.25cm} = \alpha_{0} \int_{0}^{t}{\exp\left(\alpha_{0} s \right)\int_{0}^{s}{\exp\left(-\alpha_{0} u \right)P_{\rm{FVc}}(u)P_{\xi_{0}, M(0)}\dd u}\dd s } + \int_{0}^{t}{P_{\rm{FVc}}(s)P_{\xi_{0}, M(0)}(\mu)\dd s} \\
& \hspace{0.25cm} = \exp\left(\alpha_{0}t \right) \int_{0}^{t}{\exp\left(-\alpha_{0} u \right)P_{\rm{FVc}}(u)P_{\xi_{0}, M(0)}(\mu)\dd u}.
 \end{aligned}
 \label{Eq_2_psi_t_Lem_6.6}
 \end{equation}
 
\noindent Then, the first announced result follows from (\ref{Eq_psi_t_Lem_6.6}) and (\ref{Eq_2_psi_t_Lem_6.6}) and the second one from (\ref{Eq_Borne_phi_t}). \qedhere
\end{proof}

\subsubsection{Proof of Lemma \ref{Lem_proba_invariante}\label{Sous_sous_section_6.3.2_Preuve_effective}}

The key point of this proof is to establish that for all $\xi_{0} \in \CCCC_{b}^{4}\left(\R^{M(0)}, \R\right)$,  $t \geqslant 0$ and $\mu \in \MM_{1}^{c, 4}(\R)$,  
\begin{equation}
\begin{aligned}
P_{\rm FVc}(t)P_{\xi_{0}, M(0)}(\mu) & = \left\langle \xi_{0}, \mu^{M(0)} \right\rangle + \int_{0}^{t}{\LL_{\rm FVc}^{\lambda}P_{\rm FVc}(s)P_{\xi_{0}, M(0)}(\mu)\dd s} \\
& = P_{\xi_{0}, M(0)}(\mu) + \int_{0}^{t}{\LL_{\rm FVc}^{\lambda}P_{V_{s}, M(0) + 1}(\mu)\dd s}.
\end{aligned}
\label{Eq_point_cle_lemme_6_3}
\end{equation}
where we use Proposition \ref{Prop_Semi_groupe_T_FVc} in the second equality. To do this, we will first  prove that for all $t \geqslant 0$, there exists a constant $C_{t} >0$ such that for all $x \in \R^{M(0) + 1}$, 
$\left\|\Hess\left(V_{t}\right)(x) \right\|_{\infty} \leqslant C_{t}\left(1 + \left\|x \right\|_{\infty}^{2} \right)$, so that $V_{t} \in \CCCC^{2}_{\|\cdot\|}\left(\R^{M(0)+1}, \R \right)$. \\

For all $n \in \N^{\star}$, for any function $f : \R^{n} \to \R$,  $k \in \N$ times differentiable, we denote by $D^{k}f$ the differential of order $k$ of $f$. Let %us consider 
$t \geqslant 0$, $\lambda >0$, $\mu \in \MM_{1}^{c, 2}(\R)$ and $\xi_{0} \in \CCCC^{4}_{b}\left(\R^{M(0)}, \R \right)$ be fixed. \\

\textit{Step 1. Preliminary bounds.} Let $x \in \R^{M(0)}$. It follows from (\ref{Eq_clef_Semi_groupe}) that 
\begin{align*}
& \partial_{x_{i}x_{j}x_{k}x_{\ell}}^{4}T^{\left(M(0) \right)}(t)\xi_{0}(x) \\
& \hspace{0.25cm} = \sum_{m, p, q, r \, = \, 1}^{M(0)}\partial_{x_{i}}\left(m_{t,x}^{\lambda} \right)_{m}\partial_{x_{j}}\left(m_{t,x}^{\lambda} \right)_{p}\partial_{x_{k}}\left(m_{t,x}^{\lambda} \right)_{q}\partial_{x_{\ell}}\left(m_{t,x}^{\lambda} \right)_{r}\left(g_{t, 0}^{\lambda} \ast \partial_{y_{m}y_{p}y_{q}y_{r}}^{4}{\xi_{0}}\right)\left(m_{t, x}^{\lambda} \right) 
\end{align*}
where $\left(m_{t,x}^{\lambda} \right)_{p}$ designates the $p^{\rm th}$ component of $m_{t, x}^{\lambda}$. Then, for all $k \in \left\{0, \cdots, 4 \right\}$, there exists a constant $C_{1}>0$ independent of $x$ and $t$ such that  \[\left\|D^{k}T^{\left(M(0) \right)}(t)\xi_{0}(x) \right\|_{\infty} \leqslant C_{1}\left(1 + \left\|D^{k}\xi_{0} \right\|_{\infty} \right). \]
From (\ref{Processus_dual_explicite}), (\ref{Expression_Phi_ij}) and (\ref{Operateur_K_ij}), we deduce that there exists a constant $C_{2} >0$ independent of $t$ such that for all $x \in \R^{M(0) + 1}$, for all $k \in \{0, 1, 2 \}$, 
\begin{align*}
\left\|D^{k}\xi_{t\wedge \tau_{1}}^{\left(M(0) + 1\right)}(x) \right\|_{\infty} & \leqslant C_{1}\left(1 + \left\|D^{k}\xi_{0} \right\|_{\infty} \right) + \sum_{i, j \, = \, 1}^{n}{\left\|D^{k}K_{i,j}T_{\lambda}^{\left(M(0) \right)}\left( \tau_{1}\right)\xi_{0}(x) \right\|_{\infty}}  \\ 
& \hspace{1cm} + \sum_{\substack{i, j \, = \, 1 \\ i \, \neq \, j}}^{n}{\left\|D^{k}\Phi_{i,j}T_{\lambda}^{\left(M(0) \right)}\left(\tau_{1} \right)\xi_{0}\left(x_{1}, \cdots, x_{M(0) - 1}\right) \right\|_{\infty}} \\
& \leqslant  C_{2}\left(1 + \left\|D^{k}\xi_{0} \right\|_{\infty}  + \left\|D^{k+2}\xi_{0} \right\|_{\infty}x_{M(0) + 1}^{2} \right).
\end{align*}
By the theorem of differentiation under the integral sign, for all $k \in \{0,1,2 \}$ and $x \in \R^{M(0)+1}$, we deduce that
\begin{align*}
\left\|D^{k}\psi_{t}(x) \right\|_{\infty} &  \leqslant \exp\left(\lambda M(0)^{2}t \right)\E_{\left(\mu, \xi_{0} \right)}\left(\left\|D^{k}\xi_{t\wedge \tau_{1}}^{\left(M(0) + 1\right)}(x) \right\|_{\infty}\right)  \\ 
& \leqslant C_{2}\exp\left(\lambda M(0)^{2}t \right)\left(1 + \left\|x \right\|_{\infty}^{2} \right).
\end{align*}
Then, from the definition of $V_{t}$ in (\ref{Eq_Def_V_t}) we deduce that there exists a constant $C_{t}$ such that for all $k \in \{0,1,2\}$, for all $x \in \R^{M(0)+1}$, $\left\|D^{k}V_{t}(x) \right\|_{\infty} \leqslant C_{t}\left(1 + \left\|x \right\|_{\infty}^{2} \right)$, so that $V_{t} \in \CCCC^{2}_{\|\cdot\|}\left(\R^{M(0)+1}, \R \right)$.  \\

\textit{Step 2. Proof of the key point (\ref{Eq_point_cle_lemme_6_3}).} Thanks to (\ref{PB_Mg_FVc_Polynomes}), for all $t \geqslant 0$,
\[P_{\rm FVc}(t)P_{\xi_{0}, M(0)}(\mu) = P_{\xi_{0}, M(0)}(\mu) + \E_{\mu}\left(\int_{0}^{t}{\LL_{\rm FV c}^{\lambda}P_{\xi_{0}, M(0)}\left(X_{s} \right)\dd s}  \right). \]
From (\ref{Def_Gene_FVc_polynome}), there exists a constant $C_{3}>0$ such that for all $t \geqslant 0$, $\E_{\mu}\left(\left|\LL_{\rm FV c}^{\lambda}P_{\xi_{0}, M(0)}\left(X_{t} \right) \right|\right) \leqslant C_{3}\left(1 + \E_{\mu}\left(M_{2}\left(X_{t}\right) \right)\right)$ which is finite from {\color{blue} \cite[\color{black} Proposition 2.11]{Champagnat_Hass_FVr_2022}} since $\mu \in \MM_{1}^{c, 4}(\R)$. Hence, from \textsc{Fubini}'s theorem we obtain that 
\begin{equation}
P_{\rm FVc}(t)P_{\xi_{0}, M(0)}(\mu) = P_{\xi_{0}, M(0)}(\mu) + \int_{0}^{t}{\E_{\mu}\left(\LL_{\rm FV c}^{\lambda}P_{\xi_{0}, M(0)}\left(X_{s} \right)\right)\dd s}.
\label{Eq_P_FVc_Fubini}
\end{equation}
As $\LL_{\rm FVc}^{\lambda}P_{\xi_{0}, M(0)}(\mu)$ is a polynomial in $\mu$ and since $$\left|\LL_{\rm FVc}^{\lambda}\left(\LL_{\rm FVc}^{\lambda}P_{\xi_{0}, M(0)}(\mu) \right) \right| \leqslant C_{4} \left(1 + M_{4}\left(\mu \right) \right)$$ for some constant $C_{4}>0$, we deduce as in (\ref{PB_Mg_FVc_Polynomes_J}) that
\begin{align*}
& \E_{\mu}\left(\LL_{\rm FVc}^{\lambda}P_{\xi_{0}, M(0)}\left(X_{t}\right) \right) - \E_{\mu}\left(\LL_{\rm FVc}^{\lambda}P_{\xi_{0}, M(0)}\left(X_{0}\right) \right) \\
& \hspace{5cm} = \E_{\mu}\left(\int_{0}^{t}{\LL_{\rm FV c}^{\lambda}\left(\LL_{\rm FV c}^{\lambda}P_{\xi_{0}, M(0)}\right)\left(X_{s} \right)\dd s} \right).
\end{align*}
In particular, $t\mapsto \E_{\mu}\left(\LL_{\rm FVc}P_{\xi_{0}, M(0)}\left(X_{t} \right) \right)$ is continuous, so
\begin{align*}
 \E_{\mu}\left(\LL_{\rm FVc}^{\lambda}P_{\xi_{0}, M(0)}\left(X_{t}\right) \right) & = \lim_{h \to 0}{\frac{1}{h}\E_{\mu}\left(\int_{t}^{t + h}{\LL_{\rm FV c}P_{\xi_{0}, M(0)}\left(X_{s} \right)\dd s} \right)} \\
&  = \lim_{h \to 0}\frac{1}{h}\E_{\mu}\left(P_{\xi_{0}, M(0)}\left(X_{t + h} \right) - P_{\xi_{0}, M(0)}\left(X_{t} \right) \right) \\
&  = \lim_{h \to 0}{\frac{1}{h}}\E_{\mu}\left(P_{\rm{FVc}}\left(t \right)P_{\xi_{0}, M(0)}\left(X_{h} \right) \right) - P_{\rm{FVc}}\left(t \right)P_{\xi_{0}, M(0)}(\mu) \\
& = \lim_{h \to 0}\frac{1}{h}\E_{\mu}\left(P_{V_{t}, M(0) + 1}\left(X_{h} \right) - P_{V_{t}, M(0) + 1}\left(\mu \right)  \right)
\end{align*}
where we used \textsc{Markov}'s property in the third equality and Proposition \ref{Prop_Semi_groupe_T_FVc} in the last equality. From Step 1, we deduce that there exists a constant $C_{5} > 0$ such that 
\begin{equation}
\left|\LL_{\rm FVc}^{\lambda}P_{V_{t}, M(0) + 1}(\mu) \right| \leqslant C_{5} \left(1 + M_{4}\left(\mu \right) \right).
\label{Eq_Majoration_L_FVc_V_t}
\end{equation}
Thus, from the martingale problem (\ref{PB_Mg_FVc_Polynomes_J}) and the continuity of $t\mapsto \E_{\mu}\left(\LL_{\rm FVc}P_{V_{t}, M(0)}\left(\mu \right) \right)$, we have that
\begin{align*}
\lim_{h \to 0}\frac{1}{h}\E_{\mu}\left(P_{V_{t}, M(0) + 1}\left(X_{h} \right) - P_{V_{t}, M(0) + 1}\left(\mu \right)  \right) & = \lim_{h\to 0}{\frac{1}{h} \E_{\mu}\left(\int_{0}^{h}{\LL_{\rm FVc}^{\lambda}P_{V_{t}, M(0)+1}\left(X_{s} \right)\dd s} \right)} \\
& = \LL_{\rm FVc}^{\lambda}{P_{V_{t}, M(0) + 1}\left(\mu \right)}.
\end{align*}
Therefore, we have proved that for all $t \geqslant 0$, for all $\mu \in \MM_{1}^{c, 4}(\R)$,  
\begin{align*}
 \E_{\mu}\left(\LL_{\rm FVc}^{\lambda}P_{\xi_{0}, M(0)}\left(X_{t}\right) \right)  = \LL_{\rm FVc}^{\lambda}P_{V_{t}, M(0) + 1}(\mu),
\end{align*}
and so, by (\ref{Eq_P_FVc_Fubini}), we obtain (\ref{Eq_point_cle_lemme_6_3}). \\ 

\textit{Step 3. Conclusion.} From (\ref{Eq_point_cle_lemme_6_3}) and (\ref{Eq_Majoration_L_FVc_V_t}) and since $\int_{\MM_{1}(\R)}^{}{M_{4}\left(\mu \right)\gamma(\dd \mu)} < \infty$, we deduce that 
\begin{align*}
&\int_{\MM_{1}(\R)}^{}{P_{\rm FVc}(t)P_{\xi_{0}, M(0)}(\mu) \gamma(\dd \mu)} \\
& \hspace{1.5cm} = \int_{\MM_{1}\left(\R \right)}^{}{P_{\xi_{0}, M(0)}(\mu) \gamma(\dd \mu)} + \int_{0}^{t}{\int_{\MM_{1}(\R)}^{}{\LL_{\rm FVc}^{\lambda}{P_{V_{s}, M(0) + 1}\left(\mu \right)}\gamma(\dd \mu)\dd s}} \\
& \hspace{1.5cm} = \int_{\MM_{1}\left(\R \right)}^{}{P_{\xi_{0}, M(0)}(\mu) \gamma(\dd \mu)},
\end{align*}
where we used the assumption of Lemma \ref{Lem_proba_invariante} in the last equality. \\

As the set of test functions $\left\{P_{f,n} \left|\phantom{1^{1^{1}}} \hspace{-0.5cm} \right.  f \in \CCCC_{b}^{4}(\R^{n}, \R), n\in \N^{\star} \right\}$ is $\MM_{1}\left(\MM_{1}(\R) \right)-$ convergence determining {\color{blue} \cite[\color{black} Lemma 2.1.2]{Dawson}}, it is $\MM_{1}\left(\MM_{1}(\R) \right)-$separating {\color{blue} \cite[\color{black} Chapter 3, Section 4, p.112]{Ethier_markov_1986}}, so we have for any bounded continuous function $\phi$ from $\MM_{1}(\R)$ to $\R$ that \[\int_{\MM_{1}(\R)}^{}{P_{\rm FVc}(t)\phi(\mu)\gamma(\dd \mu)} = \int_{\MM_{1}(\R)}^{}{\phi(\mu)\gamma(\dd \mu)}. \]  Hence $\gamma$  is an invariant probability measure
for the centered \textsc{Fleming-Viot} process with resampling rate $\lambda$. Now, from {\color{blue}\cite[\color{black} Theorem 4.1]{Champagnat_Hass_FVr_2022}}, $\pi^{\lambda}$ is its  unique invariant probability measure which ends the proof of Lemma~\ref{Lem_proba_invariante}.

\section{Characterisation of the limiting values of the slow component\label{Section_7_Caract_Composante_LENTE}}
Combining the results of Sections \ref{Section_5_Tension_LENT-RAPIDE} and \ref{Section_6_Caract_Gamma_Limite}, we have proved that for all $R >0$, the sequence of laws of $\left(\left(z^{K, R}, \Gamma^{K, R} \right)\right)_{K\in \N^{\star}}$ is tight in $\MM_{1}\left( \D\left([0, T], \T_{R} \right) \times \MM_{m}^{T}\left(\MM_{1}(\R) \right)\right)$ and for any limiting value $\Q$ of this sequence, the canonical process $\left(\zeta^{R}, \Gamma^{R} \right)$ on $\CCCC^{0}\left([0, T], \T_{R} \right) \times \MM_{m}^{T}\left(\MM_{1}(\R) \right)$ satisfies for all $f \in \CCCC^{1}_{b}\left(\T_{R}, \T_{R} \right)$, for all $t \in [0, T]$, $\Q-$a.s., 
\begin{equation}
f\left(\zeta_{t}^{R} \right) = f\left(\zeta_{0}^{R} \right) + \int_{0}^{t}{\int_{\MM_{1}(\R)}^{}{\LL_{\rm SLOW}f\left(\zeta_{s}^{R}, \mu \right)\pi^{\lambda\left(\zeta_{s}^{R} \right)}\left(\dd \mu \right)\dd s}}. \label{Eq_CEAD_f}
\end{equation}
From now on, we will establish in Section  \ref{Sous_section_7_1} with Lemma \ref{Lem_Unicite_ODE} that the sequence of laws of $\left(z^{K, R} \right)_{K \in \N^{\star}}$ converges weakly in $\D\left( [0, T], \T_{R}\right)$ to the solution of an ODE in the torus. Finally,  Section \ref{Sous_section_7_2} allows us to get away from the torus and to prove Theorem \ref{Thm_CEAD_Jouet}.

\subsection{Convergence of the slow component on the torus\label{Sous_section_7_1}}

\begin{Lem} The sequence $\left(z^{K, R} \right)_{K\in \N^{\star}}$ converges in law in $\D\left([0, T], \T_{R} \right)$ to the unique solution of
\begin{equation}
\forall t \in [0, T], \quad z_{t} = x_{0} + \frac{1}{2}\int_{0}^{t}{\partial_{1}\Fit\left(z_{s}, z_{s}\right)\beta\left(z_{s}\right)m_{2}\left(z_{s} \right)\dd s}
\label{Eq_CEAD_bis}
\end{equation}
on the torus $\T_{R}$.
\label{Lem_Unicite_ODE} 
\end{Lem}

\begin{proof}
For any limiting value $\Q$ of the sequence of laws of $\left(z^{K, R} \right)_{K\in \N^{\star}}$,
from (\ref{Generateur_LENT}) and applying (\ref{Eq_CEAD_f}) with $f = \id \in \CCCC_{b}^{2}\left(\T_{R}, \T_{R} \right)$, we deduce that $\Q-$a.s., \[\forall t \in [0, T], \quad \zeta_{t}^{R} = \zeta_{0}^{R} + \int_{0}^{t}{\left(\int_{\MM_{1}(\R)}^{}{M_{2}\left( \mu\right)\pi^{\lambda\left(\zeta_{s}^{R} \right)}(\dd \mu)}\right)\partial_{1}\Fit\left(\zeta_{s}^{R}, \zeta_{s}^{R} \right)\dd s}.\] From {\color{blue} \cite[\color{black} Corollary 4.16]{Champagnat_Hass_FVr_2022}}, for all $s \in [0, t]$, \[\int_{\MM_{1}\left(\R \right)}^{}{M_{2}\left(\mu \right)\pi^{\lambda\left(\zeta_{s}^{R} \right)}\left(\dd \mu\right)} = \frac{1}{2\lambda\left(\zeta_{s}^{R} \right)} =  \frac{\beta\left(\zeta_{s}^{R} \right)m_{2}\left(\zeta_{s}^{R} \right)}{2},\] so $\zeta^{R}$ is solution of (\ref{Eq_CEAD_bis}) $\Q-$a.s. By uniqueness of the solution of (\ref{Eq_CEAD_bis}), the announced result follows. \qedhere
\end{proof}

\subsection{End of the proof of Theorem \ref{Thm_CEAD_Jouet}\label{Sous_section_7_2}}

Let $T >0$ be fixed and recall that $x_{0}$ is the mean trait value of $\nu_{0}^{K}$. From Lemma \ref{Lem_Unicite_ODE} and since the ODE (\ref{Eq_CEAD_bis}) is non-explosive, we can choose $R >0$ large enough such that $$\lim_{K\to + \infty}{\P\left(z_{t}^{K, R} \in \left[x_{0} -\frac{R}{2}, x_{0} + \frac{R}{2} \right], \forall t \in [0, T] \right)} = 1.$$ From Proposition \ref{Prop_Convergence_tau_et_theta} and (\ref{Def_distribution_recentre_dilate}), we have that, 
\[\lim_{K\to + \infty}{\P\left(\Diam\left(\Supp \nu_{t}^{K, R}  \right) \leqslant \frac{1}{K^{1+\varepsilon/2}}, \forall t \in [0, T] \right)} = 1.\] Hence,
\[\lim_{K\to + \infty}{\P\left(\Supp \nu_{t/K{\color{black}\sigma_{K}^{2}}}^{K, R} \subset \left[x_{0} - R, x_{0} + R \right], \forall t \in [0, T] \right)} = 1. \]
Now, on the event $\left\{\Supp \nu_{t/K{\color{black}\sigma_{K}^{2}}}^{K, R} \subset \left[x_{0} - R , x_{0} + R \right], \forall t \in [0, T] \right\}$,
 $\nu_{t/K{\color{black}\sigma_{K}^{2}}}^{K, R} = \nu_{t/K{\color{black}\sigma_{K}^{2}}}^{K}$ for all $t\in [0, T]$, identifying $x \in \T_{R}$ with its unique representant in $\left[x_{0} - 2R, x_{0} + 2R \right]$. In particular, \[\forall t \in [0, T], \quad z_{t}^{K, R} = z_{t}^{K}.\]
Theorem \ref{Thm_CEAD_Jouet} then follows from Lemma \ref{Lem_Unicite_ODE}. \\

\appendix
\makeatletter
\def\@seccntformat#1{Annexe~\csname the#1\endcsname:\quad}
\@addtoreset{equation}{section}
  \renewcommand\theequation{\thesection.\arabic{equation}}
\makeatother

\textbf{Funding.}
This work was partially funded by the Chair ``Modélisation Mathématique et Biodiversité'' of VEOLIA-Ecole Polytechnique-MNHN-F.X. N.C. was partially funded by the European Union (ERC, SINGER, 101054787).
Views and opinions expressed are however those of the author(s) only and do not necessarily reflect those of the European Union or the European Research Council. Neither the European Union nor the granting authority can be held responsible for them. \\

{\color{black} \textbf{Acknowledgments.} We thank the two anonymous referees for their useful comments.}

\bibliographystyle{plain} 
\small \bibliography{These_biblio} \normalsize
\end{document}